\documentclass[a4paper,11pt,oneside]{amsart}
\pdfoutput=1

\usepackage[utf8]{inputenc}
\usepackage{bm}
\usepackage{mathtools,amssymb}
\usepackage{esint}
\usepackage{hyperref}
\hypersetup{
    colorlinks=true,
    linkcolor=black,
    filecolor=black,      
    urlcolor=cyan,
    citecolor=black
    }
\usepackage{tikz}
\usepackage{dsfont}
\usepackage{relsize}
\usepackage{url}
\usepackage{xcolor}
\usepackage{graphicx}
\usepackage{mathrsfs}
\usepackage[shortlabels]{enumitem}
\usepackage{lineno}
\usepackage{amsmath}
\usepackage{enumitem}
\usepackage{amsthm} 
\usepackage{fullpage} 
\usepackage{verbatim}
\usepackage{dsfont}

\allowdisplaybreaks

\mathtoolsset{showonlyrefs}

\graphicspath{{images/}}

\newtheorem{theorem}{Theorem}[section]
\newtheorem{lemma}[theorem]{Lemma}
\newtheorem{definition}[theorem]{Definition}
\newtheorem{corollary}[theorem]{Corollary}
\newtheorem{conjecture}[theorem]{Conjecture}

\newtheorem{proposition}[theorem]{Proposition}

\newtheorem{remark}[theorem]{Remark}

\title[Fractional Calderón problems on unbounded domains]{Fractional Calderón problems and Poincaré inequalities on unbounded domains}
\keywords{Fractional Laplacian, fractional gradient, Calderón problem, Poincaré inequality}
\subjclass[2020]{Primary 35R30; secondary 26A33, 42B37, 46F12}

\author{Jesse Railo}
\thanks{Department of Pure Mathematics and Mathematical Statistics, University of
Cambridge (\href{mailto:jr891@cam.ac.uk}{jr891@cam.ac.uk})}
\address{Department of Pure Mathematics and Mathematical Statistics, University of
Cambridge, Cambridge CB3 0WB, UK}
\email{jr891@cam.ac.uk}
\author{Philipp Zimmermann}
\thanks{Department of Mathematics, ETH Zurich (\href{mailto:philipp.zimmermann@math.ethz.ch}{philipp.zimmermann@math.ethz.ch})}
\address{Department of Mathematics, ETH Zurich, Z\"urich, Switzerland}
\email{philipp.zimmermann@math.ethz.ch}
\date{\today}

\newcommand{\C}{{\mathbb C}}
\newcommand{\R}{{\mathbb R}}
\newcommand{\Z}{{\mathbb Z}}
\newcommand{\N}{{\mathbb N}}

\newcommand{\schwartz}{\mathscr{S}}

\newcommand{\tempered}{\mathscr{S}^{\prime}}

\newcommand{\slowly}{\mathscr{O}_M}
\newcommand{\fraclaplace}{(-\Delta)^s}
\newcommand{\fourier}{\mathcal{F}}
\newcommand{\ifourier}{\mathcal{F}^{-1}}
\newcommand{\vev}[1]{\left\langle#1\right\rangle}

\newcommand{\cdistr}{\mathscr{E}'}
\newcommand{\distr}{\mathscr{D}^{\prime}}
\newcommand{\dimens}{n}
\newcommand{\norm}[1]{\lVert #1 \rVert}
\newcommand{\abs}[1]{\left\lvert #1 \right\rvert}
\newcommand{\aabs}[1]{\left\lVert #1 \right\rVert}
\newcommand{\ip}[2]{\left\langle #1,#2 \right\rangle}
\DeclareMathOperator{\Div}{div} 
\DeclareMathOperator{\supp}{supp} 

\begin{document}

\maketitle
\begin{abstract} We generalize many recent uniqueness results on the fractional Calderón problem to cover the cases of all domains with nonempty exterior. The highlight of our work is the characterization of uniqueness and nonuniqueness of partial data inverse problems for the fractional conductivity equation on domains that are bounded in one direction for conductivities supported in the whole Euclidean space and decaying to a constant background conductivity at infinity. We generalize the uniqueness proof for the fractional Calderón problem by Ghosh, Salo and Uhlmann to a general abstract setting in order to use the full strength of their argument. This allows us to observe that there are also uniqueness results for many inverse problems for higher order local perturbations of a lower order fractional Laplacian. We give concrete example models to illustrate these curious situations and prove Poincaré inequalities for the fractional Laplacians of any order on domains that are bounded in one direction. We establish Runge approximation results in these general settings, improve regularity assumptions also in the cases of bounded sets and prove general exterior determination results. Counterexamples to uniqueness in the inverse fractional conductivity problem with partial data are constructed in another companion work.
\end{abstract}

\setcounter{tocdepth}{1}
\tableofcontents

\section{Introduction}

The primary objective of this article is to study the fractional Calderón problem for perturbations of the fractional Laplacian $(-\Delta)^{s}u \vcentcolon = \ifourier(\abs{\xi}^{2s} \hat{u})$ where $s \in \R_+ \setminus \Z$. This is the nonlocal analog of the classical Calderón problem for the Schrödinger equation which appears in many imaging applications \cite{UHL-electrical-impedance-tomography, UH-inverse-problems-seeing-the-unseen}. The study of this type of fractional inverse problems was initiated by Ghosh, Salo and Uhlmann in \cite{GSU20}. In particular, we consider generalizations of the model problem \begin{align}\label{eq:fractionalschrodingerequation} \left\{\begin{array}{rl}
        (\fraclaplace+q)u&=0  \;\;\text{in}\;  \Omega,\\
        u &=f   \;\;\text{in}\;\Omega_e ,
        \end{array}\right.\end{align}
where $\Omega_e=\R^n \setminus\overline{\Omega}$ is the exterior of an open set $\Omega$ and is assumed to be nonempty. The associated Dirichlet to Neumann (DN) map to the exterior value problem \eqref{eq:fractionalschrodingerequation}, under certain conditions on $q$ and $\Omega$, is a bounded linear operator $\Lambda_q\colon H^s(\Omega_e)\rightarrow (H^s(\Omega_e))^*$ which acts under some stronger regularity assumptions than imposed in this article as $\Lambda_q f=\fraclaplace u|_{\Omega_e}$ (see \cite[Lemma A.1]{GSU20}). We need to impose some positivity or smallness assumptions on the perturbations $q$ to guarantee that the exterior DN map is well-defined due to the lack of compact Sobolev embeddings on unbounded domains. The studied fractional Calderón problem can be simply formulated as follows: \emph{Given $f \mapsto \Lambda_q f$, determine $q$.}

We extend the recent uniqueness results in \cite{CMR20,RS-fractional-calderon-low-regularity-stability} for singular potentials and for lower order local linear perturbations \cite{CLR18-frac-cald-drift,CMRU20-higher-order-fracCald} to domains that are bounded in one direction. In fact, we prove a general result which applies also in the case of local higher order perturbations (i.e. the order of the perturbation is higher than the order of the fractional Laplacian) and for all domains with nonempty exterior. We also consider the setting of the fractional conductivity equation and the related Liouville transformation studied earlier in \cite{covi2019inverse-frac-cond}. The main contributions of our work are Theorem \ref{thm: characterization of uniqueness}, which rather completely characterizes the uniqueness of the fractional Calderón problem for the conductivity equation in a general setting, and an extension of the usual methods on bounded domains to the case of nonlocal (or fractional) inverse problems of any order on unbounded domains (see Theorems \ref{thm:PoincUnboundedDoms} and \ref{theorem:GeneralFractionalCalderon}).

The methods needed on unbounded domains allow us to slightly extend the uniqueness results of the fractional Calderón problem in the usual setting of bounded domains to larger classes of singular perturbations than in the scientific literature (see e.g.~\cite{CMR20,CMRU20-higher-order-fracCald,RS-fractional-calderon-low-regularity-stability}). We refer to the survey~\cite{Sal17} for a more detailed but a bit outdated treatment of the topic. See also \cite{CLL19,CLR18-frac-cald-drift,CO-magnetic-fractional-schrodinger,covi2021uniquenessQuasiLocal, GLX-calderon-nonlocal-elliptic-operators,GRSU-fractional-calderon-single-measurement,LL19, LL-fractional-semilinear-problems, LLR19,LI-fractional-magnetic, LILI-semilinear-magnetic,  LILI-fractional-magnetic-calderon,ruland2019quantitative} and the references therein for other recent studies of the fractional Calderón problems. These include for example perturbations of nonlocal elliptic anisotropic operators, nonlinear perturbations and nonlocal perturbations among many other model problems as well as studies of the stability and reconstruction. Most recently, inverse problems for the fractional spectral Laplacian have been studied also on Riemannian manifolds \cite{feizmohammadi2021fractional,feizmohammadiEtAl2021fractional}. We discuss some of these works with more details in the later sections.

The classical Calderón problem and related inverse problems have been studied on unbounded domains earlier. A historical reference is the work of Langer, which studies the inverse conductivity problem and geophysical applications for conductivities in a half-space under the assumption that the conductivity is depending only upon the depth \cite{Langer33-calderon-halfspace}. The classical Calderón problem is studied in a half-space, e.g., for convex inclusions in \cite{FI89-calderonProblem-one-measurement} and in two dimensions for anisotropic conductivities in \cite{APL05-Calderon-anistropic-plane}. These inverse problems are also studied in infinite slabs or similar unbounded structures in \cite{CKS17-Calderon-problem-periodic-cylindrical-domain,Ikehata01-calderon-infinite-slab,KLU12-magnetic-calderon-infinite-slab,SW06-complex-spherical-waves-unbounded-domains} when $n \geq 3$. Inverse problems for general second order elliptic equations are studied recently in \cite{BR20-inverseproblems-elliptic-unbounded-domains}. The above list is incomplete and more references can be found from the aforementioned works. In the classical Calderón problems on unbounded domains, one often demands some radiation condition for the solutions at infinity which are not present in the nonlocal formulations studied in this article.

The second goal of this article is to study Poincaré inequalities for the fractional Laplacian of any positive order $s \geq 0$ in the local Bessel potential spaces. In the simplest form, we prove that for any bounded open set $\Omega \subset \R^n$, $n \geq 1$, $1 < p < \infty$ and $s \geq 0$ there exists a constant $C(n,p,s,\Omega) > 0$ such that
\begin{equation}\label{eq:basicPoincareIntro}
    \norm{u}_{L^p(\Omega)} \leq C\norm{(-\Delta)^{s/2}u}_{L^p(\R^n)}
\end{equation}
for all $u \in C_c^\infty(\Omega)$. We obtain generalizations of this inequality to domains that are bounded in one direction. We do not impose any boundary regularity assumptions on $\Omega$. We refer to the following surveys and expository articles on the fractional Laplacian, fractional Sobolev spaces and their applications \cite{DINEPV-hitchhiker-sobolev,KWA-ten-definitions-fractional-laplacian,KM19-FracLapOp-its-properties,LPGetAl-WhatIsFracLap,LSW-Fractional-Gaussian-fields-survey}. See also the books \cite{Interpolation-spaces,Functional-spaces-elliptic-equations,Leoni17-FirstCourseInSobolevSpaces,Singular-Integrals-Stein,Triebel-Theory-of-function-spaces,TRI-interpolation-function-spaces} covering nearly the whole spectra of classical methods related to the fractional order Sobolev spaces.

The Poincaré inequalities for the fractional Laplacian are important as they allow to show existence, uniqueness and boundedness theorems for weak solutions of PDEs (see e.g. \cite{CMRU20-higher-order-fracCald,GSU20,Boundary-Regularity-fract-laplacian,rosoton2013extremal,RS-fractional-calderon-low-regularity-stability}). The standard example is the fractional Laplace equation $(-\Delta)^{s/2}u =F$ with the interior data $F$ in some domain $\Omega \subset \R^n$ (see e.g.~ \cite{RosOton16-NonlocEllipticSurvey}). The cases when $p \neq 2$ appear naturally in the studies of nonlinear PDEs and the standard examples are different kind of $p$-Laplacians (see e.g.~\cite{BLP13-Cheeger,BS19-Note-homog-Sob-space-frac-ord,NDS19-WeightedIneqFracLap,PBL18-Opt-Prob-First-Eigenvalues-pFracLap,DSV15-Dens-Frac-Sobspaces,RR08-nonLingroundHardyInq,FP14-Frac-p-eigenvalues,LL14--Fractional-eigenvalues,MS21-Best-frac-p-Poincare-unboundedDoms,MPSY16-Brez-Niren-prob16} and references therein). The Poincaré and other closely related inequalities for the fractional Laplacians are studied recently, for example, in \cite{AFM19-Asymptotic-FracLap-unbounded,BC18-fractionalHardy,IGPF21-fracPoincUnbounded,PBL18-Opt-Prob-First-Eigenvalues-pFracLap,HYZ12-FracGaglNirenHardy,LL14--Fractional-eigenvalues,MS21-Best-frac-p-Poincare-unboundedDoms,yeressian2014asymptotic}. On recent studies and further references of the higher order fractional Laplacians, we point to \cite{CMR20,CMRU20-higher-order-fracCald,RS15-higher-order-frac}. The higher order fractional Laplacian also appears in the studies of the Radon transform and imaging methods since the normal operator of the $d$-plane Radon transform for an integer $1 \leq d < n$ can be identified with $(-\Delta)^{-d/2}$ and the inverse operator can be defined using $(-\Delta)^{d/2}$ (see e.g.~\cite{CMR20,HE:integral-geometry-radon-transforms,IM-unique-continuation-riesz-potential}).

We state our main results on the fractional Poincaré inequalities in Section \ref{sec:FracPoincResults}. We discuss our main results on the fractional Calderón problems in Section \ref{sec:GenFracCaldResults}. We recall preliminaries in Section \ref{sec:preliminaries}. We prove the results related to the generalized fractional Calderón problem in Section \ref{sec:GeneralizedFractionalCalderon}. We prove the results related to the fractional Poincaré inequalities in Section \ref{sec:poincare}. We define some classes and decompositions of Sobolev multipliers and PDOs in Section \ref{sec:SobMultDecompositions}, and these are used to obtain uniqueness results for concrete example models in Section \ref{sec:fractionalCalderonMainSec}. We study the fractional conductivity equation and the related inverse problem in Section \ref{sec:conductivityCalderon}. We prove some basic properties for products of functions in Bessel potential spaces in Appendix \ref{subsubsec: Multiplication map in Bessel potential spaces} as they are essential for our main results on the fractional conductivity equation and we could not locate all of them in the required form in the earlier literature.

\subsection*{Acknowledgements} The authors thank Prof. Mikko Salo for commenting an earlier version of the manuscript and suggesting changes. The authors thank Dr. Giovanni Covi for many fruitful discussions related to the topics of this article. J.R. was supported by the Vilho, Yrjö and Kalle Väisälä Foundation of the Finnish Academy of Science and Letters.

\section{Main results of the article}

\subsection{Poincaré inequalities for the fractional Laplacians}\label{sec:FracPoincResults}

Our first main result is the higher order fractional Poincaré inequality for domains that are bounded in one direction. We begin with the following definition.
\begin{definition} \label{def:bounded1dir}
\begin{enumerate}[(i)]
    \item\label{item 1 bounded one dir} We say that an open set $\Omega_\infty \subset\R^n$ of the form $\Omega_\infty=\R^{n-k}\times \omega$, where $n\geq k>0$ and $\omega \subset \R^k$ is a bounded open set, is a \emph{cylindrical domain}.
    \item\label{item 2 bounded one dir} We say that an open set $\Omega \subset \R^n$ is \emph{bounded in one direction} if there exists a cylindrical domain $\Omega_\infty \subset \R^n$ and a rigid Euclidean motion $A(x) = Lx + x_0$, where $L$ is a linear isometry and $x_0 \in \R^n$, such that
$\Omega \subset A\Omega_\infty$.
\end{enumerate}

\end{definition}
One could equivalently require that there exists a rigid Euclidean motion $A$ and a cylindrical domain $\Omega_{\infty}\subset\R^n$ such that $A\Omega \subset \Omega_\infty$. In Section \ref{sec:poincare} we prove the following theorem.

\begin{theorem}[Poincaré inequality]\label{thm:PoincUnboundedDoms} Let $\Omega\subset\R^n$ be an open set that is bounded in one direction. Suppose that $2 \leq p < \infty$ and $0\leq s\leq t < \infty$, or $1 < p < 2$, $1 \leq t < \infty$ and $0 \leq s \leq t$. Then there exists $C(n,p,s,t,\Omega)>0$ such that
    \begin{equation}
    \label{eq: poincare on L1}
        \|(-\Delta)^{s/2}u\|_{L^p(\R^n)}\leq C\|(-\Delta)^{t/2}u\|_{L^p(\R^n)}
    \end{equation}
    for all $u\in \Tilde{H}^{t,p}(\Omega)$.
\end{theorem}

We obtain this result and also information on the Poincaré constants using a generalization of the interpolation result of fractional Poincaré constants in \cite{CMR20}. The lower order cases $0 < t < 1$ when $p > 2$ rely on the Poincaré inequalities for the Gagliardo seminorms (see e.g. \cite[Theorem 1.2]{MS21-Best-frac-p-Poincare-unboundedDoms} and references therein). Unfortunately, we do not have a proof which would also settle the cases $0 < t < 1$ and $1 < p < 2$. We have stated and proved the fractional Poincaré inequality on bounded sets in Lemma \ref{lemma:PoincareBoundedsets} without these limitations on the integrability and Sobolev scales. See also the discussion of different methods of proof preceding the statement of Lemma \ref{lemma:PoincareBoundedsets}.

We have the following result on the Poincaré constants. The proof is given in Section \ref{sec:poincare}.

\begin{theorem}[Interpolation of Poincaré constants]
\label{theorem: Higher Order Fractional Poincare Inequality} Let $\Omega \subset \R^n$ be an open set and $1 < p < \infty$. Suppose that $r>z$, and $t\geq s\geq r >z\geq 0$ or $t\geq r\geq s\geq z\geq 0$ hold. If there exists a fractional Poincar\'e constant $C_{r,z}>0$ with
    \begin{equation}
    \label{eq: assumption higher order fractional poincare constant I1}
        \|(-\Delta)^{z/2}v\|_{L^p(\R^n)}\leq C_{r,z}\|(-\Delta)^{r/2}v\|_{L^p(\R^n)}
    \end{equation}
    for all $v\in \widetilde{H}^{r,p}(\Omega)$, then we have
    \begin{equation}
    \label{eq: conclusion higher order fractional poincare constant I2}
        \|(-\Delta)^{s/2}u\|_{L^p(\R^n)}\leq C_{r,z}^{\frac{t-s}{r-z}}\|(-\Delta)^{t/2}u\|_{L^p(\R^n)}
    \end{equation}
    for all $u\in \widetilde{H}^{t,p}(\Omega)$.
\end{theorem}

We believe that the constant in Theorem \ref{theorem: Higher Order Fractional Poincare Inequality} remains optimal if $C_{r,z}$ is the optimal Poincaré constant. The resulting constant in \eqref{eq: conclusion higher order fractional poincare constant I2} has the same scaling properties as the optimal constants and the constant cannot be improved by making the interpolation argument twice. Moreover, the Poincaré constants for $\tilde{H}^{t+2,p}(\Omega)$ give upper bounds for lower order Poincaré constants for all $\overline{(-\Delta)\tilde{H}^{t+2,p}(\Omega)} \subset \tilde{H}^{t,p}(\Omega)$, and if ''$\overline{(-\Delta)\tilde{H}^{t+2,p}(\Omega)} = \tilde{H}^{t,p}(\Omega)$'' \emph{were} true for all $t >0$, then the optimality and the lower order cases when $1 < p <2$ would also follow by a short proof using the interpolation result. For these heuristical reasons, we have formulated the optimality of the constants as Conjecture \ref{conj:equivalenceOfPoincareConstants}. See also Remark \ref{remark:generalizationsInterpolation} on a possibility to generalize the assumptions in Theorem \ref{theorem: Higher Order Fractional Poincare Inequality}.

\subsection{Generalized fractional Calderón problems}\label{sec:GenFracCaldResults}
We next describe our results for the generalized fractional Calderón problems. We first state an abstract formulation of the argument given in \cite{GSU20} which was later modified for lower order PDO perturbations in \cite{CLR18-frac-cald-drift,CMRU20-higher-order-fracCald}. A related question is discussed in \cite[Remark, p.~7]{RS-fractional-calderon-low-regularity-stability} by Rüland and Salo but not studied there further. The approach outlined in \cite{RS-fractional-calderon-low-regularity-stability} is more related to the work accomplished in \cite{covi2021uniquenessQuasiLocal} by Covi for quasilocal perturbations of the fractional Laplacian. We then discuss concrete model problems with our focus in the inverse problem for the fractional conductivity equation studied first in \cite{covi2019inverse-frac-cond}.

We first introduce the related definitions for the abstract formulation.

\begin{definition}[Properties of bilinear forms] Let $s \in \R$ and $B\colon H^s(\R^n) \times H^s(\R^n) \to \R$ be a bounded bilinear form.
\begin{enumerate}[(i)]
\item\label{item left UCP} We say that $B$ has the \emph{left UCP} on an open nonempty set $W \subset \R^n$ when the following holds: If $u \in H^s(\R^n)$, $u|_W = 0$ and $B(u,\phi) = 0$ for all $\phi \in C_c^\infty(W)$, then $u \equiv 0$.
\item\label{item right UCP} We say that $B$ has the \emph{right UCP} on an open nonempty set $W \subset \R^n$ when the following holds: If $u \in H^s(\R^n)$, $u|_W = 0$ and $B(\phi,u) = 0$ for all $\phi \in C_c^\infty(W)$, then $u \equiv 0$.
\item\label{item strong UCP} We say that $B$ has the \emph{symmetric UCP} on an open nonempty set $W \subset \R^n$ when it has the left UCP and the right UCP on $W$.
\item\label{item local} We say that $B$ is \emph{local} when the following holds: If $u,v \in H^s(\R^n)$ and $\supp(u) \cap \supp(v) = \emptyset$, then $B(u,v) = 0$.
\item\label{item adjoint} We denote by $B^*(u,v) \vcentcolon = B(v,u)$ the \emph{adjoint (or transpose) bilinear form} of $B$.
\end{enumerate}
\end{definition}

\begin{remark} If $B$ is local, then $B^*$ is local. If $B$ has the left UCP, then $B^*$ has the right UCP and vice versa. If $B$ is symmetric, that is $B = B^*$, then $B$ has the symmetric UCP if and only if it has the left or the right UCP.
\end{remark}

Our axiomatic formulation for the uniqueness of the generalized fractional Calderón problem comes with two important improvements in comparison to the existing theorems in this area. Firstly, we replace the assumption that $\Omega$ is a bounded set by the assumption that $\Omega$ has nonempty exterior. Secondly, we allow general bilinear forms in the statement. This let us to study problems where the local perturbation $q$ has higher order than the nonlocal bilinear form $L$ with suitable unique continuation properties. The third benefit of this general statement is that it explicitly states what kind of unique continuation properties are really needed for the nonlocal term $L$. Especially, we do not have to assume that bilinear forms are symmetric. A closely related approach was taken earlier in the work \cite{covi2021uniquenessQuasiLocal} in the case of quasilocal perturbations but these and our results are not included into each other. We also prove an exterior determination result which does not appear in the earlier literature even in the simplest case of the fractional Schrödinger equation, to the best of our knowledge. The proofs together with the necessary definitions are given in Section \ref{sec:GeneralizedFractionalCalderon}.

\begin{theorem}[Interior determination]\label{theorem:GeneralFractionalCalderon} Let $s \in \R$, and $\Omega \subset \R^n$ be open such that $\Omega_e \neq \emptyset$. Let $L\colon H^s(\R^n) \times H^s(\R^n) \to \R$ be a bounded bilinear form with the following properties:
\begin{enumerate}[(i)]
    \item\label{item 1 interior det} There exists a nonempty open set $W_1 \subset \Omega_e$ such that $L$ has the right UCP on $W_1$.
    \item\label{item 2 interior det} There exists a nonempty open set $W_2 \subset \Omega_e$ such that $L$ has the left UCP on $W_2$.
    \item\label{item 3 interior det} $W_1 \cap W_2 = \emptyset$.
\end{enumerate}Let $q_j\colon H^s(\R^n) \times H^s(\R^n) \to \R$, $j=1,2$, be local and bounded bilinear forms. Suppose that $B_{L,q_j} = L + q_j$ are (strongly) coercive in $\tilde{H}^s(\Omega)$. 

If the exterior data $\Lambda_{L,q_1}[f][g] = \Lambda_{L,q_2}[f][g]$ agree for all $f \in C_c^\infty(W_1)$ and $g \in C_c^\infty(W_2)$, then $q_1 = q_2$ in $\tilde{H}^s(\Omega) \times \tilde{H}^s(\Omega)$. 
\end{theorem}

We identify functions $q \in L^1_{loc}(\R^n)$ with the bilinear form $q(u,v) = \int_{\R^n} q uv dx$ for all $u,v \in C_c^\infty(\R^n)$ and, in particular, one of the assumptions in Corollary \ref{cor:fracCaldExteriorDet} is that the related bilinear form is bounded in $H^s(\R^n)$. For example, $q\in L^{\frac{n}{2s}}(\R^n)$ when $0<s<n/2$ (see Lemma~\ref{lemma: continuity of multiplication map in Bessel potential spaces I V}), or $q\in L^{\frac{2p}{2-p}}(\R^n)$ when $s=n/2$ and any $1<p<2$ (see Proposition~\ref{proposition: multiplication map in fractional Sobolev spaces}, \cite[p. 261]{Ozawa}), or $q\in L^1(\R^n)$ when $s>n/2$. The locality of the bilinear forms of this type follows automatically. The idea behind Corollary \ref{cor:fracCaldExteriorDet} turns out to give strong partial data uniqueness results for the inverse fractional conductivity problem studied in details in Section \ref{sec:conductivityCalderon}.

\begin{corollary}[Exterior determination]\label{cor:fracCaldExteriorDet} Suppose that the assumptions of Theorem \ref{theorem:GeneralFractionalCalderon} hold. If additionally $\Lambda_{L,q_1}[f][g] = \Lambda_{L,q_2}[f][g]$ for all $f,g \in C_c^\infty(W)$ where $W \subset \Omega_e$ is open, then $q_1=q_2$ in $\tilde{H}^s(W) \times \tilde{H}^s(W)$.

In particular, we have the following full data uniqueness result: If $W = \Omega_e$ and $q_1,q_2 \in C(\R^n)$ (or $q_1,q_2 \in L^1_{loc}(\R^n)$ and $\partial \Omega$ has measure zero), then $q_1=q_2$ (a.e.) in $\R^n$.
\end{corollary}

Theorem \ref{theorem:GeneralFractionalCalderon} can be seen as a general uniqueness lemma for nonlocal exterior value inverse problems. It might be possible to obtain more abstract versions of Theorem \ref{theorem:GeneralFractionalCalderon} in general Hilbert or Banach spaces (especially on geometric settings on manifolds) but that is out of our scope here. One can remove the assumption that $B$ is strongly coercive by assuming that the related interior and exterior value problems have unique weak solutions for $B$ and its adjoint $B^*$ (cf. Lemma \ref{lemma:generalWellposedness} and Remark \ref{remark:generalizationCalderon}). See also Theorem \ref{lemma:GeneralRungeApproximation} for a general Runge approximation result which may be of interest also in other applications. The main point of the article is to explain what Theorem \ref{theorem:GeneralFractionalCalderon} together with the Poincaré inequalities (Theorem \ref{thm:PoincUnboundedDoms}) and some additional analysis imply for many new model problems to be discussed next. We also give minor improvements for the regularity assumptions of many existing theorems.
\subsubsection{Fractional conductivity equation}
We highlight our result for the fractional conductivity equation which rather completely characterizes the uniqueness and nonuniqueness of the related partial data inverse problem on domains that are bounded in one direction. Theorem \ref{thm: characterization of uniqueness} would also generalize to all domains with finite fractional Poincaré constant. This problem was studied earlier in \cite{covi2019inverse-frac-cond} by Covi for bounded domains under the assumption that conductivities are trivial in the exterior. Here we relax the assumption that the conductivities take constant values (and are thus equal) in the exterior by making suitable decay assumptions at infinity. We also define here the Liouville reduction slightly differently than in the article \cite{covi2019inverse-frac-cond} in order to have access to exterior determination and obtain more general results. As a minor improvement, we remove the superfluous Lipschitz assumption of the domain.

The fractional gradient of order $s$ is the bounded linear operator $\nabla^s\colon H^s(\R^n)\to L^2(\R^{2n};\R^n)$ given by
    \[
        \nabla^su(x,y)=\sqrt{\frac{C_{n,s}}{2}}\frac{u(x)-u(y)}{|x-y|^{n/2+s+1}}(x-y),
\]
where $C_{n,s}>0$ is a constant. This operator appears in the studies of nonlocal diffusion (see e.g. \cite{NonlocDiffusion} and references therein) and is naturally associated with the $L^2$ Gagliardo seminorm. We denote the adjoint of $\nabla^s$ by $\Div_s$. Let $\gamma \in L^\infty(\R^n)$ be a positive conductivity. We denote the conductivity matrix by $\Theta_{\gamma}(x,y)\vcentcolon =\gamma^{1/2}(x)\gamma^{1/2}(y)\mathbf{1}_{n\times n}$ for $x,y\in\R^n$ and by $m_{\gamma}\vcentcolon =\gamma^{1/2}-1$ the background deviation. We use the notation $\Lambda_\gamma$ for the exterior DN map associated to the equation
\begin{equation}
        \begin{split}
            \Div_s(\Theta\cdot\nabla^s u)&= 0\quad\text{in}\quad\Omega,\\
            u&= f\quad\text{in}\quad\Omega_e.
        \end{split}
\end{equation}
See the beginning of Section \ref{sec:conductivityCalderon} for further details of the basic definitions. 

We note that Lemmas \ref{thm: uniqueness conductivity equation} and \ref{thm: sharpness of condition (ii) in uniqueness thm} state some interesting special cases of Theorem \ref{thm: characterization of uniqueness} (see also Remark \ref{rmk: fractional conductivity UCP remark}). Our main theorem is the following. The proof is given in the end of Section \ref{sec:conductivityCalderon} and it utilizes nearly all of our other results that are proved in the earlier sections.

\begin{theorem}
\label{thm: characterization of uniqueness}
    Let $\Omega\subset \R^n$ be an open set which is bounded in one direction and $0<s<\min(1,n/2)$. Assume that $\gamma_1,\gamma_2\in L^{\infty}(\R^n)$ with background deviations $m_1,m_2$ satisfy $\gamma_1(x),\gamma_2(x)\geq \gamma_0>0$ and $m_1,m_2\in H^{2s,\frac{n}{2s}}(\R^n)$. Moreover, assume that $m_0\vcentcolon =m_1-m_2\in H^s(\R^n)$ and $W_1,W_2\subset\Omega_e$ are nonempty open sets with
    \begin{equation}\label{eq:conductivitysupportassumptions}
         (\supp(m_1)\cup\supp(m_2))\cap (W_1\cup W_2)=\emptyset.
    \end{equation}
Then the following statements hold:
\begin{enumerate}[(i)]
    \item\label{item 1 characterization of uniqueness} If $W_1\cap W_2\neq \emptyset$, then  $\left.\Lambda_{\gamma_1}f\right|_{W_2}=\left.\Lambda_{\gamma_2}f\right|_{W_2}$ for all $f\in C_c^{\infty}(W_1)$ if and only if $\gamma_1=\gamma_2$ in $\R^n$.
    \item\label{item 2 characterization of uniqueness} If $W_1\cap W_2=\emptyset$, then $\left.\Lambda_{\gamma_1}f\right|_{W_2}=\left.\Lambda_{\gamma_2}f\right|_{W_2}$ for all $f\in C_c^{\infty}(W_1)$ if and only if $m_1-m_2$ is the unique solution of
    \begin{equation}
    \label{eq: PDE uniqueness cond eq}
        \begin{split}
            (-\Delta)^sm-\frac{(-\Delta)^sm_1}{\gamma_1^{1/2}}m&=0\quad\text{in}\quad \Omega,\\
            m&=m_0\quad\text{in}\quad \Omega_e.
        \end{split}
    \end{equation}
\end{enumerate}
\end{theorem}

\begin{remark}
There were three points which seem to be overlooked in \cite{covi2019inverse-frac-cond}. However, all arguments in \cite{covi2019inverse-frac-cond} are easily justified for $C_c^\infty(\Omega)$ perturbations of the uniform background conductivity. To cover a few missing steps in the low regularity setting, we establish regularity results related to the fact that the (fractional) Liouville transformation is bijective (cf.~ Theorem \ref{theorem: Liouville transformation}). In the proof, we need to approximate equations with equations having regular coefficients so that the used pointwise computation with a singular integral definition of the fractional Laplacian can be applied to the \emph{coefficients} coming from the conductivity equation (this particular approximation uses results in \cite{SilvestreFracObstaclePhd}). The second point is that we show that the Liouville transformation can be reversed from the solutions of the associated fractional Schrödinger equation to the solutions of a conductivity equation, so that the first mentioned problem has unique solutions, by the same property of conductivity equations, and has \emph{always} a well-defined exterior DN map (in particular $0$ is never a Dirichlet eigenvalue for the related exterior value problem). This uses an argument showing sufficient regularity of the division by $\gamma^{1/2}$ (employing \cite{AdamsComposition}). It is rather clear that this property is desired so that the result applies for all conductivities from the given class without any open questions related to the well-posedness of the reduction to the fractional Schrödinger equation. Finally, we consider multiplication results for Bessel potential spaces in Appendix \ref{subsubsec: Multiplication map in Bessel potential spaces} by adapting the general methods in \cite{BrezisComposition}. This is done as the reference \cite{BrezisComposition} used in \cite{covi2019inverse-frac-cond} establishes the multiplication results with respect to the fractional Gagliardo norms but not directly for the Bessel potential norms as actually desired in \cite{covi2019inverse-frac-cond} and our work. Nevertheless, the main results in \cite{covi2019inverse-frac-cond} are correct whenever $0 < s < 1$ and $n \geq 2$, or $0 < s < 1/2$ when $n=1$, as the aforesaid gaps in the proof are made rigour in our article, but the cases $1/2 \leq s < 1$ when $n=1$ are actually still open in low regularity to the best of our understanding. Otherwise, the strategy of our interior determination proof follows that of \cite{covi2019inverse-frac-cond} whereas the exterior determination requires yet another argument, given in Section \ref{sec:uniquenessFracCaldSubsection}.
\end{remark}

This result is interesting for several reasons to be discussed next. Since the problem studied here is considerably more general than the one in \cite{covi2019inverse-frac-cond}, nonuniqueness may occur in some cases when $W_1 \cap W_2 = \emptyset$. A striking difference to the earlier results for the fractional Calderón problems is that when $W_1 \cap W_2 \neq \emptyset$, then this actually helps to solve the inverse problem. In this case, one obtains a global uniqueness result also in $\Omega_e$ without assuming that conductivities agree there in the first place. The Liouville reduction is analogous to the one for the classical Calderón problem \cite{SU87-CalderonProblem-annals} but we do not have to use any kind of special solutions in the end. In the proof of Theorem \ref{thm: characterization of uniqueness}, we do not reduce the inverse problem only back to the Calderón problem for the fractional Schrödinger equation like in \cite{covi2019inverse-frac-cond} but also invoke a new characterization of equal exterior data and the UCP of fractional Laplacians. This part resembles the earlier studies which avoid using the Runge approximation and rather directly rely on the UCP to show uniqueness in the fractional Calderón problem (see e.g. \cite{GRSU-fractional-calderon-single-measurement} for single measurement results).

The exterior determination result is not possible with only finitely many measurement whereas interior determination is possible with just single measurement whenever the conductivity is known in the exterior. Exterior determination can be thought as an anolog of boundary determination in the classical Calderón problem \cite{KV84-CalderonBdryDetermination}, which is an important step in the proof of uniqueness \cite{SU87-CalderonProblem-annals}. However, we do not have a complete exterior determination result and we still need to rely to the fact that conductivities are assumed to be constant in the sets where we make our exterior measurements (this is the content of the formula \eqref{eq:conductivitysupportassumptions}). A more complete answer to the exterior determination problem remains open. We remark that when the exterior conditions (fractional voltage) are set disjointly with respect to the set where the measurements (fractional current) are made, we may have the loss of uniqueness (see the article~\cite{RZ2022counterexamples} for constructions of counterexamples). In the classical Calderón problem there are also uniqueness results for some partial data problems which are analogous with Theorem \ref{thm: characterization of uniqueness} \ref{item 1 characterization of uniqueness} (see \cite[Corollary 1.4]{KSU07-CalderonPartialData}).

Below in Figures~\ref{fig: Geometric setting 2} and \ref{fig: Geometric setting 1} we graphically illustrate the two possible geometric settings in Theorem~\ref{thm: characterization of uniqueness} corresponding to the cases \ref{item 1 characterization of uniqueness} and \ref{item 2 characterization of uniqueness}, respectively.

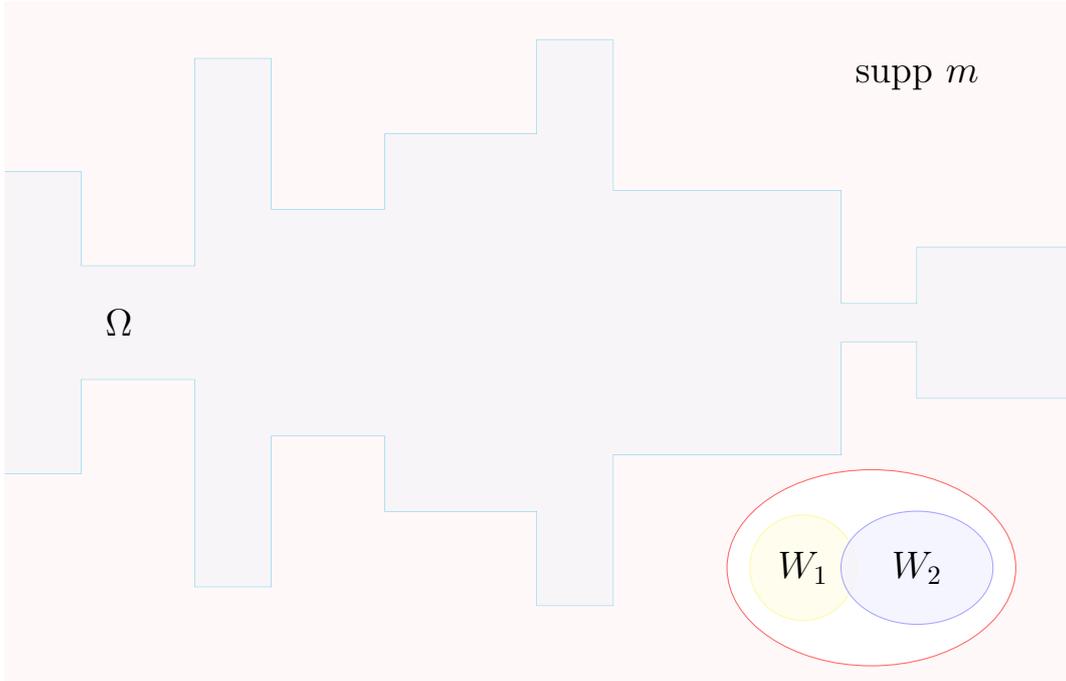
\begin{figure}[!ht]
    \centering
    \begin{tikzpicture}
    \draw [cyan, xshift=4cm]  (-1,0.25)-- (0,0.25) --(0,-1)--(1.5,-1)--(1.5,1.75)--(2.5,1.75)--(2.5,-0.25)--(4,-0.25)-- (4,0.75)--(6,0.75)-- (6,2) --(7,2)-- (7,0) --(10,0)--(10,-1.5)-- (11,-1.5)--(11,-0.75)-- (13,-0.75);
    \draw [cyan, xshift=4cm]  (-1,-3.75)-- (0,-3.75) --(0,-2.5)--(1.5,-2.5)--(1.5,-5.25)--(2.5,-5.25)--(2.5,-3.25)--(4,-3.25)-- (4,-4.25)--(6,-4.25)-- (6,-5.5) --(7,-5.5)-- (7,-3.5) --(10,-3.5)--(10,-2)-- (11,-2)--(11,-2.75)-- (13,-2.75)--(13,-2.75);
    \fill[cyan!5, xshift=4cm] (-1,0.25)-- (0,0.25) --(0,-1)--(1.5,-1)--(1.5,1.75)--(2.5,1.75)--(2.5,-0.25)--(4,-0.25)-- (4,0.75)--(6,0.75)-- (6,2) --(7,2)-- (7,0) --(10,0)--(10,-1.5)-- (11,-1.5)--(11,-0.75)-- (13,-0.75) --(13,-2.75)-- (13,-2.75)--(11,-2.75)-- (11,-2)--(10,-2)--(10,-3.5)-- (7,-3.5)--(7,-5.5)-- (6,-5.5)--(6,-4.25)-- (4,-4.25)--(4,-3.25)--(2.5,-3.25)--(2.5,-5.25)--(1.5,-5.25)--(1.5,-2.5)--(0,-2.5)-- (0,-3.75)--(-1,-3.75)--(-1,0.25);
    \filldraw[red!5, xshift=4cm, opacity=0.5](-1,0.25)-- (-1,-6.5)--(13,-6.5)--(13,2.5)--(-1,2.5)--(-1,0.25);
    \filldraw[color=white,xshift=12.5cm, yshift=-2.5cm](1.9,-2.5) ellipse (1.9 and 1.3);
    \draw[color=red!60,xshift=12.5cm, yshift=-2.5cm](1.9,-2.5) ellipse (1.9 and 1.3);
    \filldraw[color=yellow!50, fill=yellow!10, xshift=2cm, yshift=-0.5cm, opacity=0.8](11.5,-4.5) circle (0.7);
    \filldraw[color=blue!50, fill=blue!5, xshift=12cm, yshift=-2.5cm, opacity=0.7] (3,-2.5) ellipse (1 and 0.75);
    \node[xshift=12cm, yshift=-2.5cm] at (3,4) {$\raisebox{-.35\baselineskip}{\Large\ensuremath{\supp\,m}}$};
    \node[xshift=12cm, yshift=-2.5cm] at (3,-2.5) {$\raisebox{-.35\baselineskip}{\Large\ensuremath{W_2}}$};
    \node[xshift=12cm, yshift=-2.5cm] at (1.5,-2.5) {$\raisebox{-.35\baselineskip}{\Large\ensuremath{W_1}}$};
    \node[xshift=-1.5cm] at (6,-1.75) {$\raisebox{-.35\baselineskip}{\Large\ensuremath{\Omega}}$};
\end{tikzpicture}
    \caption{Graphical illustration of the geometric setting of part \ref{item 1 characterization of uniqueness} of Theorem~\ref{thm: characterization of uniqueness}. Here the set $\Omega\subset \R^n$ is a domain which is bounded in one direction and the measurements are performed in the nonempty open subsets $W_1,W_2\subset \Omega_e$ that have a nonempty intersection. The support of $m$ is represented in red. The DN map $\Lambda_{\gamma}$ determines $m$, and hence $\gamma$, uniquely. }
    \label{fig: Geometric setting 2}
\end{figure}

\begin{figure}[!ht]
    \centering
    \begin{tikzpicture}
    \draw [cyan, xshift=4cm]  (-1,0.25)-- (0,0.25) --(0,-1)--(1.5,-1)--(1.5,1.75)--(2.5,1.75)--(2.5,-0.25)--(4,-0.25)-- (4,0.75)--(6,0.75)-- (6,2) --(7,2)-- (7,0) --(10,0)--(10,-1.5)-- (11,-1.5)--(11,-0.75)-- (13,-0.75);
    \draw [cyan, xshift=4cm]  (-1,-3.75)-- (0,-3.75) --(0,-2.5)--(1.5,-2.5)--(1.5,-5.25)--(2.5,-5.25)--(2.5,-3.25)--(4,-3.25)-- (4,-4.25)--(6,-4.25)-- (6,-5.5) --(7,-5.5)-- (7,-3.5) --(10,-3.5)--(10,-2)-- (11,-2)--(11,-2.75)-- (13,-2.75)--(13,-2.75);
    \fill[cyan!5, xshift=4cm] (-1,0.25)-- (0,0.25) --(0,-1)--(1.5,-1)--(1.5,1.75)--(2.5,1.75)--(2.5,-0.25)--(4,-0.25)-- (4,0.75)--(6,0.75)-- (6,2) --(7,2)-- (7,0) --(10,0)--(10,-1.5)-- (11,-1.5)--(11,-0.75)-- (13,-0.75) --(13,-2.75)-- (13,-2.75)--(11,-2.75)-- (11,-2)--(10,-2)--(10,-3.5)-- (7,-3.5)--(7,-5.5)-- (6,-5.5)--(6,-4.25)-- (4,-4.25)--(4,-3.25)--(2.5,-3.25)--(2.5,-5.25)--(1.5,-5.25)--(1.5,-2.5)--(0,-2.5)-- (0,-3.75)--(-1,-3.75)--(-1,0.25);
    \filldraw[color=yellow!50, fill=yellow!10, xshift=2cm](2,1.25) circle (0.7);
    \filldraw[color=blue!50, fill=blue!5, xshift=12cm, yshift=-2cm] (3,-2.5) ellipse (1 and 0.75);
    \node[xshift=12cm, yshift=-2cm] at (3,-2.5) {$\raisebox{-.35\baselineskip}{\Large\ensuremath{W_2}}$};
    \node[xshift=2cm] at (2,1.25) {$\raisebox{-.35\baselineskip}{\Large\ensuremath{W_1}}$};
    \node[xshift=-2cm] at (6,-1.75) {$\raisebox{-.35\baselineskip}{\Large\ensuremath{\Omega}}$};
    \filldraw [color=orange!80, fill = orange!5, xshift=8cm, yshift=-2cm,opacity=0.8] plot [smooth cycle] coordinates {(-2,0.75)(-2,3) (-1,1) (3,1.5) (6,-3.5) (3,-1) (-3,-0.25)};
     \filldraw [color=red!60, fill = red!5, xshift=10cm, yshift=-2cm, opacity=0.8] plot [smooth cycle] coordinates {(0,3) (0,4.5) (1,4) (2.2,4) (2.2,1.8) (5.6,1.65) (4, 0.5) (3.65, -1.3)(3.1,-0.7) (1,0)};
    \node[xshift=2cm] at (7,-1.75) {$\raisebox{-.35\baselineskip}{\Large\ensuremath{\supp\,m_1}}$};
    \node[xshift=2cm] at (10.7,-1.05) {$\raisebox{-.35\baselineskip}{\Large\ensuremath{\supp\,m_2}}$};
\end{tikzpicture}
    \caption{Graphical illustration of the geometric setting of part \ref{item 2 characterization of uniqueness} of Theorem~\ref{thm: characterization of uniqueness}. Here the set $\Omega\subset \R^n$ is a domain which is bounded in one direction and the measurements are performed in the disjoint nonempty open subsets $W_1,W_2\subset \Omega_e$. Moreover, the supports of the background deviations $m_1,m_2$, which are represented in orange and red, respectively, do not necessarily coincide in the exterior $\Omega_e$ and can even be disjoint as illustrated above. In this case uniqueness may be lost.}
    \label{fig: Geometric setting 1}
\end{figure}
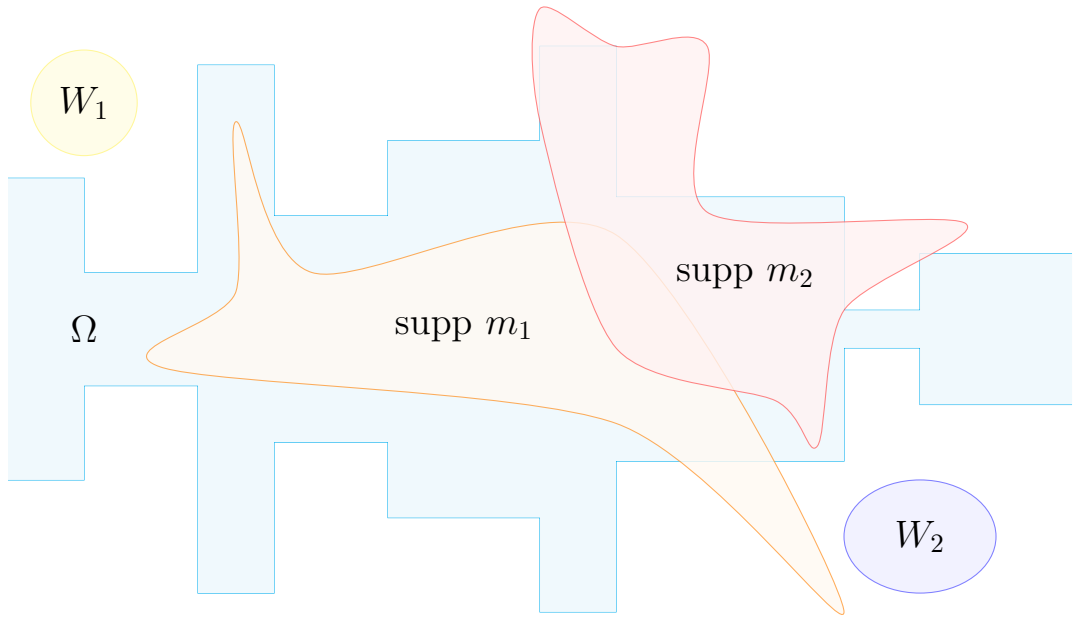

\subsubsection{Fractional Schrödinger equations} Later in this article we present many applications of Theorem \ref{theorem:GeneralFractionalCalderon} in which we always fix $L = (-\Delta)^t$ (in the weak form) for some $t \in \R_+ \setminus \Z$ with $t \leq s$. One could also take many other bilinear forms in the place of $L$, for example:
\begin{itemize}
    \item One can take $L$ to be the bilinear form related to the Bessel potential $\vev{D}^t$ whenever $t \in \R \setminus \Z$ \cite[Theorem 1.6 and Remark 1.7]{kenigetal2020-uniqueDispersive}.
    \item One can take $L$ to be the bilinear form related to the nonlocal operator $\mathcal{L}^t$ where $t \in (0,1)$ and $\mathcal{L}$ is an elliptic second order operator $\mathcal{L} = \nabla \cdot (A \nabla u)$ \cite[Theorem 1.2]{GLX-calderon-nonlocal-elliptic-operators} and \cite[Proposition 3.2]{ghosh2021calderon}.
    \item See \cite[Theorem 1.2]{covi2021uniquenessQuasiLocal} for the fractional Schrödinger operators with rapidly decaying quasilocal perturbations having the UCP.
    \item See \cite[Theorem 4.1]{covi2021uniquenessQuasiLocal} for the fractional Schrödinger operators with perturbations of finite propagation having the UCP.
    \item Any local perturbation to an operator with the symmetric UCP (this is by now a well-known fact, see e.g. the discussion in \cite[Section 3.1]{CMR20}).
\end{itemize}

\begin{remark}We note that this does not imply that the methods used in this article would solve the inverse problems studied in \cite{covi2021uniquenessQuasiLocal} since the role of $q$ in our work is more restricted. This only implies that the inverse problems for local perturbations of the fractional Schrödinger operators used in \cite{covi2021uniquenessQuasiLocal} could be studied directly with Theorems \ref{thm:PoincUnboundedDoms} and \ref{theorem:GeneralFractionalCalderon}. One would have to refine Theorem \ref{theorem:GeneralFractionalCalderon} from local perturbations into the direction of nonlocal perturbations to have access to the fractional Calderón problem for quasilocal perturbations on unbounded domains. The generalization to the case of perturbations with finite propagation should be now a rather easy problem to solve but the case of quasilocal perturbations would require one to be more careful. We do not pursue this matter here as this would put us too far away from the proof of Theorem \ref{thm: characterization of uniqueness}, and our work and the work \cite{covi2021uniquenessQuasiLocal} were put through independently around the same time.
\end{remark}

We apply Theorems \ref{thm:PoincUnboundedDoms} and \ref{theorem:GeneralFractionalCalderon} to prove uniqueness results for the fractional Calderón problems on domains that are bounded in one direction. The methods needed to bypass the lack of compact Sobolev embeddings on unbounded domains turn out to be useful also in the case of bounded domains and allow us to slightly extend the recent results in \cite{RS-fractional-calderon-low-regularity-stability,CMR20,CMRU20-higher-order-fracCald} in the usual setting of bounded domains.

Rather than introducing right now all new concrete model examples where our methods apply, we only point to the following results that are stated and proved later in this article:
\begin{itemize}
    \item Remark \ref{rmk: exteriorPDOcase} explains how Corollary \ref{cor:fracCaldExteriorDet} on exterior determination extends to the case of PDO perturbations with variable coefficients.
    \item Proposition \ref{prop:unboundedExamples} gives example model classes with higher order local perturbations (i.e. the perturbations have higher order than the fractional Laplacian) for all domains that have nonempty exterior and the fractional Calder\'on problem is uniquely solvable. Especially, these domains are not assumed to be bounded in one direction.
    \item Theorem \ref{main-theorem-singular} gives example models with linear local lower order perturbations on domains that are bounded in one direction.
    \item Theorem \ref{thm:schrodingeruniquenessUnbounded} gives further example models on domains that are bounded in one direction in the case where the perturbations are functions rather than operators.
    \item Theorems \ref{thm:schrodingeruniqueness} and \ref{main-theorem-singular-bounded} give slight generalizations of the existing theorems on bounded domains.
\end{itemize}

Needless to say, there are many other situations where the combination of existing results and the ideas presented in this article allow various generalizations for the fractional Calderón problems, the Runge approximation results and related problems. More interesting research questions, mathematically speaking, would be to understand if the single measurement, stability, quantitative Runge approximation, and nonlinear and nonlocal perturbation results \cite{BGU-lower-order-nonlocal-perturbations,covi2021uniquenessQuasiLocal,covietal2021calderon-directionally-antilocal,GRSU-fractional-calderon-single-measurement,LL-fractional-semilinear-problems,RS-fractional-calderon-low-regularity-stability} have some extensions to the cases of unbounded domains or higher order perturbations, where higher order perturbations are understood in the sense that the local part has higher order than the nonlocal part of an equation.

\section{Preliminaries}\label{sec:preliminaries}

We introduce in this subsection key definitions and basic notation. More sophisticated definitions and notation are given in later subsections when needed.

\subsection{Spaces of distributions}\label{subsec: spaces of dist}
Throughout the article we use the convention that $n \geq 1$ and $\Omega \subset \R^n$ is an open set. We denote by $\schwartz(\R^n)$ and $\tempered(\R^n)$ Schwartz functions and tempered distributions respectively. We say that $f\in C^{\infty}(\R^n)$ belongs to $\slowly(\R^n)$ if for any $\alpha\in\N^n_0$ there exists an $N\in\N_0$ such that $|\partial^{\alpha}f(x)| \leq C\langle x\rangle^N$, where $\langle x\rangle=\sqrt{1+|x|^2}$ is the Japanese bracket.
 
We let $\distr(\Omega)$ stand for the space of distributions on $\Omega$ and by $\cdistr(\Omega)$ the space of compactly supported distributions on $\Omega$.

We define the Fourier transform $\fourier\colon \schwartz(\R^n)\to \schwartz(\R^n)$ by
\begin{equation}
    \fourier f(\xi) \vcentcolon = \int_{\R^n} f(x)e^{-ix \cdot \xi} \,dx,
\end{equation}
which is occasionaly also denoted by $\hat{f}$. By duality it can be extended to the space of tempered distributions and will again be denoted by $\fourier u = \hat{u}$, where $u \in \tempered(\R^n)$, and we denote the inverse Fourier transform by $\ifourier$.

\subsection{Fractional Sobolev spaces, Bessel potential spaces and fractional Laplacians}\label{sec:Bessel-basics}

We denote by $W^{s,p}(\Omega)$ the fractional Sobolev spaces, when $s\in (0,\infty)\setminus\N$, $1 \leq p < \infty$. These spaces are also called Slobodeckij spaces or Gagliardo spaces. In particular, for $0<s<1$, $1\leq p<\infty$ the space $W^{s,p}(\Omega)$ consists of functions $u\in L^p(\Omega)$ such that $\frac{u(x)-u(y)}{|x-y|^{n/p+s}}\in L^p(\Omega\times \Omega)$ and it is endowed with the norm
\[
    \|u\|_{W^{s,p}(\Omega)}\vcentcolon =\left(\|u\|_{L^p(\Omega)}^p+[u]_{W^{s,p}(\Omega)}^p\right)^{1/p},
\]
where
\[
    [u]_{W^{s,p}(\Omega)}\vcentcolon =\left(\int_{\Omega}\int_{\Omega}\frac{|u(x)-u(y)|^p}{|x-y|^{n+sp}}\,dxdy\right)^{1/p}
\]
is the so called Gagliardo (semi-) norm. For $s>1$ such that $s=m+\sigma$ with $m\in\N,\sigma\in (0,1)$ we set 
\[
W^{s,p}(\Omega)\vcentcolon =\{\,u\in W^{m,p}(\Omega)\,;\, \partial^{\alpha} u\in W^{\sigma,p}(\Omega)\quad \forall |\alpha|=m \},
\]
where $W^{m,p}(\Omega)$ with $m\in \N, 1\leq p\leq \infty$ is the classical Sobolev space, and endow it with the norm
\[
\|u\|_{W^{s,p}(\Omega)}\vcentcolon =\left(\|u\|_{W^{m,p}(\Omega)}^p+\sum_{|\alpha|=m}[\partial^{\alpha}u]_{W^{\sigma,p}(\Omega)}^p\right)^{1/p}.
\]
We mainly use them in Section \ref{sec:poincare} in which we recall recent results on the Poincar\'e inequalities on certain subspaces of $W^{s,p}(\R^n)$ and Section \ref{sec:conductivityCalderon} dealing with the fractional conductivity equation but for the sake of presentation this part of the proof is done in Appendix \ref{subsubsec: Multiplication map in Bessel potential spaces}.\\ 
Next we introduce the Bessel potential spaces $H^{s,p}(\R^n)$ and various notions of local Bessel poential spaces which we introduce next. We define the Bessel potential of order $s \in \R$ as a Fourier multiplier $\vev{D}^s\colon \tempered(\R^n) \to \tempered(\R^n)$
\begin{equation}\label{eq: Bessel pot}\vev{D}^s u \vcentcolon = \ifourier(\vev{\xi}^s\hat{u}).
\end{equation} We define the Bessel potential spaces for any $1 \leq p \leq \infty$ and $s \in \R$ by \begin{equation}\label{eq: Bessel pot spaces}
    H^{s,p}(\R^n) \vcentcolon = \{\, u \in \tempered(\R^n)\,;\, \vev{D}^su \in L^p(\R^n)\,\}
\end{equation}
and it is equipped with the norm $\norm{u}_{H^{s,p}(\R^n)} \vcentcolon = \norm{\vev{D}^su}_{L^p(\R^n)}$.

\begin{definition} 
Let $\Omega, F \subset \R^n$ be open and closed sets respectively, $1 \leq p \leq \infty$ and $s \in \R$. We define the following \emph{local Bessel potential spaces}:
\begin{equation}\begin{split}\label{eq: local bessel pot spaces}
    H^{s,p}(\Omega) &\vcentcolon = \{\,u|_{\Omega} \,;\, u \in H^{s,p}(\R^n)\,\},\\
    \widetilde{H}^{s,p}(\Omega) &\vcentcolon = \overline{C_c^\infty(\Omega)}^{H^{s,p}(\R^n)},\\
    H_0^{s,p}(\Omega) &\vcentcolon =\overline{C_c^\infty(\Omega)}^{H^{s,p}(\Omega)},\\
    H_F^{s,p}(\R^n) &\vcentcolon =\{\,u \in H^{s,p}(\R^n)\,;\, \supp(u) \subset F\,\}.
    \end{split}
\end{equation}
The space $H^{s,p}(\Omega)$ is equipped with the quotient norm
\begin{equation}
\label{eq: norm local space}
    \norm{u}_{H^{s,p}(\Omega)} \vcentcolon = \inf\{\,\norm{w}_{H^{s,p}(\R^n)}\,;\, w \in H^{s,p}(\R^n), w|_\Omega = v\,\}.
\end{equation}
The spaces $\widetilde{H}^{s,p}(\Omega)$ and $H_F^{s,p}(\R^n)$ are equipped with the $H^{s,p}(\R^n)$ norm. The space $H_0^{s,p}(\Omega)$ is equipped with the $H^{s,p}(\Omega)$ norm. As usual we set $H^s \vcentcolon = H^{s,2}$.
\end{definition}

We have that $\tilde{H}^{s,p}(\Omega) \hookrightarrow \tilde{H}_{\overline{\Omega}}^{s,p}(\R^n)$, $(\tilde{H}^{s,p}(\Omega))^* = H^{-s,p'}(\Omega)$ and 
$\tilde{H}^{s,p}(\Omega) = (H^{-s,p'}(\Omega))^*$ for every $1 < p < \infty$ and $s \in \R$.

Finally, we define the fractional Laplacian of order $s \in \R_+$ as a Fourier multiplier
\begin{equation}\label{eq:fracLapFourDef}
    (-\Delta)^{s/2} u = \ifourier(\abs{\xi}^s\hat{u}) 
\end{equation}
for $u \in \tempered(\R^n)$ whenever the formula is well-defined. In particular, the fractional Laplacian is a well-defined bounded mapping $(-\Delta)^{s/2}\colon H^{t,p}(\R^n) \to H^{t-s,p}(\R^n)$ for every $1 \leq p < \infty$, $s \geq0$ and $t \in \R$.

\section{Generalized fractional Calderón problem}\label{sec:GeneralizedFractionalCalderon}

In this section, we formulate and prove a uniqueness result for a generalized fractional Calderón problem. The structure of our argument follows similar ideas as introduced by Ghosh, Salo and Uhlmann \cite{GSU20} and later modified in \cite{CLR18-frac-cald-drift,CMRU20-higher-order-fracCald} to obtain uniqueness results for some special cases of nonsymmetric local perturbations of the nonlocal Schr\"odinger equation. The axiomatic approach to be taken here has several benefits. We observe that the argument can be applied on all domains with nonempty exterior, and in general, the argument works whenever the direct problem is well-posed. One can even consider and prove uniqueness results for some higher order local perturbations of the fractional Laplacian. Either of these two examples, unbounded domains or higher order perturbations, do not appear in the earlier literature to the best of our knowledge. These general considerations in part motivate the rest of article. In later sections of this article, we will study fractional Poincaré inequalities on unbounded domains and the related decompositions of Sobolev multipliers and PDOs to be able to provide new nontrivial examples where the theory presented in this section applies.

\begin{lemma}[Well-posedness]\label{lemma:generalWellposedness} Let $s \in \R$, and $\Omega \subset \R^n$ be open set such that $\Omega_e \neq \emptyset$. Let $B\colon H^s(\R^n) \times H^s(\R^n) \to \R$ be a bounded bilinear form that is (strongly) coercive in $\tilde{H}^s(\Omega)$, that is, there exists some $c >0$ such that $B(u,u) \geq c\norm{u}_{H^s(\R^n)}^2$ for all $u \in \tilde{H}^s(\Omega)$. Then the following hold:
\begin{enumerate}[(i)]
\item For any $f \in H^s(\R^n)$ and $F \in (\tilde{H}^s(\Omega))^*$ there exists a unique $u \in H^s(\R^n)$ such that $u-f \in \tilde{H}^s(\Omega)$ and $B(u,\phi) = F(\phi)$ for all $\phi \in \tilde{H}^s(\Omega)$. When $F\equiv 0$, we denote this unique solution by $u_f$.\label{item:uniqunessOfSolutions}

\item If $f_1-f_2 \in \tilde{H}^s(\Omega)$ and $u_i$ are solutions to the problems $u_i-f_i \in \tilde{H}^s(\Omega)$ and $B(u_i,\cdot)=F$ in $\tilde{H}^s(\Omega)$ for $i=1,2$, then $u_1 = u_2$.\label{item:independenceOfExteriorData}
\item Let $X \vcentcolon = H^s(\R^n)/\tilde{H}^s(\Omega)$ be the abstract trace space. Then the exterior DN map $\Lambda_B\colon X \to X^*$ defined by $\Lambda_B[f][g] \vcentcolon = B(u_f,g)$ for $[f],[g] \in X$ is a well-defined bounded linear map.\label{item:DNmaps}
\item The above statements hold true for the adjoint bilinear form $B^*$ and the corresponding unique solutions, when $F\equiv 0$, are denoted by $u_f^*$. Moreover,
$\Lambda_B[f][g] = \Lambda_{B^*}[g][f]$ for any $[f],[g] \in X$, that is, $\Lambda_B^*=\Lambda_{B^*}$. \label{item:adjointBilinear}
\end{enumerate}
\end{lemma}
\begin{remark}\label{remark:generalizationCalderon} We remark that the assumption on the coercivity of $B$ can be replaced by the assumption that the exterior and interior value problems have unique solutions depending continuously on the data. 
If we similarly assume that this is true for the adjoint problem, then the presented proofs work without any changes. This applies in the whole section.
\end{remark}
\begin{proof} \ref{item:uniqunessOfSolutions} Let $f \in H^s(\R^n)$, $F \in (\tilde{H}^s(\Omega))^*$ and set $\tilde{u} \vcentcolon = u-f$. By linearity the problem is reduced to finding the unique $\tilde{u} \in \tilde{H}^s(\Omega)$ such that $B(\tilde{u},\phi) = \tilde{F}(\phi)$ for all $\phi \in \tilde{H}^s(\Omega)$ where $\tilde{F} \vcentcolon = F - B(f,\cdot)$. Observe that the modified functional $\tilde{F}$ belongs to $(\tilde H^s(\Omega))^*$, since there holds
$$|\tilde{F}(\phi)| \leq |F(\phi)| + |B(f,\phi)| \leq (\aabs{F}_{(\widetilde{H}^s(\Omega))^*}+C\|f\|_{H^s(\mathbb R^n)})\|\phi\|_{H^s(\mathbb R^n)}$$
for all $\phi \in \tilde{H}^s(\Omega)$. The Lax--Milgram theorem implies that there exists a bounded linear operator $G\colon (\widetilde{H}^{s}(\Omega))^* \rightarrow \widetilde H^s(\Omega)$ associating to each functional in $(\widetilde{H}^{s}(\Omega))^*$ its unique representative in the bilinear form $B(\cdot,\cdot)$ on $\widetilde H^s(\Omega)$. Thus $\tilde u \vcentcolon = G \tilde F$ solves $$ B(\tilde u,v) = \tilde F(v)\quad \mbox{for all}\quad v\in \widetilde H^s(\Omega) $$
and it is the required unique solution $\tilde u \in \widetilde H^s(\Omega)$. Moreover, there holds \begin{equation}\label{eq: norm estimate abstract well-posedness}\norm{u}_{H^s(\R^n)} \leq c^{-1}\|\tilde{F}\|_{(\tilde{H}^s(\Omega))^*}+\norm{f}_{H^s(\R^n)} \leq C(\aabs{F}_{(\widetilde{H}^s(\Omega))^*}+\|f\|_{H^s(\mathbb R^n)})
\end{equation}
for some $C>0$.

\ref{item:independenceOfExteriorData} Since $f_1-f_2\in \tilde{H}^s(\Omega)$ there exists $\phi \in \tilde{H}^s(\Omega)$ such that $f_2=f_1+\phi$ and the function $u\vcentcolon =u_2-u_1 \in \tilde{H}^s(\Omega)$ solves the homogenous problem with exterior value $\phi$.
Since $w=0$ is also a solution to this problem we obtain by uniqueness of solutions that $u_1\equiv u_2$ in $\R^n$.

\ref{item:DNmaps} By \ref{item:uniqunessOfSolutions} and \ref{item:independenceOfExteriorData}, we have $$ B(u_{f+\phi},g+\psi) = B(u_f,g) + B(u_f,\psi) = B(u_f,g)$$
for all $f,g \in H^s(\mathbb R^n)$ and $\phi, \psi \in \tilde{H}^s(\Omega)$. Moreover, if $f,g \in H^s(\mathbb R^n)$ and $\phi,\psi \in \widetilde H^s(\Omega)$, then
\begin{align*}
|\langle\Lambda_B[f],[g]\rangle| & = |B(u_{f-\phi},g-\psi)| \leq C \|u_{f-\phi}\|_{H^s(\R^n)} \|g-\psi\|_{H^s(\R^n)}\\ &\leq C \|f-\phi\|_{H^s(\R^n)} \|g-\psi\|_{H^s(\R^n)}
\end{align*}
for some $C>0$. 
This implies that for some $C>0$ it holds that
$$ |\langle\Lambda_B[f],[g]\rangle| \leq C\inf_{\phi\in\widetilde H^s(\Omega)}\|f-\phi\|_{H^s(\R^n)}\inf_{\psi\in\widetilde H^s(\Omega)}\|g-\psi\|_{H^s(\R^n)} = C \|[f]\|_X \|[g]\|_X$$ and hence $\Lambda_B[f]\in X^*$ for all $[f]\in X$. If $f_1,f_2 \in H^s(\R^n)$, then $u_{f_1+f_2}=u_{f_1}+u_{f_2}$ since the solutions are unique and $u_{f_1}+u_{f_2}$ solves the homogeneous equation with the exterior condition $u_{f_1}+u_{f_2}-(f_1+f_2) \in \tilde{H}^s(\Omega)$. Therefore we have shown that $\Lambda_B\colon X\to X^*$ is a bounded linear operator.

\ref{item:adjointBilinear} Notice that $B^*$ is also a bounded bilinear form with the same coercivity estimate. Therefore \ref{item:uniqunessOfSolutions}, \ref{item:independenceOfExteriorData} and \ref{item:DNmaps} hold when $B$ is replaced by $B^*$. Suppose that $u_f$ solves $B(u_f,\cdot)=0$ in $\tilde{H}^s(\Omega)$ with $u_f-f \in \tilde{H}^s(\Omega)$, and $u_g^*$ solves $B^*(u_g^*,\cdot)=0$ in $\tilde{H}^s(\Omega)$ with $u_g^* - g \in \tilde{H}^s(\Omega)$. Now 
\begin{equation}
\label{eq: DN maps coincide abstract}
    \Lambda_B[f][g] = B(u_f,u_g^*) = B^*(u_g^*,u_f) = \Lambda_{B^*}[g][f].\qedhere
\end{equation}
\end{proof}

\begin{theorem}[Runge approximation property]\label{lemma:GeneralRungeApproximation} Let $s \in \R$ and $\Omega \subset \R^n$ be an open set such that $\Omega_e \neq \emptyset$. Let $L,q\colon H^s(\R^n) \times H^s(\R^n) \to \R$ be bounded bilinear forms and assume that $q$ is local and that $B_{L,q} \vcentcolon = L + q$ is (strongly) coercive in $\tilde{H}^s(\Omega)$.
\begin{enumerate}[(i)]
\item If $L$ has the right UCP on a nonempty open set $W \subset \Omega_e$, then $\mathcal{R}(W) \vcentcolon = \{\,u_f -f \,;\, f \in C_c^\infty(W)\,\} \subset \tilde{H}^s(\Omega)$ is dense.\label{item:generalRunge}
\item If $L$ has the left UCP on a nonempty open set $W \subset \Omega_e$, then $\mathcal{R}^*(W) \vcentcolon = \{\,u_g^* -g \,;\, g \in C_c^\infty(W)\,\} \subset \tilde{H}^s(\Omega)$ is dense.\label{item:generaladjointRunge}
\end{enumerate}
\end{theorem}

\begin{proof} We only prove here \ref{item:generalRunge}, since the proof of \ref{item:generaladjointRunge} is completely analogous.

Since $\mathcal{R}(W)\subset \tilde{H}^s(\Omega)$ is a subspace it suffice by the Hahn-Banach theorem to show that if $F\in (\tilde{H}^s(\Omega))^*$ vanishes on $\mathcal{R}(W)$, then $F$ is identically zero. Hence, let $F \in (\tilde{H}^s(\Omega))^*$ and suppose that $F(u_f-f) = 0$ for all $f \in C_c^\infty(W)$. By the Lax--Milgram theorem there exists unique $\phi \in \tilde{H}^s(\Omega)$ such that $B^*(\phi,\cdot)=F$. Notice first that $B^*(\phi,u_f) = B(u_f,\phi)=0$ since $u_f$ is a solution. We have also that $q^*(\phi,f)=q(f,\phi) = 0$ since $\supp(f) \cap \supp(\phi) \subset \Omega_e \cap \overline{\Omega} = \emptyset$ and $q$ is local. We may now calculate that
\begin{equation}
    0 = -F(u_f-f) = -B^*(\phi,u_f-f) = B^*(\phi,f) = L^*(\phi,f) + q^*(\phi,f) = L(f,\phi)
\end{equation}
for all $f \in C_c^\infty(W)$. Since $\phi|_W = 0$, we have by the right UCP that $\phi\equiv 0$ and therefore $F\equiv 0$.
\end{proof}

In Figure \ref{fig: Runge Approximation}, we give a concrete example model where Theorem \ref{lemma:GeneralRungeApproximation} applies and the domain is not bounded in one direction.

\begin{figure}[!ht]
    \centering
    \begin{tikzpicture}
    \fill[cyan!5, xshift=4cm,yshift=2.5cm] (-1,0.25)-- (13,0.25) --(13,-2.75)--(-1,-2.75)--(-1,0.25);
    \filldraw[color=yellow!50, fill=yellow!10, xshift=7.5cm](2,1.25) circle (1);
    \node[xshift=7.5cm, yshift=-0.6cm] at (2,1.25)
    {$\raisebox{-.35\baselineskip}{\Large\ensuremath{B_{\epsilon}}}$};
    \node[xshift=3.5cm] at (2,1.25)
    {$\raisebox{-.35\baselineskip}{\Large\ensuremath{((-\Delta)^s+\delta)u=0}}$};
    \node[xshift=7.5cm, yshift=0.4cm] at (2,1.25)
    {$\raisebox{-.35\baselineskip}{\Large\ensuremath{u=f}}$};
\end{tikzpicture}
    \caption{An example illustrating Theorem \ref{lemma:GeneralRungeApproximation}. Let us denote $B_\epsilon = B(0;\epsilon) \subset \R^n$ for any $\epsilon >0$ and $n \geq 1$. For any $\epsilon, \delta >0$ and $\Omega := \R^n \setminus \overline{B_\epsilon}$, we have that the restriction to $\R^n \setminus \overline{B_\epsilon}$ of the unique solutions $u_f$ to the equation $((-\Delta)^s+\delta)u=0$ in $\R^n \setminus \overline{B_\epsilon}$ are dense in $\tilde{H}^s (\R^n \setminus \overline{B_\epsilon})$ with exterior conditions $f \in C_c^\infty(B_\epsilon)$. This means \emph{heuristically} speaking that \emph{all} functions in $H^s(\R^n)$ are almost $s$-harmonic up to an error depending on $\delta, \epsilon > 0$. This observation is analogous to the celebrated result on bounded domains \cite{DSV-all-functions-are-s-harmonic} where one may take $\delta=0$, but this possible choice of $\delta$ is also true on domains that are bounded in one direction. We do not perform a quantitative analysis of this phenomenon in this article.}
    \label{fig: Runge Approximation}
\end{figure}

\begin{lemma}[Alessandrini identity]\label{lemma:generalAlessandrini} Let $s \in \R$, and $\Omega \subset \R^n$ be an open set such that $\Omega_e \neq \emptyset$. Let $B_i\colon H^s(\R^n) \times H^s(\R^n) \to \R$, $i=1,2$, be bounded bilinear forms that are (strongly) coercive in $\tilde{H}^s(\Omega)$. Then
\begin{equation}
    (\Lambda_{B_1}-\Lambda_{B_2})[f][g] = (B_1-B_2)(u_f,u_g^*)
\end{equation}
for any $[f],[g] \in X$ where 
\begin{enumerate}[(i)]
\item\label{item 1 abstract alessandrini} $u_f$ solves $B_1(u_f,\cdot)=0$ in $\tilde{H}^s(\Omega)$ with $u_f-f \in \tilde{H}^s(\Omega)$;
\item\label{item 2 abstract alessandrini} $u_g^*$ solves $B_2^*(u_g^*,\cdot)=0$ in $\tilde{H}^s(\Omega)$ with $u_g^* - g \in \tilde{H}^s(\Omega)$.
\end{enumerate}
\end{lemma}
\begin{proof} Notice that the solutions $u_f$ and $u_g^*$ exist by Lemma \ref{lemma:generalWellposedness}. We can directly calculate using \ref{item:adjointBilinear} of Lemma \ref{lemma:generalWellposedness} that
\begin{equation}
    (\Lambda_{B_1}-\Lambda_{B_2})[f][g] = \Lambda_{B_1}[f][g]-\Lambda_{B_2^*}[g][f] = B_1(u_f,u_g^*)-B_2^*(u_g^*,u_f) = (B_1-B_2)(u_f,u_g^*).\qedhere
\end{equation}
\end{proof}

We can now prove the uniqueness theorem for the generalized fractional Calderón problem (Theorem~\ref{theorem:GeneralFractionalCalderon}).

\begin{proof}[Proof of Theorem \ref{theorem:GeneralFractionalCalderon}] Using the Alessandrini identity (Lemma \ref{lemma:generalAlessandrini}) we have \begin{equation}
    0=(\Lambda_{L,q_1}-\Lambda_{L,q_2})[f][g] = (B_{L,q_1}-B_{L,q_2})(u_f,u_g^*) = (q_1-q_2)(u_f,u_g^*)
\end{equation}
for all $f \in C_c^\infty(W_1)$ and $g \in C_c^\infty(W_2)$, where 
\begin{enumerate}[(i)]
\item\label{item 1 proof abstract DN maps} $u_f$ solves $B_{L,q_1}(u_f,\cdot)=0$ in $\tilde{H}^s(\Omega)$ with $u_f-f \in \tilde{H}^s(\Omega)$;
\item\label{item 2 proof abstract DN maps} $u_g^*$ solves $B_{L,q_2}^*(u_g^*,\cdot)=0$ in $\tilde{H}^s(\Omega)$ with $u_g^* - g \in \tilde{H}^s(\Omega)$.
\end{enumerate}

By the Runge approximation property (Lemma \ref{lemma:GeneralRungeApproximation}), we have that there exists sequences $u_{f_k}-f_k \to v$ and $u_{g_k}^* -g_k \to v^*$ as $k \to \infty$ where $f_k \in C_c^\infty(W_1), g_k \in C_c^\infty(W_2)$ and $v,v^* \in \tilde{H}^s(\Omega)$. Since $q_1,q_2$ are bounded bilinear forms and local, we find that
\begin{equation}
    (q_1-q_2)(v,v^*) = \lim_{k\to\infty}(q_1-q_2)(u_{f_k}-f_k,u_{g_k}^* -g_k) = \lim_{k\to\infty}(q_1-q_2)(u_{f_k},u_{g_k}^*) = 0
\end{equation}
where the first identity follows from the boundedness of $q_1-q_2$, the second identity follows from the support conditions and the locality of $q_1-q_2$, and the last identity follows from the assumption that the exterior DN maps are equal. This implies that $q_1 = q_2$ in $\tilde{H}^s(\Omega) \times \tilde{H}^s(\Omega)$.
\end{proof}

\begin{proof}[Proof of Corollary \ref{cor:fracCaldExteriorDet}]
First of all, we have that $q_1 = q_2$ in $\tilde{H}^s(\Omega) \times \tilde{H}^s(\Omega)$ by Theorem \ref{theorem:GeneralFractionalCalderon}. Let us now take $f,g \in C_c^\infty(W)$. We have by the Alessandrini identity (Lemma \ref{lemma:generalAlessandrini}) and the assumption that $\Lambda_{L,q_1}[f][g]=\Lambda_{L,q_2}[f][g]$ that
\begin{equation}
    0 = (q_1-q_2)(u_f,u_g^*).
\end{equation}
The rest is a simple calculation and an approximation argument. Let us write that $B:=q_1-q_2$ and calculate
\begin{equation}\begin{split}
    B(u_f,u_g^*) &= B((u_f-f)+f,(u_g^*-g)+g)\\
    &=B(u_f-f,u_g^*-g)+B(u_f-f,g)+B(f,u_g^*-g)+B(f,g)\\
    &=B(f,g)
\end{split}
\end{equation}
where we used that $u_f-f,u_g^*-g \in \tilde{H}^s(\Omega)$ so that $B(u_f-f,u_g^*-g)$ vanishes by interior determination, and $B(u_f-f,g)$ and $B(f,u_g^*-g)$ vanish by the locality of $B$. This shows that $q_1(f,g)=q_2(f,g)$ for all $f,g \in C_c^\infty(W)$. Approximation gives that $q_1=q_2$ in $\tilde{H}^s(W) \times \tilde{H}^s(W)$. This completess the proof of the first part of the statement.

We next prove the conclusions about the full data case. We fix now that $W = \Omega_e$. Suppose first that $q_1,q_2 \in L^1_\text{loc}(\R^n)$. Now we have that $\int_{\R^n} q_1\phi dx=\int_{\R^n} q_2\phi dx$ for all $\phi \in C_c^\infty(\Omega\cup\Omega_e)$ by interior and exterior determination. This implies that $q_1 = q_2$ a.e. in $\Omega\cup \Omega_e$. If $\partial \Omega$ has measure zero, then $q_1 = q_2$ a.e. in $\R^n$. On the other hand, if $q_1,q_2 \in C(\R^n)$, then we find first that $q_1 = q_2$ in $\Omega \cup \Omega_e$, but this must then also hold in $\R^n = \overline{\Omega \cup \Omega_e}$ by the continuity without assuming anything about the regularity or measure of $\partial \Omega$. 
\end{proof}

\begin{remark}\label{rmk: exteriorPDOcase} Corollary \ref{cor:fracCaldExteriorDet} also applies in the case of local linear perturbations studied in Section \ref{sec:fractionalCalderonMainSec}. One could obtain similar exterior and full data results as stated in the case of zeroth order perturbations in Corollary \ref{cor:fracCaldExteriorDet}. The argument first uses this general result for bilinear forms, and then similar test functions as used in the proof of Theorem \ref{main-theorem-singular} (but now for the exterior conditions). Finally, the full data case can be concluded either using continuity or basic measure theory. This would not involve any new ideas beyond these and the details are thus omitted. 
\end{remark}

We give the following simple examples as corollaries right now. Both results given in the examples are new but we do not aim to write the most general examples into this direction here. We also remark that the generalized Runge approximation result in Lemma \ref{lemma:GeneralRungeApproximation} may find applications in other problems on unbounded domains. We study more sophisticated examples in Section \ref{sec:fractionalCalderonMainSec}.

\begin{proposition}[Example models]\label{prop:unboundedExamples} Let $s \geq 0$, and $\Omega \subset \R^n$ be an open set such that $\Omega_e \neq \emptyset$. The generalized fractional Calderón problem is well-defined and has the uniqueness property (in the weak sense) in the following cases:
\begin{enumerate}[(i)]
    \item (Domains without Poincaré inequalities) For $(-\Delta)^s + q$ in $\Omega$ where $s \in \R_+ \setminus \Z$ and the potential $q$ is uniformly positive and bounded, i.e. $q \in L^\infty(\R^n)$ is such that $q \geq c > 0$ a.e. in $\Omega$ for some $c(q) > 0$.\label{item:example1}
    \item (Higher order perturbations) For $(-\Delta)^t + (-\Delta)^{s/2}(\gamma (-\Delta)^{s/2}) + q$ in $\Omega$ where $t \in \R_+ \setminus \Z$, $s \in 2\Z$ and $t < s$, and $\gamma$, $q$ are uniformly positive and bounded in $\Omega$.\label{item:example2}
\end{enumerate}
\end{proposition}
\begin{remark} In the domains, with finite fractional Poincaré constants, one could formulate results for more general local higher order perturbations. If additionally the domain is bounded, then the spectral theory would allow to consider a large class of local higher order perturbations as long as a weaker coercivity estimate can be established. We note that the problem must then be set in the space whose Sobolev scale depends on the local perturbation rather than the order of the fractional Laplacian. The key observation is that still the fractional Laplacian permits the uniqueness results for the related inverse problems whenever the forward problems are well-posed.
\end{remark}
\begin{proof}\ref{item:example1} The associated bilinear forms are $L(u,v) = \langle (-\Delta)^{s/2}u, (-\Delta)^{s/2}v \rangle$ and $q(u,v) = \langle qu, v \rangle$. It follows that the related bilinear form $L+q$ is equivalent with $H^s(\R^n)$ norm in $\tilde{H}^s(\Omega)$. It follows from \cite[Theorem 1.2]{CMR20} and \cite[Theorem 1.2]{GSU20} that $L$ has the symmetric UCP on any open set, and by the definition $q$ is local. Hence, Theorem \ref{theorem:GeneralFractionalCalderon} applies.

\ref{item:example2} We have now that $L = \langle (-\Delta)^{t/2}u, (-\Delta)^{t/2}v \rangle$ has again the symmetric UCP on any open set for the functions in $H^s(\R^n)$. It is also clear that $L\colon H^s(\R^n) \times H^s(\R^n) \to \R$ is bounded since $0 < t < s$. The bilinear form related to $(-\Delta)^{s/2}(\gamma (-\Delta)^{s/2}) + q$ is given by
\begin{equation}Q(u,v) = \langle \gamma(-\Delta)^{s/2}u, (-\Delta)^{s/2}v \rangle + \langle qu, v \rangle.\end{equation} It is clear that this bilinear form is also bounded in $H^s(\R^n) \times H^s(\R^n)$. The coercivity estimate is also easy to argue since $Q$ is (strongly) coercive in $\tilde{H}^s(\Omega)$ and $L(v,v)\geq 0$ for any $v \in H^s(\R^n)$. Since $s/2 \in \Z$, we have that the bilinear form $Q$ is local, and $L$ has the symmetric UCP on any open set as earlier. Again, the rest follows from Theorem \ref{theorem:GeneralFractionalCalderon}.
\end{proof}

\section{Interpolation of fractional Poincaré constants on unbounded domains}\label{sec:poincare}

The purpose of this section is to extend the fractional Poincar\'e inequality to open sets which are bounded in one direction as introduced in Definition~\ref{def:bounded1dir}. The main result of this section is the proof of Theorem \ref{thm:PoincUnboundedDoms}. The Poincaré inequalities on these sets allow us to come up with concrete nontrivial example models that satisfy the assumptions of Theorem \ref{theorem:GeneralFractionalCalderon}.

We first recall a recent result on the fractional Poincaré inequalities when $1 \leq p < \infty$ and $0 < s < 1$ in \cite{IGPF21-fracPoincUnbounded,MS21-Best-frac-p-Poincare-unboundedDoms}. They study the Poincaré constant via the variational problem 
\begin{equation}
    P_{n,s,p}(\Omega) \vcentcolon = \inf_{u \in W_\Omega^{s,p}(\R^n),u\neq0} \frac{[u]_{s,p,\R^n}^p}{\norm{u}_{L^p(\Omega)}^p},
\end{equation}
where
\begin{equation}
    [u]_{s,p,\R^n} \vcentcolon = \left(\frac{C(n,s,p)}{2} \int_{\R^n} \int_{\R^n} \frac{\abs{u(x)-u(y)}^p}{\abs{x-y}^{n+sp}}dxdy\right)^{1/p}
\end{equation}
is the (normalized) Gagliardo seminorm, $C(n,s,p) > 0$ is some positive constant and $W^{s,p}_{\Omega}(\R^n)$ is the closure of $C_c^{\infty}(\Omega)$ with respect to the norm $(\|u\|_{L^p(\Omega)}^p+[u]_{s,p,\R^n}^p)^{1/p}$ (see e.g.~\cite{MS21-Best-frac-p-Poincare-unboundedDoms} for details). The associated Poincaré constant is then defined as $C^*_{n,s,p}(\Omega) \vcentcolon = P_{n,s,p}(\Omega)^{-1/p}$. If $C^*_{n,s,p}(\Omega)  > 0$, then it holds that
\begin{equation}
    \norm{u}_{L^p(\R^n)} \leq C^*_{n,s,p}(\Omega) [u]_{s,p,\R^n}.
\end{equation}

We remark that $[u]_{s,2,\R^n} = \norm{(-\Delta)^{s/2}u}_{L^2(\R^n)}$ and $W_\Omega^s(\R^n) = \tilde{H}^s(\Omega)$ in the notation used in this article. The fractional Poincaré inequalities with respect to the Gagliardo seminorms are not equivalent with the Poincaré inequalities for the fractional Laplacian when $p \neq 2$. However, there exist the following embeddings when $0 < s < 1$ (see e.g. \cite[p. 47 Proposition 2 (iii)/(9), pp. 51--52, p. 242 Theorem 2 and Remark 3, and Section 5.2.5]{Triebel-Theory-of-function-spaces} and \cite[Theorems 6.2.4, 6.3.2 and 6.4.4]{Interpolation-spaces} where one identifies $W^{s,p}(\R^n) = B_{pp}^s(\R^n)$ as a Besov space and $H^{s,p}(\R^n) = F_{p2}^s(\R^n)$ as a Triebel--Lizorkin space):
\begin{align}\label{eq:ListOfembeddings}
    W^{s,p}(\R^n) &\hookrightarrow H^{s,p}(\R^n),\quad \norm{u}_{H^{s,p}(\R^n)} \leq C\norm{u}_{W^{s,p}(\R^n)} &(1 < p \leq 2) \\
    H^{s,p}(\R^n) &\hookrightarrow W^{s,p}(\R^n),\quad  \norm{u}_{W^{s,p}(\R^n)} \leq C\norm{u}_{H^{s,p}(\R^n)} &(2 \leq p < \infty)\\
    \dot{W}^{s,p}(\R^n) &\hookrightarrow \dot{H}^{s,p}(\R^n),\quad \norm{(-\Delta)^{s/2}u}_{L^p(\R^n)} \leq C [u]_{s,p,\R^n} &(1 < p \leq 2)\\
    \dot{H}^{s,p}(\R^n) &\hookrightarrow \dot{W}^{s,p}(\R^n),\quad [u]_{s,p,\R^n} \leq C \norm{(-\Delta)^{s/2}u}_{L^p(\R^n)} &(2 \leq p < \infty)\\
     H^{s,p}(\R^n) &\hookrightarrow \dot{H}^{s,p}(\R^n), \quad \norm{(-\Delta)^{s/2}u}_{L^p(\R^n)} \leq C\norm{u}_{H^{s,p}(\R^n)} &(1 < p < \infty, s \geq 0) \\
     \quad W^{s,p}(\R^n) &\hookrightarrow \dot{W}^{s,p}(\R^n), \quad [u]_{s,p,\R^n} \leq C\norm{u}_{W^{s,p}(\R^n)} &(1 < p < \infty, 0 < s < 1)
\end{align}
with some constants $C > 0$ independent of the choice of a function from the left hand side spaces. The last two embeddings are the same as continuity of the associated fractional Laplacian operators, in $H^{s,p}(\R^n)$ the operator does not depend on a particular choice of $p$ and always coincides with the Fourier multiplier definition of $(-\Delta)^{s/2}$.

\begin{theorem}[{\cite[Theorem 1.2]{IGPF21-fracPoincUnbounded}, \cite[Theorem 1.2]{MS21-Best-frac-p-Poincare-unboundedDoms}}]\label{thm:UnboundedLowOrdPoinc} Let $\Omega = \R^{n-k} \times \omega \subset \R^n$ where $n\geq k>0$ and $\omega \subset \R^{k}$ is a bounded open set. Then for any $1 < p < \infty$ and $0 < s  < 1$ it holds that
\begin{equation}
    P_{n,s,p}(\Omega) = P_{k,s,p}(\omega).
\end{equation}
\end{theorem}

We remark that it is well-known that in open bounded sets one has finite Poincaré constants for the Gagliardo seminorms. See also \cite[Theorem 1.1]{AFM19-Asymptotic-FracLap-unbounded} where the result of Theorem \ref{thm:UnboundedLowOrdPoinc} was obtained when $p=2$ by considering asymptotics of the fractional Poincaré constants for domains that are bounded in one direction. The fractional Hardy inequalities on (possibly unbounded) convex sets \cite[Theorem 1.2]{BC18-fractionalHardy} also imply that $P_{n,s,p}(\Omega) > 0$ whenever $\Omega$ is bounded in one direction. See also the related work \cite{yeressian2014asymptotic} proving that $P_{n,s,2}(\Omega) > 0$.

Theorem \ref{theorem: Higher Order Fractional Poincare Inequality} is a generalization of a recent interpolation result for the fractional Poincaré constants of the fractional Laplacians in \cite[Theorem 3.17]{CMR20}. Their result relates the classical Poincaré constant for the gradient in $L^2$ to the Poincaré constants for pairs different order fractional Laplacians. In comparison to the result in \cite{CMR20}, we observe here that a similar result holds also on unbounded domains, without reducing the problem to the classical Poincaré constants, and for any $1 < p < \infty$. The proof of Theorem \ref{theorem: Higher Order Fractional Poincare Inequality} is given in the end of this section.

We can now prove Theorem \ref{thm:PoincUnboundedDoms} using Theorems \ref{thm:UnboundedLowOrdPoinc} and \ref{theorem: Higher Order Fractional Poincare Inequality}. The case when $s \geq 1$ uses the boundedness of Riesz transformation. We do not know a proof which would settle the cases $1 < p < 2$ and $0 < s < 1$. We also give one alternative proof when $p = 2$ as its structure may be relevant for the future attempts to prove the missing special cases.

We first recall a simple lemma due to the absence of a reference.

\begin{lemma}\label{lemma:rigidInvariance} Let $A\colon \R^n\to \R^n$ be a rigid Euclidean motion $A(x) = Lx + x_0$ where $L$ is a linear isometry and $x_0 \in \R^n$. If $u \in C_c^\infty(\R^n)$ and $s \geq 0$, then 
\begin{equation}\label{eq:RigidMotionInvariance}
    (-\Delta)^{s/2}(u \circ A)(x) = ((-\Delta)^{s/2}u\circ A)(x)
\end{equation}
for all $x \in \R^n$. In particular,
$\norm{(-\Delta)^{s/2}u}_{L^p(\R^n)} = \norm{(-\Delta)^{s/2}(u\circ A)}_{L^p(\R^n)}$ for any $p \in [1,\infty]$.
\end{lemma}
\begin{proof} By a simple change of variables it is easy to see that \eqref{eq:RigidMotionInvariance} implies $\norm{(-\Delta)^{s/2}u}_{L^p(\R^n)} = \norm{(-\Delta)^{s/2}(u\circ A)}_{L^p(\R^n)}$ for any $p \in [1,\infty]$. 

On the other hand, the formula \eqref{eq:RigidMotionInvariance} follows by a direct calculation using the symbols and a change of variables. In fact,
\begin{equation}
    \fourier((-\Delta)^{s/2}(u\circ A))(\xi) = \abs{\xi}^s \mathcal{F}(u \circ A)(\xi) = \abs{\xi}^s e^{iL^{-1}(x_0)\cdot \xi}\fourier(u \circ L)(\xi)
\end{equation}
and 
\begin{equation}
    \fourier((-\Delta)^{s/2}u\circ A)(\xi) = e^{iL^{-1}(x_0)\cdot \xi} \int_{\R^n}e^{-ix \cdot \xi}((-\Delta)^{s/2}u)(Lx)dx = e^{iL^{-1}(x_0)\cdot \xi}\abs{L\xi}^s\hat{u}(L\xi).
\end{equation}
Now we only need to notice that $\abs{L\xi} = \abs{\xi}$ and $\fourier(u \circ L)(\xi) = \hat{u}(L\xi)$ hold. This completes the proof.
\end{proof}

\begin{remark} \label{remark: poincare constants and euclidean motion} Lemma \ref{lemma:rigidInvariance} implies that $\Omega$ and $A\Omega$ have the same fractional Poincaré constants. It is clear that if $\Omega$ has a finite fractional Poincaré constant then all open subsets of $\Omega$ also have finite fractional Poincaré constants which are at most the fractional Poincaré constant of $\Omega$. Therefore, without loss of generality, we will often work directly with the cylindrical domains in the proofs.
\end{remark}

\begin{proof}[{Proof of Theorem \ref{thm:PoincUnboundedDoms} when $p \geq 2$}] By Remark~\ref{remark: poincare constants and euclidean motion} it suffices to prove the assertion for cylindrical domains $\Omega =\R^{n-k}\times \omega$, where $\omega\subset\R^k$ is an open bounded set and $n\geq k>0$. Moreover, by Theorem \ref{thm:UnboundedLowOrdPoinc} and the fact that $P_{k,r,p}(\omega) >0$ we know that the assumptions of Theorem \ref{theorem: Higher Order Fractional Poincare Inequality} hold for the parameters $z=0$ and $0<r<1$. This follows rather directly from the embeddings \eqref{eq:ListOfembeddings} between the (in)homogenenous Besov and Triebel spaces when $p \geq 2$.

We are free to choose any $t \geq s \geq r$ to find that
\begin{equation}
    \label{eq: conclusion higher order fractional poincare constant I3}
        \|(-\Delta)^{s/2}u\|_{L^p(\R^n)}\leq C_{r,0}^{\frac{t-s}{r}}\|(-\Delta)^{t/2}u\|_{L^p(\R^n)}
\end{equation}
where $C_{r,0}\vcentcolon = CP_{k,r,p}(\omega)^{-1/p} > 0$, where $C>0$ is the operator norm of the embedding $\dot{H}^{r,p}(\R^n)\hookrightarrow \dot{W}^{r,p}(\R^n)$. Since $r > 0$ can be taken as small as possible, this proves Theorem \ref{thm:PoincUnboundedDoms} for any $t,s > 0$ with $t \geq s$.

On the other hand, when $s=0$, we have that $r > s=z$ and we are free to choose any $t \geq r$ in Theorem \ref{theorem: Higher Order Fractional Poincare Inequality}. This proves Theorem \ref{thm:PoincUnboundedDoms} in the case $t > 0$ and $s=0$.
\end{proof}

\begin{proof}[{Proof of Theorem \ref{thm:PoincUnboundedDoms} when $1<p<2$, $s \geq 1$}] Notice that the boundedness of the Riesz transform gives $\norm{\nabla u}_{L^p(\R^n,\R^n)} \sim \norm{(-\Delta)^{1/2}u}_{L^p(\R^n)}$ for any $1 < p < \infty$ (see \cite[Chapter III, Section 1.2]{Singular-Integrals-Stein}). By the classical Poincaré inequality, we know that there exists some $C > 0$ such that
\begin{equation}
    \norm{u}_{L^p(\Omega)} \leq C\norm{\nabla u}_{L^p(\Omega,\R^n)} \sim \norm{(-\Delta)^{1/2}u}_{L^p(\R^n)}.
\end{equation}
Now the proof continues as the earlier proof of Theorem \ref{thm:PoincUnboundedDoms}. 
\end{proof}

Next we recall and prove the fractional Poincaré inequality on bounded sets. We make a few remarks about it now. One could replace the argument using the Hardy--Littlewood--Sobolev inequality by the embedding theorem of homogeneous Triebel--Lizorkin spaces \cite[Theorem A]{Han95-EmbeddingTheoremForHomogBesovTriebel}. One could replace the interpolation argument by an iteration with integer orders and monotonicity arguments, using certain special inequalities for the critical cases ($tp=n$) \cite[p. 261]{Ozawa} and Hölder estimates for homogeneous convolution kernels in the supcritical cases ($tp>n$) \cite[Theorem 4.5.10]{HO:analysis-of-pdos} (see \cite[Theorem 3.14]{CMR20} for this proof in the simpler special case $p=2$). One can also prove the fractional Poincaré inequality on bounded sets using the compact Sobolev embeddings (see e.g. \cite[Lemma 3.3]{ARS21-OnFractVersionMurat} and the related discussion for the outline) but we are not sure if this directly works for all possible parameters (i.e. in the critical and supcritical cases) without interpolation or iteration. One can also obtain some stronger statements, at least in many special cases, using the maximum principle (see \cite[Proposition 1.4]{rosoton2013extremal} and \cite[Claim 2.8]{Boundary-Regularity-fract-laplacian}).

\begin{lemma}[Fractional Poincaré inequality on bounded sets]\label{lemma:PoincareBoundedsets} Let $\Omega \subset \R^n$ be a bounded open set, $1 < p < \infty$ and $0 \leq s \leq t < \infty$. Then there exists $C(n,p,t,s,\Omega)>0$ such that
\begin{equation}
    \norm{(-\Delta)^{s/2}u}_{L^p(\R^n)} \leq C\norm{(-\Delta)^{t/2}u}_{L^p(\R^n)}
\end{equation}
for all $u \in \tilde{H}^t(\Omega)$.
\end{lemma}
\begin{proof}
Let us first assume that $0 < s < n$ and $1 < p < n/s$. It follows from the Hardy--Littlewood--Sobolev inequality and Hölder's inequality that for any $u \in C_c^\infty(\Omega)$ we have
\begin{equation}
    \norm{u}_{L^p(\Omega)} \leq C'\norm{u}_{L^q(\Omega)} \leq C'C\norm{(-\Delta)^{s/2} u}_{L^p(\R^n)}
\end{equation}
for some $C'(p,q,\Omega), C(n,s,p) >0$ where $q = \frac{np}{n-sp}$. This implies that for any $1 < p < \infty$, there exists some $\epsilon > 0$ such that
\begin{equation}
    \norm{u}_{L^p(\Omega)} \leq C\norm{(-\Delta)^{s/2} u}_{L^p(\R^n)}
\end{equation}
for all $s \in (0,\epsilon)$. It follows now from the interpolation of fractional Poincaré constants (Theorem \ref{theorem: Higher Order Fractional Poincare Inequality}) that there exists $C(n,s,p,\Omega)>0$ such that
\begin{equation}
    \norm{u}_{L^p(\R^n)} \leq C\norm{(-\Delta)^{s/2} u}_{L^p(\R^n)}
\end{equation}
holds for any $s \geq 0$ and $u \in \tilde{H}^s(\Omega)$. This shows the fractional Poincaré inequality on bounded sets.

The rest follows from Theorem \ref{theorem: Higher Order Fractional Poincare Inequality}, or alternatively from the monotonicity of Bessel potential spaces and boundedness of the fractional Laplacian.
\end{proof}

We can now give an alternative proof of Theorem \ref{thm:PoincUnboundedDoms} when $p=2$. We find this to be a quite natural approach for the problem but there is an issue which we do not know how to overcome when $p\neq 2$.

\begin{proof}[Alternative proof of Theorem \ref{thm:PoincUnboundedDoms} when $p=2$]

Suppose for simplicity that $\Omega = \R^{n-1} \times (a,b)$ for some $a < b$ and let us denote points $x$ in $\Omega$ by $x = (x',x_n)$. By Fubini's theorem and the fractional Poincaré inequality for bounded domains we deduce that there exists $C(s,p,a,b)>0$ such that
\begin{equation}
    \norm{u}_{L^p(\Omega)}^p = \int_{\R^{n-1}} \norm{u(x',x_n)}_{L^p(a,b)}^pdx' \leq C\int_{\R^{n-1}} \norm{(-\Delta_{x_n})^{s/2}u(x',x_n)}_{L^p(\R)}^pdx'
\end{equation}
for any $s \geq 0$ and $u \in C_c^\infty(\Omega)$.
We have that 
\begin{equation}
    \int_{\R^{n-1}} \norm{(-\Delta_{x_n})^{s/2}u(x',x_n)}_{L^p(\R)}^pdx' = \norm{(-\Delta_{x_n})^{s/2} u}_{L^p(\R^n)}^p.
\end{equation}
We need an estimate
\begin{equation}
\norm{(-\Delta_{x_n})^{s/2} u}_{L^p(\R^n)}^p \leq C\norm{(-\Delta)^{s/2} u}_{L^p(\R^n)}^p
\end{equation}
for some $C(n,s,p)>0$. This holds with $C=1$ when $p=2$, since the related multiplier 
\begin{equation}\label{eq:PoincareMultiplier}
a(\xi) = \frac{\abs{\xi_n}^s}{\abs{\xi}^s} = \left(\frac{\xi_n^2}{\xi_1^2+\cdots+\xi_n^2}\right)^{s/2},\quad \xi \neq 0,
\end{equation}
takes values in $[0,1]$.

A slight modification of this argument also applies for cylindrical domains $\Omega = \R^{n-k} \times \omega$, which in turn gives the desired result for open sets which are bounded in one direction. In fact, the above proof establishes an estimate for the optimal Poincaré constants, which reads
$C_n(s,\Omega) \leq C_{k}(s,\omega)$. (We know actually that equality holds because of Theorem \ref{thm:UnboundedLowOrdPoinc} but this gives an alternative proof for the upper bound when $p=2$.) 
\end{proof}

\begin{remark} The multiplier \eqref{eq:PoincareMultiplier} is also homogeneous of degree zero but fails to be smooth when restricted to $S^{n-1}$, which is the main obstruction to generalize this natural argument to cover also the cases $p\neq 2$.
\end{remark}

We now move on towards proving Theorem \ref{theorem: Higher Order Fractional Poincare Inequality} but first we recall an interpolation result for the fractional Laplacians in \cite{Interpolation-spaces}. The result is stated also when $p=2$ in \cite[Proposition 1.32]{BCD-fourier-analysis-nonlinear-pde}  but we apply the following more general version in the proof of Theorem \ref{theorem: Higher Order Fractional Poincare Inequality}.

\begin{lemma}[{\cite[Exercise~6.8.31]{Interpolation-spaces}}]\label{lem:GeneralInterpBerghLöf} Let $n \geq 1$, $1 < p_0,p_1 < \infty$, $s_0,s_1 \in \R$ and $0 \leq \theta \leq 1$. If $s_\theta = (1-\theta)s_0 + \theta s_1$ and $\frac{1}{p_\theta}=\frac{1-\theta}{p_0}+\frac{\theta}{p_1}$, then
\begin{equation}
    \norm{(-\Delta)^{s_\theta/2}u}_{L^{p_\theta}(\R^n)} \leq \norm{(-\Delta)^{s_0/2}u}_{L^{p_0}(\R^n)}^{1-\theta}\norm{(-\Delta)^{s_1/2}u}_{L^{p_1}(\R^n)}^{\theta}
\end{equation}
for all $u \in \dot{H}^{s_0,p_0}(\R^n)\cap \dot{H}^{s_1,p_1}(\R^n)$.
\end{lemma}
\begin{proof}[Sketch of the proof]
As this formula is stated without a proof in \cite{Interpolation-spaces}, we give a brief explanation here. See also \cite[Lemma 2.1]{AdamsComposition} for this statement and a brief outline for an alternative proof.

By the complex interpolation method in homogeneous Triebel--Lizorkin spaces $\dot{F}^{s}_{p,q}$ we know 
$[\dot{H}^{s_0,p_0}(\R^n),\dot{H}^{s_1,p_1}(\R^n)]_\theta = \dot{H}^{s_\theta,p_\theta}(\R^n)$ (see \cite[Sections 2.4.7, 5.2.3 and 5.2.5]{Triebel-Theory-of-function-spaces} or \cite[Exercise 6.8.29]{Interpolation-spaces}) and the complex interpolation functor is exact of exponent $\theta \in (0,1)$ (see \cite[Theorem 4.1.2]{Interpolation-spaces}), since the spaces in question are Banach spaces \cite[Section 5.1.5]{Triebel-Theory-of-function-spaces}. The end point cases of Lemma \ref{lem:GeneralInterpBerghLöf} hold trivially. Now the conclusion follows from the definition of exactness of exponent $\theta$ (see \cite[p. 27 (6)]{Interpolation-spaces}) by considering the operator $T\colon\C\to \dot{H}^{s_0,p_0}(\R^n)+\dot{H}^{s_1,p_1}(\R^n)$ with $T(\lambda)=\lambda u$, where $\lambda\in\C$ and $u\in \dot{H}^{s_0,p_0}(\R^n)\cap\dot{H}^{s_1,p_1}(\R^n)$ is fixed.
\end{proof}
We remark that Lemma~\ref{lem:GeneralInterpBerghLöf} in particular holds for all $u\in H^{s_0,p_0}(\R^n)\cap H^{s_1,p_1}(\R^n)$, since $H^{s,p}(\R^n)=L^p(\R^n)\cap\dot{H}^{s,p}(\R^n)$ for all $s>0$, $1< p<\infty$ by \cite[Theorem 6.3.2]{Interpolation-spaces}.

\begin{proof}[Proof of Theorem \ref{theorem: Higher Order Fractional Poincare Inequality}]
We set $p = p_1 = p_2$ in Lemma \ref{lem:GeneralInterpBerghLöf} so that the other condition is automatically satisfied. We assume in the proof that $u\in C_c^{\infty}(\Omega)$. The result follows from this by approximation.

\emph{The first case: $t\geq s\geq r >z\geq 0$.} The case $s=t$ is trivial and so we can assume $t>s$. First of all recall that  $\|u\|_{\dot{H}^{s,p}(\R^n)}=\|(-\Delta)^{s/2}u\|_{L^p(\R^n)}$ for all $u\in \dot{H}^{s,p}(\R^n)$, $s\geq 0$.
    \begin{enumerate}[(i)]
        \item\label{first item case 1 interpolation} If we define $s_0=z,s_1=t,\theta =\frac{r-z}{t-z}\in (0,1)$, then
        \begin{equation}
            s_{\theta}=(1-\theta)s_0+\theta s_1=\frac{t-r}{t-z}z+\frac{r-z}{t-z}t=r.
        \end{equation}
        By Lemma~\ref{lem:GeneralInterpBerghLöf} and \eqref{eq: assumption higher order fractional poincare constant I1} we obtain
        \begin{align}
            \|u\|_{\dot{H}^{r,p}(\R^n)}&\leq \|u\|_{\dot{H}^{z,p}(\R^n)}^{1-\theta}\|u\|_{\dot{H}^{t,p}(\R^n)}^{\theta} \leq C_{r,z}^{1-\theta}\|u\|_{\dot{H}^{r,p}(\R^n)}^{1-\theta}\|u\|_{\dot{H}^{t,p}(\R^n)}^{\theta}.
        \end{align}
        Hence, we obtain
        \begin{equation}
        \label{eq: first interpolation I}
            \|u\|_{\dot{H}^{r,p}(\R^n)}\leq C_{r,z}^{\frac{1-\theta}{\theta}}\|u\|_{\dot{H}^{t,p}(\R^n)}.
        \end{equation}
        \item\label{second item case 1 interpolation} Now, let $s_0=r, s_1=t, \Tilde{\theta}=\frac{s-r}{t-r}\in [0,1)$ and therefore 
        \begin{equation}
            s_{\Tilde{\theta}}=(1-\Tilde{\theta})s_0+\Tilde{\theta}s_1=\frac{t-s}{t-r}r+\frac{s-r}{t-r}t=s.
        \end{equation}
        As before we get the estimate
        \begin{equation}
            \|u\|_{\dot{H}^{s,p}(\R^n)}\leq \|u\|_{\dot{H}^{r,p}(\R^n)}^{1-\Tilde{\theta}}\|u\|_{\dot{H}^{t,p}(\R^n)}^{\Tilde{\theta}}.
        \end{equation}
    \end{enumerate}
    Combining these estimates we get
    \begin{equation}
    \|u\|_{\dot{H}^{s,p}(\R^n)}\leq \|u\|_{\dot{H}^{r,p}(\R^n)}^{1-\Tilde{\theta}}\|u\|_{\dot{H}^{t,p}(\R^n)}^{\Tilde{\theta}}\leq C_{r,z}^{\frac{1-\theta}{\theta}(1-\Tilde{\theta})}\|u\|_{\dot{H}^{t,p}(\R^n)}.
    \end{equation}
    Since, $$\frac{1-\theta}{\theta}(1-\Tilde{\theta})=\frac{\frac{t-r}{t-z}}{\frac{r-z}{t-z}}\frac{t-s}{t-r}=\frac{t-s}{r-z}$$ we have
    \begin{equation}
        \|u\|_{\dot{H}^{s,p}(\R^n)}\leq C_{r,z}^{\frac{t-s}{r-z}}\|u\|_{\dot{H}^{t,p}(\R^n)}
    \end{equation}
    and therefore we have shown \eqref{eq: conclusion higher order fractional poincare constant I2} for $u\in C_c^{\infty}(\Omega)$. It follows from approximation that the same estimate holds in $H^{t,p}(\R^n)$ by the continuity of the fractional Laplacian.

\emph{The second case: $t\geq r\geq s\geq z\geq 0$ and $r>z$.} Again without loss of generality we can assume $t>s$. We can additionally assume that $r > s$ since when $r=s$, we would have that $t > r=s > z$ and this case was already proved in the first part of the proof. This observation is important for the following proof to be valid.
    \begin{enumerate}[(i)]
        \item\label{first item case 2 interpolation} We apply the interpolation result to $s_0=s,s_1=t,\theta =\frac{r-s}{t-s}\in (0,1)$ and
        \begin{equation}
            s_{\theta}=(1-\theta)s_0+\theta s_1=\frac{t-r}{t-s}s+\frac{r-s}{t-s}t=r.
        \end{equation}
         By Lemma~\ref{lem:GeneralInterpBerghLöf} we obtain
        \begin{align}
            \|u\|_{\dot{H}^{r,p}(\R^n)}&\leq \|u\|_{\dot{H}^{s,p}(\R^n)}^{1-\theta}\|u\|_{\dot{H}^{t,p}(\R^n)}^{\theta}.
        \end{align}
        \item\label{second item case 2 interpolation} This time we take $s_0=z, s_1=r, \Tilde{\theta}=\frac{s-z}{r-z}\in [0,1)$ and therefore 
        \begin{equation}
            s_{\Tilde{\theta}}=(1-\Tilde{\theta})s_0+\Tilde{\theta}s_1=\frac{r-s}{r-z}z+\frac{s-z}{r-z}r=s.
        \end{equation}
        Hence, we obtain
        \begin{equation}
        \|u\|_{\dot{H}^{s,p}(\R^n)}\leq \|u\|_{\dot{H}^{z,p}(\R^n)}^{1-\Tilde{\theta}}\|u\|_{\dot{H}^{r,p}(\R^n)}^{\Tilde{\theta}}.
        \end{equation}
    \end{enumerate}
    Combining these estimates we get
    \begin{equation}
        \|u\|_{\dot{H}^{s,p}(\R^n)}\leq \|u\|_{\dot{H}^{z,p}(\R^n)}^{1-\Tilde{\theta}}\|u\|_{\dot{H}^{r,p}(\R^n)}^{\Tilde{\theta}}\leq C_{r,z}^{1-\Tilde{\theta}}\|u\|_{\dot{H}^{r,p}(\R^n)}\leq C_{r,z}^{1-\Tilde{\theta}}\|u\|_{\dot{H}^{s,p}(\R^n)}^{1-\theta}\|u\|_{\dot{H}^{t,p}(\R^n)}^{\theta}.
    \end{equation}
    
    Since, $(1-\Tilde{\theta})/\theta=\frac{r-s}{r-z}/\frac{r-s}{t-s}=\frac{t-s}{r-z}$ we have
    \begin{equation}
        \|u\|_{\dot{H}^{s,p}(\R^n)}\leq C_{r,z}^{\frac{t-s}{r-z}}\|u\|_{\dot{H}^{t,p}(\R^n)}
    \end{equation}
    and we can conclude for $u\in C_c^{\infty}(\Omega)$. The general case again follows by approximation.
\end{proof}

\begin{remark}\label{remark:generalizationsInterpolation} Theorem \ref{theorem: Higher Order Fractional Poincare Inequality} and Lemma \ref{lem:GeneralInterpBerghLöf} have their natural counterparts for the Gagliardo seminorms when $0 < s < 1$. Notice that $[\dot{W}^{s_0,p_0}(\R^n),\dot{W}^{s_1,p_1}(\R^n)]_\theta = \dot{W}^{s_\theta,p_\theta}(\R^n)$ (see again \cite[Sections 2.4.7 and 5.2.5]{Triebel-Theory-of-function-spaces}). The proof of Theorem \ref{theorem: Higher Order Fractional Poincare Inequality} works just as well in this case. In particular, this implies that it is enough to establish the Gagliardo seminorm Poincaré inequality only for all $t \in (0,\epsilon)$ with some $\epsilon >0$ as this already implies the Gagliardo seminorm inequalities for the rest of the cases $\epsilon \leq t < 1$. The argument works on arbitrary subsets of $W^{t,p}(\R^n)$ having some a priori given inequality with respect to the Gagliardo seminorms as the proof only applies generic properties of all functions in the homogeneous spaces, the monotonicity of the inhomogeneous spaces and the boundedness of the seminorms when $t \geq 0$. The same applies for the spaces $H^{t,p}(\R^n)$. Only for the sake of simplicity, we use the spaces $\tilde{H}^{t,p}(\Omega)$ in the statement of Theorem \ref{theorem: Higher Order Fractional Poincare Inequality}. Also the Poincaré--Sobolev type estimates in the mixed integrability cases $(p_1,p_2)$ might be of interest in other applications but not for us. These are quite possible modifications of the argument as the interpolation result is actually stronger than what we use in the proofs (since we set $p=p_1=p_2$).
\end{remark}

Theorem \ref{theorem: Higher Order Fractional Poincare Inequality} makes us to suggest the following conjecture on the equivalence of the fractional Poincaré constants.

\begin{conjecture}[Equivalence of the optimal fractional Poincaré constants]\label{conj:equivalenceOfPoincareConstants} Let $\Omega\subset \R^n$ be an open (bounded) domain and $1 < p < \infty$. If $C_{r,z}$ is the optimal fractional Poincaré constant for $r > z \geq 0$, then $C_{t,s}=C_{r,z}^{\frac{t-s}{r-z}}$ is the optimal Poincaré constant for $t > s \geq 0$.
\end{conjecture}

\section{Decompositions of Sobolev multipliers and PDOs}\label{sec:SobMultDecompositions}

In this section, we define classes of local perturbations to the fractional Schrödinger equations using the Poincaré inequalities. These classes of perturbations will naturally satisfy the assumptions of Theorem \ref{theorem:GeneralFractionalCalderon} in different settings.

\subsection{Basic notation and definitions}
\label{sec:Sobolev-Multiplier-basics}

We define the spaces of Sobolev multipliers between Bessel potential spaces $H^s$ in this section. For more details, we point to the article \cite[Section 2]{CMRU20-higher-order-fracCald}. For $r,t\in \R$, we define the space of multipliers $M(H^r\rightarrow H^t)$ between pairs of Bessel potential spaces according to \cite[Ch. 3]{MS-theory-of-sobolev-multipliers}. Eventually we will study a generalization of the problem \eqref{eq:fractionalschrodingerequation} where the potential $q$ is replaced by a partial differential operator $P(x,D)$ having its coefficients in the spaces of Sobolev multipliers.

If $f\in \distr(\R^n)$ is a distribution, we say that $f\in M(H^r\rightarrow H^t)$ whenever the norm 
$$\|f\|_{r,t} \vcentcolon = \sup \{\abs{\ip{f}{u v}} \,;\, u,v \in C_c^\infty(\mathbb R^n), \norm{u}_{H^r(\R^n)} = \norm{v}_{H^{-t}(\R^n)} =1 \}$$
is finite and we define $M_0(H^r \to H^t)$ to be the closure of $C_c^\infty(\R^n)$ in $M(H^r \to H^t)$. If $f\in M(H^r\rightarrow H^t)$ and $u,v \in C_c^\infty(\mathbb R^n)$ are both nonvanishing, we have the multiplier inequality
\begin{equation}\label{multiplier-inequality}
    \abs{\ip{f}{uv}} \leq \|f\|_{r,t}\norm{u}_{H^r(\R^n)} \norm{v}_{H^{-t}(\R^n)}.
\end{equation}

By density of $C_c^\infty(\R^n) \times C_c^\infty(\R^n)$ in $ H^r(\R^n) \times H^{-t}(\R^n)$, there is a unique continuous extension $(u,v) \mapsto \langle f, uv \rangle$ for $(u,v)\in H^r(\R^n)\times H^{-t}(\R^n)$. More precisely, each $f \in M(H^r\rightarrow H^t)$ gives rise to a linear multiplication map $m_f \colon H^r(\R^n) \rightarrow H^t(\R^n)$ defined by 
\begin{align}
    \langle m_f(u),v \rangle \vcentcolon = \lim_{i \to \infty}\langle f,u_iv_i \rangle \quad \mbox{for all} \quad (u,v)\in H^r(\R^n)\times H^{-t}(\R^n),
\end{align}
where $(u_i,v_i) \in C_c^\infty(\R^n) \times C_c^\infty(\R^n)$ is any sequence in $H^r(\R^n) \times H^{-t}(\R^n)$ converging to $(u,v)$. We can analogously define the unique adjoint multiplication map $m_f^*\colon H^{-t}(\R^n) \to H^{-r}(\R^n)$ such that 
\[\ip{m_f^*(v)}{u} \vcentcolon = \lim_{i \to \infty}\langle f,u_iv_i \rangle \quad \mbox{for all} \quad (u,v)\in H^r(\R^n)\times H^{-t}(\R^n).\] It is easy to verify that the adjoint of $m_f$ is $m_f^*$. We will just write $fu$ for both $m_f(u)$ and $m_f^*(u)$.

\subsection{Decompositions of Sobolev multipliers}

Now we introduce some subspaces of 
Sobolev multipliers $M(H^r \to H^t)$
to be used later in the context of fractional Calderón problems. We define for any open set $\Omega\subset\R^n$ which is bounded in one direction and $s>0$ the quantity 
$\delta(\Omega) \vcentcolon = \left(\frac{2^{-s/2}}{1+C(s,\Omega)}\right)^2$,
where $C(s,\Omega)$ is the optimal Poincaré constant on $\Omega$ (see~Theorem~\ref{thm:PoincUnboundedDoms}). We do not write the dependence of $\delta$ on $s > 0$ explicitly, but it must be noted. We have
\begin{equation}
\begin{split}
    \norm{(-\Delta)^{s/2}u}_{L^2(\R^n)} &\geq t\norm{(-\Delta)^{s/2}u}_{L^2(\R^n)}+\frac{1-t}{C(s,\Omega)}\norm{u}_{L^2(\R^n)}\\
    &=\frac{1}{C(s,\Omega)+1}\left(\norm{(-\Delta)^{s/2}u}_{L^2(\R^n)}+\norm{u}_{L^2(\R^n)}\right)\\
    &\geq \sqrt{\delta(\Omega)}\norm{u}_{H^s(\R^n)},
\end{split}
\end{equation}
if we choose $t = \frac{1}{C(s,\Omega)+1} \in (0,1)$. This calculation was presented here to motivate the definition of suitable classes of perturbations for the fractional Schrödinger equation such that the fractional Laplacian dominates the perturbations in the estimates, especially, from below. This allows to show the existence and uniqueness of weak solutions for the related perturbed fractional Schrödinger equations even when one lacks of compact Sobolev embeddings.

\begin{remark} Let $\Omega \subset\R^n$ be an open set which is bounded in one direction. Then by assumption there exists a rigid Euclidean motion $A$ such that $A\Omega\subset \Omega_{\infty}=\R^{n-k}\times \omega$, where $\omega\subset\R^k$ is a bounded open set and $n\geq k>0$. Using Remark~\ref{remark: poincare constants and euclidean motion} and Theorem~\ref{thm:UnboundedLowOrdPoinc}, we can prove the following estimate (the last inequality is obtained by optimizing the constant in \cite[Theorem 3.7]{CMR20}) \begin{equation}C(s,\Omega)=C(s,A\Omega)\leq C(s,\Omega_{\infty})\leq C(s,\omega) \leq  \frac{(\omega_k\abs{\omega})^{s/k}}{(2\pi)^s}\sqrt{\frac{\delta^{1+\frac{2s}{k}}}{\delta-1}} =\vcentcolon c(s,\Omega),\quad \delta=1+\frac{k}{2s},
\end{equation}
where $\omega_k$ is the volume of the unit ball in $\R^k$. This gives a quantitative lower bound $\delta(\Omega)^2 \geq 2^{-s}(1+c(s,\Omega))^{-1}>0$. This estimate could be used to define more concrete subclasses of Sobolev multipliers and PDOs within the classes which we let to depend directly on the Poincaré constants.
\end{remark}

\begin{definition}[Classes of Sobolev multipliers]
Let $C > 0$, $s>0$ and $\kappa\geq 0$. We define the space of \emph{small multipliers} as
\begin{equation}
    M_C(H^s \to H^{-s}) \vcentcolon = \{\,a \in M(H^s \to H^{-s}) \,;\,\norm{a}_{s,-s} < C\,\}
\end{equation}
and the spaces of \emph{nonnegative multipliers} as
\begin{equation}
    M_{\geq \kappa}(H^s \to H^{-s}) \vcentcolon = \{\,a \in M(H^s \to H^{-s}) \,;\,a(f^2) \geq \kappa\norm{f}_{L^2(\R^n)}^2 \quad \forall f \in C_c^\infty(\R^n)\,\}.
\end{equation}
\end{definition}

We have, for example, that $a \in L^\infty(\R^n)$ with $a \geq \kappa\geq 0$ a.e. in $\R^n$ belongs to $M_{\geq \kappa}(H^s \to H^{-s})$ for any $s >0$.
\begin{remark}
    One can generalize the example models in Proposition~\ref{prop:unboundedExamples} to the settings $q\in M_{\geq \kappa}(H^s\to H^{-s})$ or $\gamma\in M_{\geq \kappa}(L^2\to L^2)$ for any $\kappa>0$.
\end{remark}

\begin{definition}[Multipliers close to $M_0$]
Given a domain $\Omega \subset \R^n$ with a finite Poincaré constant of order $s$ and integrability scale $p=2$. We define the $\delta$-neighbourhood of $M_0$ as follows
\begin{equation}
    M_\Omega(H^s\to H^{-s}) \vcentcolon = \{\, a \in M(H^s\to H^{-s})\,;\, \inf_{a' \in M_0(H^s \to H^{-s})}\norm{a-a'}_{s,-s} < \delta(\Omega)\,\}.
\end{equation}
\end{definition}

It follows from the definitions that
\begin{equation}
\label{eq: decomposition}
    M_0(H^s\to H^{-s}) \subset M_\Omega(H^s\to H^{-s}) = M_0(H^s\to H^{-s}) + M_{\delta(\Omega)}(H^s \to H^{-s}).
\end{equation}
 In particular, if there were a decomposition ''$M = M_0 + M_{\geq0} + M_{\delta(\Omega)}$'' for a domain $\Omega$, then one would not have any restrictions in the class of multipliers on bounded sets in the theorems of Section \ref{sec:fractionalCalderonMainSec}. This is true if one can show that every Sobolev multiplier in $M(H^s \to H^{-s})$ can be approximated by $C_c^\infty(\R^n)$ after changing the multiplier a little bit and possibly removing an arbitrary nonnegative part. See the following remarks for further discussion.

\begin{remark}
The characterization of $M_0(H^r\rightarrow H^t)$ is an open problem and the authors are not even aware if $M_0(H^r\rightarrow H^t)\subsetneq M(H^r\rightarrow H^t)$ holds (see also \cite[Remark 2.5]{RS-fractional-calderon-low-regularity-stability} and the related results in the special cases $t=-r$.). 
\end{remark}

\begin{remark} Let $s > 0$. The following are equivalent:
\begin{enumerate}[(i)]
    \item\label{item 1 Sobolev mult} $M(H^s \to H^{-s}) = M_\Omega(H^s \to H^{-s})$ for all balls, i.e. $\Omega = B_r(0)$, $r > 0$;
    \item\label{item 2 Sobolev mult} $M(H^s \to H^{-s})=M_0(H^s \to H^{-s})$.
\end{enumerate}
This follows since the Poincaré constant on bounded domains grows like $d^s$ when $d = 2r \to \infty$, and thus $\delta(B_r(0)) \to 0$ as $r \to \infty$.

If for some $R \geq 1$ there exists $f \in M(H^s \to H^{-s})$ such that $d(f,M_0(H^s\to H^{-s})) \geq R$, then for all domains $\Omega \subset \R$ it holds that $f \notin M_\Omega(H^s \to H^{-s})$. This follows as in the best possible case $1 > \delta(\Omega) \approx 1$ when $C(s,\Omega) \approx 0$.
\end{remark}

\begin{remark}\label{rmk:multipliersRemark}
In mathematical terms, to have unique weak solutions (modulo the discrete spectrum) for the fractional Schrödinger equation with a potential $q \in M(H^s\to H^{-s})$ studied in Lemma \ref{lemma:schrodingerexistenceofsolutions} of Section \ref{sec:generalizationsCalderon}, it would be sufficient to show that for any $a \in M(H^s\to H^{-s})$ there exists $a' \in M_0(H^s\to H^{-s})$ such that $\norm{a-a'}_{s,-s} < \delta(\Omega)$, which is certainly a weaker condition than $M(H^s\to H^{-s}) = M_0(H^s\to H^{-s})$. In the works \cite{CMR20,RS-fractional-calderon-low-regularity-stability}, one restricted to take the potentials from $M_0(H^{s} \to H^{-s})$ rather than from the larger class $M_\Omega(H^{s} \to H^{-s})$. The benefit of this is of course that the class of perturbations in \cite{CMR20,RS-fractional-calderon-low-regularity-stability} do not depend on the domain $\Omega$.
\end{remark}

\subsection{PDOs with Sobolev multiplier coefficients}

Let $m \in \N_0$. We denote PDOs of order $m$ as \begin{equation}P = \sum_{|\alpha|\leq m} a_\alpha D^\alpha,\end{equation} where the coefficients are chosen appropriately. In order to prove existence and uniqueness of weak solutions on domains that are bounded in one direction, we make a suitable smallness assumption on $P$ due to the lack of compact Sobolev embeddings on unbounded domains $\Omega = \R^m \times \omega$.

More precisely, we first calculate that
\begin{align}\label{eq:PDOestimate}
\sum_{|\alpha|\leq m} |\langle a_\alpha , (D^\alpha v) w \rangle|
\leq \left(\sum_{|\alpha|\leq m}  \|a_\alpha\|_{s-|\alpha|,-s}\right) \|w\|_{H^s(\mathbb R^n)} \|v\|_{H^s(\mathbb R^n)}.
\end{align}
Using this calculation as our guide, we know how to choose the coefficients $a_\alpha$. This motivates the next definition, which gives the smallness condition that guarantees the existence of unique weak solutions on domains that have finite Poincaré constants.

\begin{definition}[Multiplier norms of PDOs] Let $\Omega \subset \R^n$ be open with a finite fractional Poincaré constant of any order $s > 0$.
We define the subspace of small PDOs of order $(m,s) \in \N_0 \times \R_+$ as follows
\begin{equation}\label{eq:smallPDOs}\mathcal{P}_{m,s}(\Omega) \vcentcolon = \{\, P =  \sum_{|\alpha|\leq m} a_\alpha D^\alpha\,;\, \norm{P}_{m,s} < \delta(\Omega),\quad a_{\alpha} \in M(H^{s-|\alpha|}\rightarrow H^{-s})\,\}\end{equation}
where the norm of PDOs is
\begin{equation}
\norm{P}_{m,s} \vcentcolon =  \sum_{|\alpha|\leq m}  \|a_\alpha\|_{s-|\alpha|,-s} = \sum_{|\alpha|\leq m}  \|a_\alpha\|_{s,|\alpha|-s}
\end{equation}
for any $P = \sum_{|\alpha|\leq m} a_\alpha D^\alpha$.
\end{definition}

\begin{remark}\label{rmk:multipliersRemark1}
Note that by Theorem~\ref{thm:PoincUnboundedDoms} the set $\mathcal{P}_{m,s}(\Omega)$ is well-defined if $s>0$ and $\Omega$ is an open set which is bounded in one direction. Moreover, we remark that in the articles \cite{CMRU20-higher-order-fracCald}, where the fractional Calderón problems were studied, one had to restrict to coefficients of PDOs or potentials from $M_0(H^r\rightarrow H^t)$. In our case of unbounded domains, the necessary smallness assumption avoids making this additional assumption and sufficiently small Bessel potential $H^{s,\infty}$ coefficients in \cite[Theorem 1.2]{CMRU20-higher-order-fracCald} belong the class we are using here (see the related estimate \cite[eq.~(22) in the proof of lemma 4.1] {CMRU20-higher-order-fracCald}).
\end{remark}

\section{Fractional Calderón problems on unbounded and bounded domains}\label{sec:fractionalCalderonMainSec}

In this section, we study fractional Calderón problems on domains that are bounded in one direction and begin to employ the tools introduced in the earlier sections. We remark that the arguments in this section follow closely those of \cite{CLR18-frac-cald-drift,CMR20,CMRU20-higher-order-fracCald,GSU20} and the emphasis is to explain the points were the differences occur. In Section \ref{sec:small local lin unbounded}, we briefly recall the basic structure of the argument for the sake of completeness and relate this to the abstract framework of Section \ref{sec:GeneralizedFractionalCalderon}. We note that this approach simplifies and generalizes the original uniqueness proofs in \cite{CLR18-frac-cald-drift,CMRU20-higher-order-fracCald}.

\subsection{Small local linear perturbations}\label{sec:small local lin unbounded}

Consider the problem
\begin{align}\label{problem}
(-\Delta)^s u + \sum_{|\alpha|\leq m} a_\alpha(D^\alpha u) & = F \quad \mbox{in} \;\;\Omega, \\
u & = f  \quad \mbox{in} \;\;\Omega_e 
\end{align}
and the corresponding adjoint problem
\begin{align}\label{adjoint-problem}
(-\Delta)^s u^* + \sum_{|\alpha|\leq m} (-1)^{|\alpha|} D^\alpha (a_\alpha u^*) & = F^* \quad \mbox{in} \;\;\Omega, \\
u^* & = f^*  \quad \mbox{in} \;\;\Omega_e .
\end{align}
Note that if $u, u^*\in H^s(\R^n)$ and $a_\alpha\in M(H^{s-\abs{\alpha}}\rightarrow H^{-s})=M(H^s\rightarrow H^{\abs{\alpha}-s})$, then $a_\alpha(D^\alpha u)\in H^{-s}(\R^n)$ and $D^\alpha(a_\alpha u^*)\in H^{-s}(\R^n)$ like $\fraclaplace u, \fraclaplace u^*\in H^{-s}(\R^n)$.

The problems \eqref{problem} and \eqref{adjoint-problem} are associated with the bilinear forms 
\begin{align}\label{bilinear}
B_P(v,w) \vcentcolon = \langle (-\Delta)^{s/2}v, (-\Delta)^{s/2}w \rangle + \sum_{|\alpha|\leq m} \langle a_\alpha ,(D^\alpha v) w \rangle
\end{align}
and
\begin{align}\label{adjoint-bilinear}
B^*_P(v,w) \vcentcolon = \langle (-\Delta)^{s/2}v, (-\Delta)^{s/2}w \rangle + \sum_{|\alpha|\leq m} \langle a_\alpha, v(D^\alpha w) \rangle,
\end{align}
defined on $v,w \in C^\infty_c(\R^n)$. Notice that definition of the bilinear forms do not explicitly depend on the particular choice of $\Omega$, but their quantitative properties like coercivity and boundedness will depend on $\Omega$. Later the domain $\Omega$ only shows up in the class of test functions $\tilde{H}^s(\Omega)$, the interior data $F,F^* \in (\tilde{H}^s(\Omega))^*$, and the exterior data $u-f,u^*-f^* \in \tilde{H}^s(\Omega)$ of the associated exterior value problem. Next we introduce the used notion of weak solutions of \eqref{problem} and \eqref{adjoint-problem}.

\begin{definition}[Weak solutions]
Let $f,f^* \in H^s(\mathbb R^n)$ and $F, F^* \in (\widetilde{H}^{s}(\Omega))^*$. We say that $u\in H^s(\mathbb R^n)$ is a weak solution to \eqref{problem} when $u-f \in \widetilde H^s(\Omega)$ and $B_P(u,v)=F(v)$ for all $v\in \widetilde H^s(\Omega)$. Similarly, we say that $u^*\in H^s(\mathbb R^n)$ is a weak solution to \eqref{adjoint-problem} when $u^*-f^* \in \widetilde H^s(\Omega)$ and $B^*_P(u^*,v)=F^*(v)$ for all $v\in \widetilde H^s(\Omega)$.
\end{definition}

We next begin the preparation to show the existence and uniqueness of weak solutions in order to finally define the exterior DN maps.

\begin{lemma}[Boundedness of the bilinear forms]\label{boundedness-bilinear}
Let $s \in \mathbb R^+ \setminus \mathbb Z$ and $m\in \mathbb N$ be such that $2s \geq m$, and let $a_{\alpha} \in M(H^{s-|\alpha|}\rightarrow H^{-s})$. Then $B_P$ and $B_P^*$ extend as bounded bilinear forms to $H^s(\mathbb R^n)\times H^s(\mathbb R^n)$.
\end{lemma}
\begin{proof}
We will verify this only for the adjoint bilinear form, since the proof for the original problem is analogous. Let $u, v \in C_c^\infty(\R^n)$. We can estimate
\begin{align}
|B_{P^*}(v,w)| & \leq |\langle (-\Delta)^{s/2}v, (-\Delta)^{s/2}w \rangle | + \sum_{|\alpha|\leq m} |\langle a_\alpha ,  v (D^\alpha w) \rangle| \\
& \leq \|w\|_{H^s(\mathbb R^n)} \|v\|_{H^s(\mathbb R^n)} + \sum_{|\alpha|\leq m}  \|a_\alpha\|_{s,|\alpha|-s} \|v\|_{H^s(\mathbb R^n)}\|D^\alpha w\|_{H^{s-|\alpha|}(\mathbb R^n)} \\ & 
\leq \left( 1+ \sum_{|\alpha|\leq m}  \|a_\alpha\|_{s-|\alpha|,-s}\right) \|v\|_{H^s(\mathbb R^n)} \|w\|_{H^s(\mathbb R^n)}.\\
&= (1+\norm{P}_{m,s})\|v\|_{H^s(\mathbb R^n)} \|w\|_{H^s(\mathbb R^n)}
\end{align}
where in the last step we used $\norm{a}_{r,s} = \norm{a}_{-s,-r}$.
Now the claim follows from the density of $C_c^\infty(\R^n)$ in $H^s(\R^n)$.
\end{proof}

\begin{lemma}[Well-posedness]\label{well-posedness}
Let $\Omega \subset \mathbb R^n$ be an open set which is bounded in one direction, assume that $s \in \mathbb R^+ \setminus \mathbb Z$, $m\in \mathbb N$ satisfy $2s > m$ and let $P \in \mathcal{P}_{m,s}(\Omega)$. Then for any $f\in H^s(\mathbb R^n)$ and $F\in (\widetilde{H}
^s(\Omega))^*$ there exists a unique $u\in H^s(\mathbb R^n)$ such that $u-f \in \widetilde H^s(\Omega)$ and
$$ B_P(u,v) = F(v) \quad \mbox{for all} \quad v \in \widetilde H^s(\Omega).$$
One has the estimate 
$$ \|u\|_{H^s(\mathbb R^n)} \leq C\left( \|f\|_{H^s(\mathbb R^n)} + \|F\|_{(\widetilde{H}^s(\Omega))^*} \right). $$
The function $u$ is the unique $u\in H^s(\mathbb R^n)$ satisfying $$(-\Delta)^su + \sum_{|\alpha|\leq m} a_\alpha D^\alpha u=F$$ in the sense of distributions in $\Omega$ and $u-f \in \widetilde H^s(\Omega)$. Moreover, if $f_1,f_2\in H^s(\R^n)$ satisfy $f_1-f_2\in \tilde{H}^s(\Omega)$ and $u_1,u_2\in H^{s}(\R^n)$ are the unique solutions to 
\begin{align*}
(-\Delta)^s u_j + \sum_{|\alpha|\leq m} a_\alpha(D^\alpha u_j) & = F \quad \mbox{in} \;\;\Omega, \\
u_j & = f_j  \quad \mbox{in} \;\;\Omega_e 
\end{align*}
for $j=1,2$, then $u_1\equiv u_2$ in $\R^n$.
\end{lemma}

\begin{proof}
We get from the definition of $B_P$ that
\begin{align}
B_P(v,v) & \geq \|(-\Delta)^{s/2}v\|^2_{L^2(\R^n)} - \sum_{|\alpha|\leq m} | \langle a_\alpha, (D^\alpha v) v \rangle | \\ 
& \geq  \|(-\Delta)^{s/2}v\|^2_{L^2(\R^n)} -  \norm{P}_{m,s}\norm{v}_{H^s(\R^n)}^2\\
& \geq c\norm{v}_{H^s(\R^n)}^2
\end{align}
where $c = \delta(\Omega)-\norm{P}_{m,s} > 0$. Therefore the assumptions of Lemma~\ref{lemma:generalWellposedness} are valid, and the statements follow.
\end{proof}

We define the exterior DN maps associated to the problems \eqref{problem} and \eqref{adjoint-problem} next. We recall that the abstract trace space is $X \vcentcolon = H^s(\mathbb R^n)/\widetilde H^s(\Omega)$ and for simplicity we denote in the rest of the article its elements by $f$ instead of $[f]$.

\begin{definition}\label{def:exteriorDNmaps}
Let $\Omega \subset \mathbb R^n$ be an open set which is bounded in one direction, assume that $s \in \mathbb R^+ \setminus \mathbb Z$, $m\in \mathbb N$ satisfy $2s > m$ and let $P \in \mathcal{P}_{m,s}(\Omega)$. The exterior DN maps $\Lambda_P$ and $\Lambda_P^*$ are  $$ \Lambda_P \colon X\rightarrow X^* \quad \mbox{defined by}\quad \langle \Lambda_Pf,g \rangle \vcentcolon = B_P(u_f, g)$$
and 
$$ \Lambda^*_P \colon X\rightarrow X^* \quad \mbox{defined by}\quad \langle \Lambda^*_Pf,g \rangle \vcentcolon = B^*_P(u^*_f, g)$$
where $u_f, u^*_f$ are the unique solutions to the equations
 \begin{align*}
(-\Delta)^s u + \sum_{|\alpha|\leq m} a_\alpha D^\alpha u & = 0 \quad \mbox{in} \;\;\Omega,\quad  u - f \in \widetilde H^s(\Omega)  
\end{align*}
and 
\begin{align*}
(-\Delta)^s u^* + \sum_{|\alpha|\leq m} (-1)^{|\alpha|} D^\alpha (a_\alpha  u^*) & = 0 \quad \mbox{in} \;\;\Omega, \quad u^* - f  \in \widetilde H^s(\Omega)
\end{align*}
with $f,g\in H^s(\mathbb R^n)$.
\end{definition}

We state in the next lemma that the exterior DN maps $\Lambda_P$ and $\Lambda_P^*$ are well-defined. The proof is an elementary consequence of Lemma \ref{well-posedness}, and the proof given in \cite[Lemma 3.6]{CMRU20-higher-order-fracCald} holds identically in our setting regardless of having the problem defined in unbounded domains. It also follows from Lemma \ref{lemma:generalWellposedness}.

\begin{lemma}[Exterior DN maps]\label{lemma-DN-maps}
The exterior DN maps $\Lambda_P$ and $\Lambda_P^*$ are well-defined, linear and continuous. One has the identity $\langle \Lambda_Pf,g \rangle = \langle f, \Lambda^*_Pg \rangle$.
\end{lemma}

We note that Lemma \ref{lemma:generalAlessandrini} implies the following.

\begin{lemma}[Alessandrini identity]\label{alex}
Let $\Omega \subset \mathbb R^n$ be an open set which is bounded in one direction, assume that $s \in \mathbb R^+ \setminus \mathbb Z$, $m\in \mathbb N$ satisfy $2s > m$ and let $P_1,P_2 \in \mathcal{P}_{m,s}(\Omega)$. If we denote by $u_1,u_2^* \in H^s(\mathbb R^n)$ the unique weak solutions of
 \begin{align*}
(-\Delta)^s u_1 + \sum_{|\alpha|\leq m} a_{1,\alpha} D^\alpha u_1 & = 0 \quad \mbox{in} \;\;\Omega,\quad  u_1 - f_1 \in \widetilde H^s(\Omega)  
\end{align*}
and 
\begin{align*}
(-\Delta)^s u_2^* + \sum_{|\alpha|\leq m} (-1)^{|\alpha|} D^\alpha (a_{2,\alpha}  u_2^*) & = 0 \quad \mbox{in} \;\;\Omega, \quad u_2^* - f_2  \in \widetilde H^s(\Omega) ,
\end{align*}
where $f_1,f_2\in H^s(\R^n)$, then we have the identity
$$ \langle (\Lambda_{P_1} - \Lambda_{P_2})f_1,f_2 \rangle = \sum_{|\alpha|\leq m} \langle a_{1,\alpha}-a_{2,\alpha}, (D^\alpha u_1) u_2^*\rangle. $$
\end{lemma}

The Runge approximation property follows directly from Theorem \ref{lemma:GeneralRungeApproximation} and the unique continuation property of the higher order fractional Laplacian (see \cite{CMR20,GSU20}).

\begin{lemma}[Runge approximation property]\label{runge}
Let $\Omega \subset \mathbb R^n$ be an open set which is bounded in one direction, assume that $s \in \mathbb R^+ \setminus \mathbb Z$, $m\in \mathbb N$ satisfy $2s > m$ and let $P \in \mathcal{P}_{m,s}(\Omega)$. Let $W \subset \mathbb R^n$ be a nonempty open set such that $\overline W \cap \overline \Omega = \emptyset$. Let \[\mathcal{R} \vcentcolon = \{\, u_f - f\,;\,f\in C^\infty_c(W) \,\},\, \mathcal{R}^* \vcentcolon = \{\, u^*_f - f\,;\, f\in C^\infty_c(W) \,\} \subset \widetilde H^s(\Omega),\] where $u_f, u^*_f\in H^s(\R^n)$ uniquely solve
\begin{align*}
(-\Delta)^s u_f + \sum_{|\alpha|\leq m} a_\alpha D^\alpha u_f & = 0 \quad \mbox{in} \quad\Omega,\quad  u_f = f \quad\mbox{in} \quad\Omega_e
\end{align*}
and
\begin{align*}
(-\Delta)^s u_f^* + \sum_{|\alpha|\leq m} (-1)^{|\alpha|} D^\alpha (a_\alpha  u_f^*) & = 0 \quad \mbox{in} \quad\Omega, \quad  u_f^* = f \quad \mbox{in} \quad\Omega_e,
\end{align*}
respectively. Then $\mathcal R$ and $\mathcal R^*$ are dense in $\widetilde H^s(\Omega)$.
\end{lemma}

We are ready to prove our main result on the fractional Calderón problem. One could follow the proof given in \cite{GSU20} using the Alessandrini identity and the Runge approximation property stated above. We however directly rely on Theorem \ref{theorem:GeneralFractionalCalderon}.

\begin{theorem}[Fractional Calderón problem]
\label{main-theorem-singular}
Let $\Omega \subset \mathbb R^n$ be an open set which is bounded in one direction, assume that $s \in \mathbb R^+ \setminus \mathbb Z$, $m\in \mathbb N$ satisfy $2s > m$ and let $$P_j = \sum_{|\alpha|\leq m} a_{j,\alpha} D^\alpha\in \mathcal{P}_{m,s}(\Omega),$$ for $j=1,2$. Let $W_1, W_2 \subset \Omega_e$ be open sets. If the exterior DN maps for the equations
$((-\Delta)^s + P_j)u = 0$ in $\Omega$, $j=1,2$, satisfy 
$ \Lambda_{P_1}f|_{W_2} = \Lambda_{P_2}f|_{W_2}$
for all $f \in C^\infty_c(W_1)$, then $a_{1,\alpha}|_{\Omega} = a_{2,\alpha}|_{\Omega}$ for all $\alpha$ of order $|\alpha| \leq m$.
\end{theorem}
\begin{proof}[Proof of Theorem \ref{main-theorem-singular}] 
Using Theorem \ref{theorem:GeneralFractionalCalderon} we deduce
\begin{equation}\label{main-2}
 \sum_{|\alpha|\leq m} \langle a_{1,\alpha}-a_{2,\alpha} , (D^\alpha v_1) v_2\rangle = 0 \quad \mbox{for all} \quad v_1,v_2 \in C^\infty_c(\Omega).
\end{equation}
For $\alpha=0$ it immediately follows that $a_{1,\alpha}|_{\Omega}=a_{2,\alpha}|_{\Omega}$. We finish the proof by induction using the identity \eqref{main-2}. Suppose that $a_{1,\beta}|_{\Omega} = a_{2,\beta}|_{\Omega}$ for all $\beta$ such that $|\beta| < N$ for some $N <m$. Let $\alpha\in\N_0^n$ satisfy $|\alpha|=N$, $v_2 \in C^\infty_c(\Omega)$ and choose $v_1\in C^\infty_c(\Omega)$ such that $v_1(x) = x^\alpha = x_1^{\alpha_1}x_2^{\alpha_2}\cdots\,x_n^{\alpha_n}$ on $\supp(v_2) \Subset \Omega$. Now the equation \eqref{main-2} gives, after a short computation, that
$$ 0=\langle a_{1,\alpha}-a_{2,\alpha},(D^\alpha x^\alpha) v_2\rangle = \alpha! \langle a_{1,\alpha}-a_{2,\alpha}, v_2 \rangle $$
which implies $a_{1,\alpha}|_{\Omega} = a_{2,\alpha}|_{\Omega}$. Hence, by induction $a_{1,\alpha}|_{\Omega} = a_{2,\alpha}|_{\Omega}$ for all $\alpha$ of order $|\alpha| \leq m$. This also implies that $P_1|_{\Omega} = P_2|_{\Omega}$.
\end{proof}

\subsection{On generalizations of the uniqueness results for the fractional Calderón problems on bounded and unbounded domains}\label{sec:generalizationsCalderon}
We consider the problem \eqref{problem} and the related inverse problem of Theorem \ref{main-theorem-singular} in the special case when $m=0$, the PDOs $P_1,P_2$ are replaced by the zeroth order potentials $q_1,q_2 \in M(H^s\to H^{-s})$ and consider also the case where $\Omega \subset \R^n$ is a bounded domain. It was shown in \cite{CMR20,RS-fractional-calderon-low-regularity-stability} that the result of Theorem \ref{main-theorem-singular} holds when $m = 0$ and $q_1,q_2 \in M_0(H^s \to H^{-s})$ and $s \in \R_+ \setminus \Z$ on all bounded domains. We extend this result on bounded sets to potentials in $M_\Omega(H^s \to H^{-s}) + M_{\geq 0}(H^s \to H^{-s})$ as this small improvement comes only with a little additional work using the same methods that are needed for unbounded domains. In the cases of unbounded domains, we restrict to potentials in $M_{\delta(\Omega)}(H^s \to H^{-s}) + M_{\geq 0}(H^s \to H^{-s})$. 

The proofs remain almost identical to the proof of Theorem \ref{main-theorem-singular}. It is also possible to obtain similar generalizations on bounded sets in the case of lower order linear perturbations studied in Section \ref{sec:small local lin unbounded}. Rather than giving full details of the proofs in this section, which are very similar to the earlier proofs in \cite{CMR20,CMRU20-higher-order-fracCald,RS-fractional-calderon-low-regularity-stability} and the proof of Theorem \ref{main-theorem-singular}, we only explicitly prove here the related well-posedness results and conclude the uniqueness results for the fractional Calderón problems. The proof of Theorem \ref{main-theorem-singular} with obvious changes can be used as a guide to complete the missing details.

We assume that the potential $q \in M(H^s \to H^{-s})$ is such that $0$ is not the Dirichlet eigenvalue of the operator $\fraclaplace+q$, namely:
\begin{equation}
\label{def:zeronotdirichleteigenvalue}
\text{If} \ u\in H^s(\R^\dimens) \ \text{solves} \ (\fraclaplace+q)u=0 \ \text{in} \ \Omega \ \text{and} \ u|_{\Omega_e}=0, \ \text{then} \ u=0.
\end{equation}
We define the associated bilinear form $B_q\colon H^s(\R^n)\times H^s(\R^n)\to \R$ as
$$
B_q(v, w)=\langle(-\Delta)^{s/2}v, (-\Delta)^{s/2}w\rangle + \ip{q}{vw}
$$
for $v, w\in H^s(\R^\dimens)$.

\begin{lemma}[Well-posedness]
\label{lemma:schrodingerexistenceofsolutions}
Let $\Omega\subset\R^\dimens$ be a bounded open set, $s\in\R^+\setminus\Z$ and $q\in M_\Omega(H^s \to H^{-s}) + M_{\geq 0}(H^s \to H^{-s})$. Then the following statements hold.
\begin{enumerate}[(i)]
    \item There exist a real number $\mu > 0$ and a countable set\label{item:partAsingularPotBoundedSet} $\Sigma\subset (-\mu,\infty)$ of eigenvalues with $\lambda_1 \leq \lambda_2 \leq \dotso\rightarrow \infty$ having the following property: If $\lambda\notin\Sigma$, then for any $f\in H^s(\R^\dimens)$ and $F\in (\widetilde{H}^s(\Omega))^*$ there is a unique $u\in H^s(\R^\dimens)$ satisfying
    $$
    B_q(u, v)-\lambda\langle u,v\rangle=F(v) \ \ \text{for} \ v\in \widetilde{H}^s(\Omega), \quad u-f\in \widetilde{H}^s(\Omega)
    $$
    with the norm estimate
    $$
    \aabs{u}_{H^s(\R^\dimens)}\leq C\Big(\aabs{F}_{(\widetilde{H}^s(\Omega))^*}+\aabs{f}_{H^s(\R^\dimens)}\Big),
    $$
    where $C$ is independent of $F$ and $f$.
    \item \label{item:partBsingularPotBoundedSet}The function $u$ in \ref{item:partAsingularPotBoundedSet} is the unique $u\in H^s(\R^\dimens)$ satisfying
    $$
    \fraclaplace u+qu-\lambda u=F
    $$
    in the sense of distributions on $\Omega$ and $u-f\in\widetilde{H}^s(\Omega)$.
    \item\label{item:partCsingularPotBoundedSet} One has $0\notin\Sigma$ if \eqref{def:zeronotdirichleteigenvalue} holds.
\end{enumerate}
\end{lemma}

\begin{proof}
Using Lemma~\ref{boundedness-bilinear} we see that it is enough to solve the problem with zero exterior condition. By the assumption there is a decomposition
\begin{equation}
    q = q_s + q_0, \quad q_s \in M_\Omega(H^s \to H^{-s}),\quad q_0 \in M_{\geq 0}(H^s \to H^{-s}).
\end{equation}
First, we may calculate
\begin{equation}
\langle q,v^2 \rangle = \langle q_s, v^2\rangle + \langle q_0, v^2\rangle \geq \langle q_s, v^2\rangle
\end{equation}
for any $v \in H^s(\R^n)$.
Using \eqref{eq: decomposition} we have $q_s = q_{s,1} + q_{s,2}$, where $q_{s,1} \in M_0(H^{s}\to H^{-s})$ and $q_{s,2} \in M_{\delta(\Omega)}(H^s \to H^{-s})$. The small part has the estimate
\begin{equation}
    \abs{\langle q_{s,2}, v^2\rangle} \leq \norm{ q_{s,2}}_{s,-s}\norm{v}_{H^s(\R^n)}^2 
    < \delta(\Omega)\norm{v}_{H^s(\R^n)}^2
\end{equation}
for any $v \in H^s(\R^n)$. Next suppose that $v \in \tilde{H}^s(\Omega)$. We may now estimate as in the proof of Lemma \ref{well-posedness} to get
\begin{equation}\label{eq:coercivityEst1}
\begin{split}
B_q(v,v) &\geq \norm{(-\Delta)^{s/2}v}_{L^2(\R^n)}^2 +  \langle q_{s,1}, v^2\rangle + \langle q_{s,2}, v^2\rangle \\
&\geq \norm{(-\Delta)^{s/2}v}_{L^2(\R^n)}^2 + \langle q_{s,1}, v^2\rangle - \abs{\langle q_{s,2}, v^2\rangle}\\
&\geq \left(\delta(\Omega)-\norm{q_{s,2}}_{s,-s}\right)\norm{v}_{H^s(\R^n)}^2 + \langle q_{s,1}, v^2\rangle
\end{split}
\end{equation}
where $c \vcentcolon = \delta(\Omega)-\norm{q_{s,2}}_{s,-s} > 0$. Choose $\phi \in C_c^\infty(\R^n)$ such that $\norm{q_{s,1}-\phi}_{s,-s} < c$ and define $\mu \vcentcolon = \|\phi_{-}\|_{L^{\infty}(\R^n)}\geq 0$, where $\phi_{-}$ denotes the negative part of $\phi$. When we write $q_{s,1} = \phi + (q_{s,1}-\phi)$, we have by \eqref{eq:coercivityEst1} for any $v \in \tilde{H}^s(\Omega)$ that
\begin{equation}
    B_{q,\mu}(v,v) \vcentcolon = B_q(v,v) + \mu\ip{v}{v} \geq c'\norm{v}_{H^s(\R^n)}^2,
\end{equation}
where $c' = c-\norm{q_{s,1}-\phi}_{s,-s} > 0$ is independent of $v$. This coercivity estimate and the boundedness of the bilinear form $B_q$ implies that $B_{q,\mu}$ defines an equivalent inner product on $\tilde{H}^s(\Omega)$.

The proof is now completed as in \cite[Lemma 2.3]{GSU20}. The Riesz representation theorem implies that for every $\widetilde{F}\in (\widetilde{H}^s({\Omega}))^*$ there is unique $u=G_{\mu}\widetilde{F}\in\widetilde{H}^s(\Omega)$ such that $B_q(u, v)+\mu\ip{u}{v}=\widetilde{F}(v)$ for all $v\in\widetilde{H}^s(\Omega)$. Next note that there is a unique $u\in\tilde{H}^s(\Omega)$ such that $B_q(u,v)-\lambda\langle u,v\rangle=\widetilde{F}(v)$ for all $v\in \widetilde{H}^s(\Omega)$ if and only if there holds $u=G_{\mu}[(\mu+\lambda)u+\widetilde{F}]$.  The map $G_{\mu}\colon(\widetilde{H}^s(\Omega))^*\rightarrow\widetilde{H}^s(\Omega)$ induces a compact, self-adjoint and positive definite operator $\widetilde{G}_{\mu}\colon L^2(\Omega)\rightarrow L^2(\Omega)$ by the compact Sobolev embedding theorem. The Fredholm alternative together with the spectral theorem for the self-adjoint, compact, positive definite operator $\widetilde{G}_{\mu}$ implies the first part of the claim in \ref{item:partAsingularPotBoundedSet}. Moreover, the estimate follows from the Riesz representation theorem and the fact that if $\lambda\notin \Sigma$, then solutions $u\in \widetilde{H}^s(\Omega)$ to 
\[
     B_q(u, v)-\lambda\langle u,v\rangle=\tilde{F}(v) \ \ \text{for} \ v\in \widetilde{H}^s(\Omega)
\]
satisfy $\|u\|_{L^2}\leq C\|\widetilde{F}\|_{(\widetilde{H}^s(\Omega))^*}$ (cf.~\cite[Section 6.2, Theorem 6]{Eva10}). 
The statement \ref{item:partBsingularPotBoundedSet} holds since $C_c^{\infty}(\Omega)$ is dense in $\widetilde{H}^s(\Omega)$. The assertion in \ref{item:partCsingularPotBoundedSet} follows from the Fredholm alternative.
\end{proof}

Lemma \ref{lemma:schrodingerexistenceofsolutions} and the same proof as given for Theorem \ref{main-theorem-singular} let us conclude the following theorem. Note that the exterior DN maps are defined according to Definition \ref{def:exteriorDNmaps} and the other properties such as the Alessandrini identity and the Runge approximation property remains valid. See also e.g. \cite{CMR20,GSU20,RS-fractional-calderon-low-regularity-stability} for a simpler proof for the case of just a potential $q$ rather than an operator $P$.

\begin{theorem}[Fractional Calderón problem with $M_\Omega + M_{\geq0}$ potentials]
\label{thm:schrodingeruniqueness}
Let $\Omega\subset\R^\dimens$ be a bounded open set, $s\in\R^+\setminus\Z$, and $q_1, q_2\in M_\Omega(H^s \to H^{-s}) + M_{\geq 0}(H^s \to H^{-s})$ satisfy the condition~\eqref{def:zeronotdirichleteigenvalue}. Let $W_1, W_2\subset\Omega_e$ be open sets. If the exterior DN maps for the equations $\fraclaplace u+q_ju=0$ in $\Omega$, $j=1,2$, satisfy $\Lambda_{q_1}f|_{W_2}=\Lambda_{q_2}f|_{W_2}$ for all $f\in C_c^{\infty}(W_1)$, then $q_1|_{\Omega}=q_2|_{\Omega}$.
\end{theorem}

We get similarly the following theorem on unbounded domains. The proof of well-posedness and uniqueness of solutions is similar to the proof of Lemma \ref{well-posedness} and only relies on the Riesz representation theorem without the spectral theorem. Therefore, we do not have to assume additionally anything about the Dirichlet eigenvalues due to the smallness and positivity assumptions on the potentials.

\begin{theorem}[Fractional Calderón problem with $M_{\delta(\Omega)} + M_{\geq0}$ potentials on unbounded domains]
\label{thm:schrodingeruniquenessUnbounded}
Let $\Omega \subset \mathbb R^n$ be an open set which is bounded in one direction, $s\in\R^+\setminus\Z$, and $q_1, q_2\in M_{\delta(\Omega)}(H^s \to H^{-s}) + M_{\geq 0}(H^s \to H^{-s})$. Let $W_1, W_2\subset\Omega_e$ be open sets. If the exterior DN maps for the equations $\fraclaplace u+q_ju=0$ in $\Omega$, $j=1,2$, satisfy $\Lambda_{q_1}f|_{W_2}=\Lambda_{q_2}f|_{W_2}$ for all $f\in C_c^{\infty}(W_1)$, then $q_1|_{\Omega}=q_2|_{\Omega}$.
\end{theorem}

We finally move on to the case of local linear perturbations on bounded domains. It has a generalization from the usual coefficients in $M_0(H^r \to H^t)$ to more general coefficients, analogously to Theorem \ref{thm:schrodingeruniqueness}. We will state and prove the related well-posedness result as it is less trivial than in any of our earlier cases. Then the uniqueness result for the related fractional Calderón problem can be proved following the proof of Theorem \ref{main-theorem-singular} (given originally in \cite{CMRU20-higher-order-fracCald} in the context of bounded domains but for slightly less general coefficients).

In the uniqueness theorem, we assume that $0$ is not a Dirichlet eigenvalue of the operator $(\fraclaplace+P(x, D))$, namely:
\begin{equation}\label{eq:dirichletEigenvaluePDOs}
\text{If} \ u\in H^s(\R^\dimens) \ \text{solves} \ (\fraclaplace+P(x, D))u=0 \ \text{in} \ \Omega \ \text{and} \ u|_{\Omega_e}=0, \ \text{then} \ u=0.
\end{equation}

\begin{lemma}[Well-posedness]\label{well-posedness-boundedPDO}
Let $\Omega \subset \mathbb R^n$ be a bounded open set, $s \in \mathbb R^+ \setminus \mathbb Z$ and $m\in \mathbb N$ be such that $2s > m$. Define \begin{equation}P = \sum_{\abs{\alpha}\leq m} (a_\alpha+b_\alpha) D^\alpha, \quad a_{\alpha} \in M_0(H^{s-|\alpha|}\rightarrow H^{-s}),\quad P_s \vcentcolon = \sum_{\abs{\alpha}\leq m} b_\alpha D^\alpha \in \mathcal{P}_{m,s}(\Omega).
\end{equation}
Then the following statements hold.
\begin{enumerate}[(i)]
    \item\label{item 1 well-posedness sum} There exist a real number $\mu > 0$ and a countable set $\Sigma \subset (-\mu, \infty)$ of eigenvalues with $\lambda_1 \leq \lambda_2 \leq \dotso\rightarrow \infty$ having the following property: If $\lambda \notin \Sigma$, then for any $f\in H^s(\mathbb R^n)$ and $F\in (\widetilde{H}
^s(\Omega))^*$ there is a unique $u\in H^s(\mathbb R^n)$ satisfying
$$ B_P(u,v)-\lambda \langle u,v \rangle = F(v) \quad \mbox{for all} \quad v \in \widetilde H^s(\Omega),\quad u-f \in \widetilde H^s(\Omega)$$
with the norm estimate
$$ \|u\|_{H^s(\mathbb R^n)} \leq C\left( \|f\|_{H^s(\mathbb R^n)} + \|F\|_{(\widetilde{H}^s(\Omega))^*} \right), $$
where $C$ is independent of $F$ and $f$.
    \item\label{item 2 well-posedness sum} The function $u$ in \ref{item 1 well-posedness sum} is also the unique $u\in H^s(\mathbb R^n)$ satisfying $$( (-\Delta)^s + P)u=F$$ in the sense of distributions in $\Omega$ and $u-f \in \widetilde H^s(\Omega)$. 
    \item\label{item 3 well-posedness sum} Moreover, if \eqref{eq:dirichletEigenvaluePDOs} holds, then $0\notin \Sigma$.
\end{enumerate}
\end{lemma}

\begin{proof} The proof is a combination of the ideas presented already in the proof of Lemmas \ref{well-posedness} and \ref{lemma:schrodingerexistenceofsolutions}, and the proof given in \cite[Lemma 3.4]{CMRU20-higher-order-fracCald}. We describe here how the problem is reduced back to the estimates used in \cite[Lemma 3.4]{CMRU20-higher-order-fracCald}. The crucial part is to establish the coercivity estimate.

Let $v \in \tilde{H}^s(\Omega)$. We get from the definition of $B_P$ that
\begin{align}
B_P(v,v) & \geq \|(-\Delta)^{s/2}v\|^2_{L^2(\R^n)} + \sum_{|\alpha|\leq m} \langle a_\alpha, (D^\alpha v) v \rangle  - \sum_{|\alpha|\leq m} | \langle b_\alpha, (D^\alpha v) v \rangle | \\ 
& \geq  \|(-\Delta)^{s/2}v\|^2_{L^2(\R^n)} + \sum_{|\alpha|\leq m} \langle a_\alpha, (D^\alpha v) v \rangle -  \norm{P_s}_{m,s}\norm{v}_{H^s(\R^n)}^2\\
& \geq c\norm{v}_{H^s(\R^n)}^2 + \sum_{|\alpha|\leq m} \langle a_\alpha, (D^\alpha v) v \rangle
\end{align}
where $c = \delta(\Omega)-\norm{P}_{m,s} > 0$. This now leads to the estimate
\begin{equation}
    B_P(v,v) \geq c\norm{v}_{H^s(\R^n)}^2 - \sum_{|\alpha|\leq m} \abs{\langle a_\alpha, (D^\alpha v) v \rangle}\label{eq:coercivityPartial}
\end{equation}
which is of the form analyzed in \cite[Lemma 3.4]{CMRU20-higher-order-fracCald}.

The role of the additional $c > 0$ is the same as having a larger but still a finite Poincaré constant when compared to the situation in \cite[Lemma 3.4]{CMRU20-higher-order-fracCald}. It follows from \eqref{eq:coercivityPartial} and the more involved estimates calculated in \cite[Lemma 3.4]{CMRU20-higher-order-fracCald} that there exists some $c_0,\mu > 0$ independent of $v$ such that
\begin{equation}\label{coercivity-3}
B_P(v,v) \geq c_0 \|v\|_{H^s(\mathbb R^n)}^2 - \mu\|v\|_{L^2(\mathbb R^n)}^2.
\end{equation}

From here, the proof is completed similarly as in the proof of Lemma \ref{lemma:schrodingerexistenceofsolutions} but this time using the Lax--Milgram theorem and the spectral theorem for the compact, positive definite operators.
\end{proof}

Lemma \ref{well-posedness-boundedPDO} allows to set up the fractional Calderón problem and to prove the following theorem using the exactly same steps as in the proof of Theorem \ref{main-theorem-singular} (or \cite[Section 3]{CMRU20-higher-order-fracCald}).

\begin{theorem}
\label{main-theorem-singular-bounded}
Let $\Omega \subset \mathbb R^n$ be a bounded open set, assume that $s \in \mathbb R^+ \setminus \mathbb Z$, $m\in \mathbb N$ satisfy $2s > m$ and let \begin{equation}P_j = \sum_{\abs{\alpha}\leq m} (a_{j,\alpha}+b_{j,\alpha}) D^\alpha, \quad a_{j,\alpha} \in M_0(H^{s-|\alpha|}\rightarrow H^{-s}),\quad P_{j,s} \vcentcolon = \sum_{\abs{\alpha}\leq m} {b_{j,\alpha}} D^\alpha \in \mathcal{P}_{m,s}(\Omega)
\end{equation}
for $j=1,2$ be such that \eqref{eq:dirichletEigenvaluePDOs} holds. Let $W_1, W_2 \subset \Omega_e$ be open sets. If the exterior DN maps for the equations
$((-\Delta)^s + P_j)u = 0$ in $\Omega$, $j=1,2$, satisfy 
$ \Lambda_{P_1}f|_{W_2} = \Lambda_{P_2}f|_{W_2}$
for all $f \in C^\infty_c(W_1)$, then $(a_{1,\alpha}+b_{1,\alpha})|_\Omega = (a_{2,\alpha}+b_{2,\alpha})|_\Omega$ for all $\abs{\alpha}\leq m$ and $P_1|_{\Omega} = P_2|_{\Omega}$.
\end{theorem}
\begin{remark} We do not have uniqueness for the decomposition $P_j = (P_j-P_{j,s})+P_{j,s}$ but only for the coefficients of the resulting operator.
\end{remark}
\begin{remark} This is a small step towards the unification of the two different approaches in \cite[Theorem 1.1 and Theorem 1.4]{CMRU20-higher-order-fracCald}) since when $\Omega$ is a Lipschitz domain we can now include $H^{s,\infty}(\Omega)$ type coefficients for the small term $P_{j,s}$. The issue is that it seems to be difficult (if even possible) to approximate $H^{s,\infty}(\Omega)$ coefficients with $C_c^\infty(\R^n)$ in the Sobolev multiplier norms. However, the relevant bounded Bessel potential coefficients belong to the spaces of all multipliers on Lipschitz domains as proved in \cite[Lemma 4.1]{CMRU20-higher-order-fracCald}. See also \cite[Propositions 1.2 and 1.3]{CMRU20-higher-order-fracCald} for the inclusion of $H^{s,\infty}(\Omega)$ type coefficients to the spaces $M_0$ when additional vanishing on the boundary is assumed. One can therefore interpret the assumptions of Theorem \ref{main-theorem-singular-bounded} as a smallness assumption on the boundary for the bounded Bessel coefficients in the case of Lipschitz domains.
\end{remark}

\section{Fractional conductivity equation and the Liouville transformation}\label{sec:conductivityCalderon}  In this last section, we will study the fractional conductivity equation and the related inverse problem, which was introduced in \cite{covi2019inverse-frac-cond}, on arbitrary bounded open sets and domains bounded in one direction. It was shown in \cite[Theorem 1.1]{covi2019inverse-frac-cond} that if $\Omega\subset\R^n$ is a bounded Lipschitz domain, $0<s<1$, $0<\gamma_0\leq \gamma_1,\gamma_2\in L^{\infty}(\R^n)$, $m_1\vcentcolon =\gamma_1^{1/2}-1,m_2\vcentcolon =\gamma_2^{1/2}-1\in \Tilde{H}^{2s,\frac{n}{2s}}(\Omega)$ and the associated DN maps restricted to open subsets $W_1,W_2\subset \Omega_e$ agree, then $\gamma_1=\gamma_2$ in $\Omega$. We will not impose any regularity assumption on the boundary of $\Omega$ and weaken the assumption $m_1,m_2\in \Tilde{H}^{2s,\frac{n}{2s}}(\Omega)$ besides generalizing the setting to domains only bounded in one direction. Our proofs are based on certain continuity properties of the multiplication map $(f,g)\to fg$ in Bessel potential spaces, which will be established in Appendix~\ref{subsubsec: Multiplication map in Bessel potential spaces}, and on our results on the generalized Calderón problem from Section~\ref{sec:GeneralizedFractionalCalderon} as well as Section~\ref{sec:generalizationsCalderon}. Our main result in this section is the proof of Theorem \ref{thm: characterization of uniqueness}.

\subsection{Fractional gradient and divergence}
\label{subsubsec: Fractional gradient and divergence}

We start by recalling the notion of fractional gradient and divergence as introduced in \cite{covi2019inverse-frac-cond, NonlocDiffusion}.
    Let $0<s<1$, then the fractional gradient of order $s$ is the bounded linear operator $\nabla^s\colon H^s(\R^n)\to L^2(\R^{2n};\R^n)$ given by
    \[
        \nabla^su(x,y)=\sqrt{\frac{C_{n,s}}{2}}\frac{u(x)-u(y)}{|x-y|^{n/2+s+1}}(x-y),
    \]
    where
    \[
        C_{n,s}=\left(\int_{\R^n}\frac{1-\cos(x_1)}{|x|^{n+2s}}\,dx\right)^{-1}<\infty,
    \]
    and there holds (cf.~\cite[Propositions 3.4 and 3.6]{DINEPV-hitchhiker-sobolev})
    \begin{equation}
    \label{eq: bound on fractional gradient}
        \|\nabla^su\|_{L^2(\R^{2n})}=\sqrt{\frac{C_{n,s}}{2}}[u]_{W^{s,2}(\R^n)}=\|(-\Delta)^{s/2}u\|_{L^2(\R^n)}\leq \|u\|_{H^s(\R^n)}.
    \end{equation}
    Here we have used that the Slobodeckij spaces $W^{s,2}(\R^n)$ coincide with the Bessel potential spaces $H^s(\R^n)$ for $0<s<1$, the continuity of the fractional Laplacian and we will write for simplicity $\|\cdot\|_{L^2(\R^{2n})}$ instead of $\|\cdot\|_{L^2(\R^{2n};\R^n)}$. Therefore, we can define the fractional divergence of order $s$ as the formal adjoint operator of the fractional gradient, that is, it is the map $\Div_s\colon L^2(\R^{2n};\R^n)\to H^{-s}(\R^n)$ with
    \[
        \langle \Div_s(u),v\rangle_{H^{-s}(\R^n)\times H^s(\R^n)}=\langle u,\nabla^sv\rangle_{L^2(\R^{2n})}
    \]
    for all $u\in L^2(\R^{2n};\R^n),v\in H^s(\R^n)$. Moreover, an easy computation shows
    \[
        \|\Div_s(u)\|_{H^{-s}(\R^n)}\leq \|u\|_{L^2(\R^{2n})}
    \]
    for all $u\in L^2(\R^{2n};\R^n)$.
    
    One can show the following relation between these fractional operators and the fractional Laplacian:
    \begin{lemma}[{\cite[Lemma 2.1]{covi2019inverse-frac-cond}}]
    \label{lemma: fractional gradient}
        Let $0<s<1$. Then there holds in the weak sense
        \[
            \Div_s(\nabla^su)=(-\Delta)^su
        \]
        for all $u\in H^s(\R^n)$, that is, we have
        \[
            \langle \nabla^s u,\nabla^s\phi\rangle_{L^2(\R^{2n})}=\langle (-\Delta)^{s/2}u,(-\Delta)^{s/2}\phi\rangle_{L^2(\R^n)}
        \]
        for all $\phi\in H^s(\R^n)$.
    \end{lemma}
    \begin{remark}
        In \cite[Section 6.2]{CMR20} it is shown that this lemma can be generalized to $s\in(0,\infty)\setminus \N$ but the definition of the higher order fractional gradient is rather technical as it involves tensor fields. We have not considered the question to what extent the theory presented in this section extends to the higher order cases.
    \end{remark}

\subsection{Liouville transformation, well-posedness of the fractional conductivity equation and DN map}

First we introduce bilinear forms associated to the fractional conductivity equation and fractional Schr\"odinger equation when the potential $q$ belongs to $L^{n/2s}(\R^n)$.

\begin{lemma}[Definition of bilinear forms and conductivity matrix]\label{prop: bilinear forms conductivity eq}
    Let $\Omega\subset\R^n$ be an open set, $0<s<\min(1,n/2)$, $q\in L^{\frac{n}{2s}}(\R^n)$, $\gamma\in L^{\infty}(\R^n)$ and define the conductivity matrix associated to $\gamma$ by
    \begin{equation}
    \label{eq: conductivity matrix}
        \Theta_{\gamma}\colon \R^{2n}\to \R^{n\times n},\quad \Theta_{\gamma}(x,y)\vcentcolon =\gamma^{1/2}(x)\gamma^{1/2}(y)\mathbf{1}_{n\times n}
    \end{equation}
    for $x,y\in\R^n$. Then the maps defined by
    \begin{equation}
    \label{eq: conductivity bilinear form}
        B_{\gamma}\colon H^s(\R^n)\times H^s(\R^n)\to \R,\quad B_{\gamma}(u,v)\vcentcolon =\int_{\R^{2n}}\Theta_{\gamma}\nabla^su\cdot\nabla^sv\,dxdy
    \end{equation}
    and 
    \begin{equation}
    \label{eq: Schroedinger bilinear form}
        B_q\colon H^s(\R^n)\times H^s(\R^n)\to \R,\quad B_q(u,v)\vcentcolon =\int_{\R^n}(-\Delta)^{s/2}u\,(-\Delta)^{s/2}v\,dx+\int_{\R^n}quv\,dx
    \end{equation}
    are continuous bilinear forms. Moreover, we have $q\in M_0(H^s\to H^{-s})$ with \begin{equation}
    \label{eq: potential as Sobolev multiplier}
        \|q\|_{s,-s}\leq C(1+\|q\|_{L^{\frac{n}{2s}}(\R^n)})
    \end{equation}
    for some $C>0$.
\end{lemma}
\begin{remark}
    If no confusion can arise we will drop the subscript $\gamma$ in the definition for the conductivity matrix $\Theta_{\gamma}$.
\end{remark}

\begin{proof}
    By H\"older's inequality, $\gamma\in L^{\infty}(\R^n)$ and \eqref{eq: bound on fractional gradient} the map $B_{\gamma}$ is well-defined and we have
    \[
        |B_{\gamma}(u,v)|\leq \|\gamma\|_{L^{\infty}(\R^n)}\|\nabla^su\|_{L^2(\R^{2n})}\|\nabla^sv\|_{L^2(\R^{2n})}\leq \|\gamma\|_{L^{\infty}(\R^n)}\|u\|_{H^s(\R^n)}\|v\|_{H^s(\R^n)}
    \]
    for all $u,v\in H^s(\R^n)$. On the other hand, if $q\in L^{\frac{n}{2s}}(\R^n)$, $u\in H^s(\R^n)$, then by Lemma~\ref{lemma: continuity of multiplication map in Bessel potential spaces I V} we know $qu\in L^{\frac{2n}{n+2s}}(\R^n)$. Since $v\in H^s(\R^n)$ we have by the Sobolev embedding $v\in L^{\frac{2n}{n-2s}}(\R^n)$ and therefore we obtain by H\"older's inequality $quv\in L^1(\R^n)$. Using the continuity of the fractional Laplacian we get the estimate
    \[
    \begin{split}
        |B_q(u,v)|&\leq \|(-\Delta)^{s/2}u\|_{L^2(\R^n)}\|(-\Delta)^{s/2}v\|_{L^2(\R^n)}+\|qu\|_{L^{\frac{2n}{n+2s}}(\R^n)}\|v\|_{L^{\frac{2n}{n-2s}}(\R^n)} \\
        &\leq \|u\|_{H^s(\R^n)}\|v\|_{H^s(\R^n)}+C\|qu\|_{L^{\frac{2n}{n+2s}}(\R^n)}\|v\|_{H^s(\R^n)}\\
        &\leq \|u\|_{H^s(\R^n)}\|v\|_{H^s(\R^n)}+C\|q\|_{L^{\frac{n}{2s}}(\R^n)}\|u\|_{H^s(\R^n)}\|v\|_{H^s(\R^n)}\\
        &\leq C(1+\|q\|_{L^{\frac{n}{2s}}(\R^n)})\|u\|_{H^s(\R^n)}\|v\|_{H^s(\R^n)}
    \end{split}
    \]
    for all $u,v\in H^s(\R^n)$. The previous estimate directly shows $q\in M(H^s\to H^{-s})$, but since $C_c^{\infty}(\R^n)$ is dense in $L^{\frac{n}{2s}}(\R^n)$ we can conclude that $q\in M_0(H^s\to H^{-s})$.
\end{proof}

This result now allows us to introduce in both cases a notion of weak solutions.

\begin{definition}[Weak solutions]
    Let $\Omega\subset\R^n$ be an open set, $0<s<\min(1,n/2)$, $q\in L^{\frac{n}{2s}}(\R^n)$ and $\gamma\in L^{\infty}(\R^n)$ with conductivity matrix $\Theta\colon \R^{2n}\to \R^{n\times n}$. If $f\in H^s(\R^n)$ and $F\in (\Tilde{H}^s(\Omega))^*$, then we say that $u\in H^s(\R^n)$ is a weak solution to the fractional conductivity equation
     \[
     \begin{split}
        \Div_s(\Theta\cdot\nabla^s u)&= F\quad\text{in}\quad\Omega,\\
        u&= f\quad\text{in}\quad\Omega_e,
     \end{split}
     \]
    if there holds
    \[
        B_{\gamma}(u,\phi)=F(\phi)\quad\text{and}\quad u-f\in\Tilde{H}^s(\Omega)
    \]
    for all $\phi\in \Tilde{H}^s(\Omega)$. Similarly, we say that $v\in H^s(\R^n)$ is a weak solution to the fractional Schr\"odinger equation
    \[
            \begin{split}
            ((-\Delta)^s+q)v&=F\quad\text{in}\quad\Omega,\\
            v&=f\quad\text{in}\quad\Omega_e,
        \end{split}
    \]
    if there holds
    \[
        B_q(v,\phi)=F(\phi)\quad\text{and}\quad v-f\in\Tilde{H}^s(\Omega)
    \]
    for all $\phi\in \Tilde{H}^s(\Omega)$.
\end{definition}

Now we proof a similar result as 
\cite[Theorem 3.1]{covi2019inverse-frac-cond} 
but which also holds on possibly unbdounded domains $\Omega\subset\R^n$.

\begin{theorem}[Liouville transformation]
\label{theorem: Liouville transformation}
    Let $\Omega\subset \R^n$ be an open set, $0<s<\min(1,n/2)$, $\gamma\in L^{\infty}(\R^n)$ with conductivity matrix $\Theta$ satisfies $\gamma(x)\geq \gamma_0>0$ and define the background deviation $m_{\gamma}\colon\R^n\to\R$ by $m_{\gamma}\vcentcolon =\gamma^{1/2}-1$. Moreover, assume that $m_{\gamma}\in H^{2s,\frac{n}{2s}}(\R^n)$, then the following assertions holds:
    \begin{enumerate}[(i)]
        \item \label{item:Liouville1} Let $u\in H^s(\R^n)$, $g\in H^s(\R^n)/\Tilde{H}^s(\Omega)$ and set $v\vcentcolon =\gamma^{1/2}u,\,f\vcentcolon =\gamma^{1/2}g,\,q_{\gamma}\vcentcolon =-\frac{(-\Delta)^{s}m_{\gamma}}{\gamma^{1/2}}$. Then $v\in H^s(\R^n), f\in H^s(\R^n)/\Tilde{H}^{s}(\Omega)$ and $u$ is a weak solution of the fractional conductivity equation
        \begin{equation}
        \label{eq: fractional conductivity equation}
        \begin{split}
            \Div_s(\Theta\cdot\nabla^s u)&= 0\quad\text{in}\quad\Omega,\\
            u&= g\quad\text{in}\quad\Omega_e
        \end{split}
        \end{equation}
        if and only if $v$ is a weak solution of the fractional Schr\"odinger equation
        \begin{equation}
        \label{eq: fractional Schroedinger equation}    
            \begin{split}
            ((-\Delta)^s+q_{\gamma})v&=0\quad\text{in}\quad\Omega\\
            v&=f\quad\text{in}\quad\Omega_e
        \end{split}
        \end{equation}
        \item \label{item:Liouville2} Let $v\in H^s(\R^n), f\in H^s(\R^n)/\Tilde{H}^s(\Omega)$ and set $u\vcentcolon =\gamma^{-1/2}v,\,g\vcentcolon =\gamma^{-1/2}f$. Then $v$ is a weak solution of \eqref{eq: fractional Schroedinger equation} if and only if $u$ is a weak solution of \eqref{eq: fractional conductivity equation}.
    \end{enumerate}
\end{theorem}

\begin{remark}
    If it is clear to which conductivity $\gamma$ the background deviation $m_{\gamma}$ corresponds, we will denote it by $m$. Moreover, from now on we will always set $q_{\gamma}\vcentcolon =-\frac{(-\Delta)^sm_{\gamma}}{\gamma^{1/2}}$ and refer to it simply as electric potential. As for the background deviation we will drop the subscript $\gamma$ if the dependence is clear from the context.
\end{remark}

\begin{proof}
    Throughout the proof we will write for simplicity $m,q$ instead of $m_{\gamma},q_{\gamma}$.\\
    \ref{item:Liouville1} First note that by Corollary~\ref{cor: continuity of multiplication map in Bessel potential spaces I I} we have $v,f\in H^s(\R^n)$, since $v=\gamma^{1/2}u=mu+u\in H^s(\R^n)$, $f=\gamma^{1/2}g=mg+g\in H^s(\R^n)$. Moreover, if $g-h\in \Tilde{H}^s(\Omega)$ then Corollary~\ref{cor: continuity of multiplication map in Bessel potential spaces I I I} shows that $f-\gamma^{1/2}h\in \Tilde{H}^s(\Omega)$ and hence $f=\gamma^{1/2}g\in H^s(\R^n)/\Tilde{H}^s(\Omega)$ is well-defined. Moreover, by the assumptions $m\in H^{2s,\frac{n}{2s}}(\R^n),\, \gamma\geq \gamma_0>0$ and the mapping properties of the fractional Laplacian we have $q\in L^{\frac{n}{2s}}(\R^n)$. Hence, Lemma~\ref{prop: bilinear forms conductivity eq} shows that the introduced notions of weak solutions to both equations make sense. 
    
    Next we show that there holds 
    \begin{equation}
    \label{eq: equivalence of bilinear forms}
        \langle \Theta\nabla^su,\nabla^s\phi\rangle_{L^2(\R^{2n})}=\langle (-\Delta)^{s/2}(\gamma^{1/2}u),(-\Delta)^{s/2}(\gamma^{1/2}\phi)\rangle_{L^2(\R^n)}+\langle q\gamma^{1/2}u,\gamma^{1/2}\phi\rangle_{L^2(\R^n)}
    \end{equation}
    for all $u,\phi\in H^s(\R^n)$. In the second part of the proof we will demonstrate that if $\phi\in \widetilde{H}^s(\Omega)$ then there holds $\gamma^{-1/2}\phi\in \widetilde{H}^s(\Omega)$ and therefore we can replace in \eqref{eq: equivalence of bilinear forms} the test function $\phi$ by $\gamma^{-1/2}\phi$ which then shows the asserted equivalence in the statement \ref{item:Liouville1}. Fix $u,\phi\in H^s(\R^n)$, then by density of $C_c^{\infty}(\R^n)$ in $H^s(\R^n)$ there exists $u_n,\phi_n\in C_c^{\infty}(\R^n)$ such that $u_n\to u,\phi_n\to \phi$ in $H^s(\R^n)$, but then continuity of the bilinear forms $B_{\gamma},B_q$ (cf.~Lemma~\ref{prop: bilinear forms conductivity eq}) and Lemma~\ref{lemma: continuity of multiplication map in Bessel potential spaces I V} combined with the Sobolev embedding imply that it suffices to show \eqref{eq: equivalence of bilinear forms} for $u,\phi\in C_c^{\infty}(\R^n)$. 
    
    Next let us fix a (radial) standard mollifier $\rho_{\epsilon}\subset C_c^{\infty}(\R^n)$ and introduce the sequences $\gamma_{\epsilon}^{1/2}\vcentcolon =\gamma^{1/2}\ast \rho_{\epsilon}\in C_b^{\infty}(\R^n)$ and $m_{\epsilon}\vcentcolon =m\ast \rho_{\epsilon}\in C_b^{\infty}(\R^n)$. Note that $m_{\epsilon}=\gamma_{\epsilon}^{1/2}-1$, since $\int_{\R^n}\rho_{\epsilon}\,dx=1$. By standard arguments we deduce
    \begin{enumerate}[(I)]
        \item $\gamma_{\epsilon}^{1/2}\overset{\ast}{\rightharpoonup} \gamma^{1/2}$ in $L^{\infty}(\R^n)$,\label{item:LiovilleProof1}
        \item $0<\gamma_0^{1/2}\leq \gamma_{\epsilon}^{1/2}\in L^{\infty}(\R^n)$,\label{item:LiovilleProof2}
        \item $m_{\epsilon}\to m$ in $H^{2s,\frac{n}{2s}}(\R^n)$,\label{item:LiovilleProof3}
        \item $m_{\epsilon}\in L^{\infty}(\R^n)$ with $\|m_{\epsilon}\|_{L^{\infty}(\R^n)}\leq \|m\|_{L^{\infty}(\R^n)}\leq 1+\|\gamma\|_{L^{\infty}(\R^n)}^{1/2}$.\label{item:LiovilleProof4}
    \end{enumerate}
    Let $\Theta_{\epsilon}(x,y)=\gamma_{\epsilon}^{1/2}(x)\gamma_{\epsilon}^{1/2}(y)\mathbf{1}_{n\times n}$, then by using either property \ref{item:LiovilleProof1} or \ref{item:LiovilleProof3} and standard arguments, we obtain
    \[
    \begin{split}
        &\left|\int_{\R^{2n}}\Theta_{\epsilon}u\,dxdy-\int_{\R^{2n}}\Theta u\,dxdy\right|\\
        &=\left|\int_{\R^{2n}}\gamma_{\epsilon}^{1/2}(x)\gamma_{\epsilon}^{1/2}(y) u(x,y)\,dxdy-\int_{\R^{2n}}\gamma^{1/2}(x)\gamma^{1/2}(y) u(x,y)\,dxdy\right|\\
        &\leq \left|\int_{\R^{2n}}\gamma_{\epsilon}^{1/2}(x)(\gamma_{\epsilon}^{1/2}(y)-\gamma^{1/2}(y)) u(x,y)\,dxdy\right|+\left|\int_{\R^{2n}}(\gamma_{\epsilon}^{1/2}(x)-\gamma^{1/2}(x))\gamma^{1/2}(y) u(x,y)\,dxdy\right|\\
       &\to 0
    \end{split}
    \]
    as $\epsilon\to 0$ for all $u\in L^1(\R^{2n})$. 
   Therefore we have
    \[
        \langle\Theta_{\epsilon}\nabla^su,\nabla^sv\rangle_{L^2(\R^{2n})}\to \langle \Theta\nabla^su,\nabla^sv\rangle_{L^{2}(\R^{2n})}
    \]
    as $\epsilon\to 0$ for all $u,v\in H^s(\R^n)$. Next let us define $q_{\epsilon}\vcentcolon =-\frac{(-\Delta)^sm_{\epsilon}}{\gamma_{\epsilon}^{1/2}}\in C^{\infty}(\R^n)$. Since the fractional Laplacian is continuous, we deduce from \ref{item:LiovilleProof3}, \ref{item:LiovilleProof4} and Corollary~\ref{cor: continuity of multiplication map in Bessel potential spaces I I} that
    \[
        \langle (-\Delta)^{s/2}(\gamma^{1/2}_{\epsilon}u),(-\Delta)^{s/2}(\gamma^{1/2}_{\epsilon}\phi)\rangle_{L^2(\R^n)}\to \langle (-\Delta)^{s/2}(\gamma^{1/2}u),(-\Delta)^{s/2}(\gamma^{1/2}\phi)\rangle_{L^2(\R^n)}
    \]
    as $\epsilon\to 0$. Similarly, using the continuity of the fractional Laplacian, the property \ref{item:LiovilleProof3}, Lemma~\ref{lemma: continuity of multiplication map in Bessel potential spaces I V} and the Sobolev embedding we deduce
    \[
    \begin{split}
       & \langle q_{\epsilon}\gamma_{\epsilon}^{1/2}u,\gamma_{\epsilon}^{1/2}\phi\rangle_{L^2(\R^n)}=\langle \gamma^{1/2}_{\epsilon}u,q_{\epsilon}\gamma_{\epsilon}^{1/2}\phi\rangle_{L^2(\R^n)}=-\langle \gamma_{\epsilon}^{1/2}u,((-\Delta)^{s/2}m_{\epsilon})\phi\rangle_{L^2(\R^n)}\\
       &\to -\langle \gamma^{1/2}u,((-\Delta)^{s/2}m)\phi\rangle_{L^2(\R^n)}=\langle q\gamma^{1/2}u,\gamma^{1/2}\phi\rangle_{L^2(\R^n)}
    \end{split}
    \]
as $\epsilon\to 0$. Therefore, if we can show that \eqref{eq: equivalence of bilinear forms} holds for all $u,\phi\in C_c^{\infty}(\R^n)$, $0<\gamma_0^{1/2}\leq \gamma^{1/2}\in C_b^{\infty}(\R^n)$ and $m\in C^t(\R^n)$ for sufficiently large $t>0$ (see Remark~\ref{rem: conv in zygmund}), then we can conclude the proof of statement \ref{item:Liouville1}. By a change of variables we have
\[
    \begin{split}
        \langle\Theta\nabla^su,\nabla^s\phi\rangle_{L^2(\R^{2n})}&=\frac{C_{n,s}}{2}\int_{\R^{2n}}\gamma^{1/2}(x)\gamma^{1/2}(y)\frac{(u(y)-u(x))(\phi(y)-\phi(x))}{|x-y|^{n+2s}}\,dxdy\\
        &=\frac{C_{n,s}}{4}\left(\int_{\R^{2n}}\gamma^{1/2}(x)\gamma^{1/2}(x+z)\frac{(u(x+z)-u(x))(\phi(x+z)-\phi(x))}{|z|^{n+2s}}\,dzdx\right.\\
        &\quad \left.+\int_{\R^{2n}}\gamma^{1/2}(x)\gamma^{1/2}(x-z)\frac{(u(x-z)-u(x))(\phi(x-z)-\phi(x))}{|z|^{n+2s}}\,dzdx\right)\\
        &=\vcentcolon\frac{C_{n,s}}{4}\int_{\R^{2n}}I(x,z)\,dzdx.
    \end{split}
\]
Using the notation 
    \[
        \delta\psi(x,y)\vcentcolon =\psi(x+y)+\psi(x-y)-2\psi(x)
    \]
for any function $\psi\colon\R^n\to\R$ and $x,y\in\R^n$, we can write the integrand as\allowdisplaybreaks
\[
\begin{split}
    I(x,z)=&\frac{\gamma^{1/2}(x)}{|z|^{n+2s}}\left[-\phi(x)(\gamma^{1/2}(x+z)u(x+z)+\gamma^{1/2}(x-z)u(x-z)\right.\\
    &-\gamma^{1/2}(x+z)u(x)-\gamma^{1/2}(x-z)u(x))\\
    &-u(x)(\gamma^{1/2}(x+z)\phi(x+z)+\gamma^{1/2}(x-z)\phi(x-z))\\
    &\left.+(\gamma^{1/2}u\phi)(x+z)+(\gamma^{1/2}u\phi)(x-z)\right]\\
    =&\frac{\gamma^{1/2}(x)}{|z|^{n+2s}}\left[-\phi(x)\delta(\gamma^{1/2}u)(x,z)+u(x)\phi(x)\delta(\gamma^{1/2})(x,z)\right.\\
    &\left.-u(x)\delta(\gamma^{1/2}\phi)(x,z)+\delta(\gamma^{1/2}u\phi)(x,z)\right]\\
    =&\frac{\gamma^{1/2}(x)}{|z|^{n+2s}}\left[-\phi(x)\delta(\gamma^{1/2}u)(x,z)+u(x)\phi(x)\delta(m)(x,z)\right.\\
    &\left.-u(x)\delta(\gamma^{1/2}\phi)(x,z)+\delta(\gamma^{1/2}u\phi)(x,z)\right],
\end{split}
\]
where in the second equality sign we added and subtracted the term $2(\gamma^{1/2}u)(x)$ in the first bracket, $2(\gamma^{1/2}\phi)(x)$ in the second bracket and in the last equality sign we used $\delta(1)=0$. Next recall that the fractional Laplacian of order $s\in (0,1)$ of a function $u\in\schwartz(\R^n)$ can be calculated by (cf.~\cite[Proposition 3.3]{DINEPV-hitchhiker-sobolev})
\begin{equation}\label{eq:fracLapDeltaDef}
(-\Delta)^su(x)=-\frac{C_{n,s}}{2}\int_{\R^n}\frac{\delta u(x,y)}{|y|^{n+2s}}\,dy\quad\text{for all}\quad x\in\R^n.
\end{equation}
Moreover, the same formula holds for $m$ by Remark~\ref{rem: conv in zygmund}.
Therefore, we obtain
\[
    \begin{split}
        \langle\Theta\nabla^su,\nabla^s\phi\rangle_{L^2(\R^{2n})}=&\frac{1}{2}\int_{\R^n}\gamma^{1/2}(x)\left[\phi(x)(-\Delta)^s(\gamma^{1/2}u)(x)+u(x)(-\Delta)^s(\gamma^{1/2}\phi)(x)\right.\\
        &\quad\left.-(-\Delta)^{s}(\gamma^{1/2}u\phi)(x)-u(x)\phi(x)(-\Delta)^sm(x)\right]\,dx.
    \end{split}
\]
The integral $\int_{\R^n}(-\Delta)^su(x)\,dx$ vanishes for any $u\in\schwartz(\R^n)$ and $0<s<1$, since it is the Fourier transform at the origin of the fractional Laplacian. Therefore we obtain
    \[
    \begin{split}
    \allowdisplaybreaks
        \langle\Theta\nabla^su,\nabla^s\phi\rangle_{L^2(\R^{2n})}=&\frac{1}{2}\int_{\R^n}\left[\gamma^{1/2}(x)\phi(x)(-\Delta)^s(\gamma^{1/2}u)(x)+\gamma^{1/2}(x)u(x)(-\Delta)^s(\gamma^{1/2}\phi)(x)\right.\\
        &\quad\left.-m(x)(-\Delta)^{s}(\gamma^{1/2}u\phi)(x)-\gamma^{1/2}(x)u(x)\phi(x)(-\Delta)^sm(x)\right]\,dx\\
        =&\frac{1}{2}\left[\langle (-\Delta)^s(\gamma^{1/2}u),\gamma^{1/2}\phi \rangle_{L^2(\R^n)}+\langle \gamma^{1/2}u, (-\Delta)^s(\gamma^{1/2}\phi)\rangle_{L^2(\R^n)}\right.\\
        &\left.-\langle m,(-\Delta)^s(\gamma^{1/2}u\phi) \rangle_{L^2(\R^n)}-\langle \phi,((-\Delta)^sm)\gamma^{1/2}u \rangle_{L^2(\R^n)}\right]\\
        =&\frac{1}{2}\left[\langle (-\Delta)^s(\gamma^{1/2}u),\gamma^{1/2}\phi \rangle_{L^2(\R^n)}+\langle (-\Delta)^{s}(\gamma^{1/2}u), \gamma^{1/2}\phi\rangle_{L^2(\R^n)}\right.\\
        &\left.-\langle (-\Delta)^{s}m,\gamma^{1/2}u\phi \rangle_{L^2(\R^n)}-\langle ((-\Delta)^sm),\gamma^{1/2}u\phi \rangle_{L^2(\R^n)}\right]\\
        =&\langle (-\Delta)^{s/2}(\gamma^{1/2}u), (-\Delta)^{s/2}(\gamma^{1/2}\phi)\rangle_{L^2(\R^n)}+\langle q\gamma^{1/2}u,\gamma^{1/2}\phi \rangle_{L^2(\R^n)},
    \end{split}
\]
where we used $\langle (-\Delta)^su,v\rangle_{L^2(\R^n)}=\langle u,(-\Delta)^sv\rangle_{L^2(\R^n)}=\langle (-\Delta)^{s/2}u,(-\Delta)^{s/2}v\rangle_{L^2(\R^n)}$ for $u,v\in C_c^{\infty}(\R^n)$ and the same formula holds if one replaces $u$ by $m$ as $m\in H^{2s,\frac{n}{2s}}(\R^n)$.

\noindent\ref{item:Liouville2} Note that we can write 
\[
    \frac{1}{\gamma^{1/2}}=1-\frac{m}{m+1}
\]
and set $\Gamma_0\vcentcolon =\min(0,\gamma_0^{1/2}-1)$. Let $\Gamma\in C^2_b(\R)$ satisfy $\Gamma(t)=\frac{t}{t+1}$ for $t\geq \Gamma_0$. By \cite[p.~156]{AdamsComposition} and $m\in H^{2s,\frac{n}{2s}}(\R^n)\cap L^{\infty}(\R^n)$, we deduce $\Gamma(m)\in H^{2s,\frac{n}{2s}}(\R^n)$, but since $m\geq \gamma_0^{1/2}-1$ it follows that $\frac{m}{m+1}\in H^{2s,\frac{n}{2s}}(\R^n)\cap L^{\infty}(\R^n)$. Therefore by Corollary~\ref{cor: continuity of multiplication map in Bessel potential spaces I I} we deduce $u=\gamma^{-1/2}v\in H^{s}(\R^n)$ and by Corollary~\ref{cor: continuity of multiplication map in Bessel potential spaces I I I} that $f=\gamma^{-1/2}g\in H^s(\R^n)/\Tilde{H}^s(\Omega)$. The rest of the statement follows from \eqref{eq: equivalence of bilinear forms} and we can conclude the proof.
\end{proof}

\begin{remark}
\label{remark: identity of PDOs}
    In fact, we have shown in the proof of Theorem~\ref{theorem: Liouville transformation} that there holds
    \begin{equation}
    \label{eq: identity schroedinger operator and conductivity operator}
        \langle \Theta\nabla^su,\nabla^s\phi\rangle_{L^2(\R^{2n})}=\langle (-\Delta)^{s/2}(\gamma^{1/2}u),(-\Delta)^{s/2}(\gamma^{1/2}\phi)\rangle_{L^2(\R^n)}+\langle q\gamma^{1/2}u,\gamma^{1/2}\phi\rangle_{L^2(\R^n)}
    \end{equation}
    for all $u,\phi\in H^s(\R^n)$.
\end{remark}

\begin{remark}\label{rem: conv in zygmund}
If $f\in L^{\infty}(\R^n)$ and $\rho_{\epsilon}\in C_c^{\infty}(\R^n)$, then $f_{\epsilon}=f\ast \rho_{\epsilon}$ satisfies $\partial^{\alpha}f_{\epsilon}=f\ast (\partial^{\alpha}\rho_{\epsilon})$ for any $\alpha\in \N^n_0$ and is bounded and uniformly continuous. If the first and second order differences are denoted by
\[
    \Delta^1_hf(x)=f(x+h)-f(x),\quad \Delta^2_hf(x)=f(x+2h)-2f(x+h)+f(x),
\]
then
\[
    \Delta^2_h(f\ast \rho_{\epsilon})=(\Delta^1_hf)\ast (\Delta^1_h\rho_{\epsilon}).
\]
Therefore we have
\[
    \|\Delta^2_h\partial^{\alpha}f_{\epsilon}\|_{L^{\infty}}\leq \|(\Delta^1_hf)\|_{L^{\infty}}\|(\Delta^1_h\partial^{\alpha}\rho_{\epsilon})\|_{L^1}\leq C \|f\|_{L^{\infty}(\R^n)} \|\nabla^{k+1}\rho_{\epsilon}\|_{L^1(\R^n)}|h|
\]  
for all $\alpha\in\N^n_0$ with $|\alpha|\leq k$. This shows that $f_{\epsilon}$ belongs to any H\"older space $C^{t}(\R^n)$ with $t\in (0,\infty)\setminus\N$ (see e.g.~\cite[Section 2.5.7]{Triebel-Theory-of-function-spaces}, \cite[Section 1.2.2]{Triebel-Theory-of-functions-spaces-I-I}).
Moreover, note that
\[
\begin{split}
    \int_{\R^n}\frac{1}{1+|x|^{n+2s}}\,dx&=\omega_n\int_0^{\infty}\frac{r^{n-1}}{1+r^{\frac{n+2s}{2}}}\,dr<\infty
\end{split}
\]
for any $s>0$ and therefore $f_{\epsilon}\in L_s$, where $L_s$ consists of all $g\in L^1_{loc}(\R^n)$ such that
\[
    \int_{\R^n}\frac{|g(x)|}{1+|x|^{n+2s}}\,dx<\infty.
\]
Hence, for any $f\in L^{\infty}(\R^n)$ the functions $f_{\epsilon}$ satisfy the assumptions of \cite[Proposition 2.1.4]{SilvestreFracObstaclePhd} and therefore the proof of \cite[Lemma 3.2]{DINEPV-hitchhiker-sobolev} shows that there holds
\[
    (-\Delta)^{s}f_{\epsilon}(x)=-\frac{C_{n,s}}{2}\int_{\R^n}\frac{\delta f_{\epsilon}(x,y)}{|y|^{n+2s}}\,dy
\]
for all $0<s<1$ and $x\in\R^n$.
\end{remark}

Next we show that the fractional conductivity equation and the fractional Schr\"odinger equation are well-posed when the potential comes from a conductivity.

\begin{lemma}[Well-posedness and DN maps]
\label{eq: well-posedness results and DN maps}
    Let $\Omega\subset \R^n$ be an open set which is bounded in one direction and $0<s<\min(1,n/2)$. Assume that $\gamma\in L^{\infty}(\R^n)$ with conductivity matrix $\Theta$, background deviation $m$ and electric potential $q$ satisfies $\gamma(x)\geq \gamma_0>0$ and $m\in H^{2s,\frac{n}{2s}}(\R^n)$. Then the following assertions hold:
    \begin{enumerate}[(i)]
        \item\label{item 1 well-posedness cond eq} For all $f\in X\vcentcolon = H^s(\R^n)/\Tilde{H}^s(\Omega)$ there are unique weak solutions $u_f,v_f\in H^s(\R^n)$ of the fractional conductivity equation
        \begin{equation}
        \begin{split}
            \Div_s(\Theta\cdot\nabla^s u)&= 0\quad\text{in}\quad\Omega,\\
            u&= f\quad\text{in}\quad\Omega_e
        \end{split}
        \end{equation}
        and the the fractional Schr\"odinger equation
        \begin{equation}
            \begin{split}
            ((-\Delta)^s+q)v&=0\quad\text{in}\quad\Omega,\\
            v&=f\quad\text{in}\quad\Omega_e.
        \end{split}
        \end{equation}
        \item\label{item 2 well-posedness cond eq} The exterior DN maps $\Lambda_{\gamma}\colon X\to X^*$, $\Lambda_q\colon X\to X^*$ given by 
        \[
        \begin{split}
            \langle \Lambda_{\gamma}f,g\rangle \vcentcolon =B_{\gamma}(u_f,g),\quad \langle \Lambda_qf,g\rangle\vcentcolon =B_q(v_f,g),
        \end{split}
        \]
        where $u_f,v_f\in H^s(\R^n)$ are the unique solutions to the conductivity equation and Schr\"odinger equation with exterior value $f$, are well-defined bounded linear maps. 
    \end{enumerate}
\end{lemma}

\begin{remark}
    Note that we are integrating over all of $\R^n$ in the term involving the potential $q$ in our definition of the bilinear form for the fractional Schr\"odinger equation, whereas in \cite{covi2019inverse-frac-cond} the integral is restricted to $\Omega$. 
\end{remark}

\begin{proof}
The bilinear form $B_{\gamma}$ is continuous by Lemma~\ref{prop: bilinear forms conductivity eq}. Moreover, we have by the properties of the fractional gradient and the fractional Poincar\'e inequality on domains bounded in one direction (cf.~Theorem~\ref{thm:PoincUnboundedDoms}) 
    \[
    \begin{split}
    B_{\gamma}(u,u)&\geq \gamma_0\|\nabla^su\|_{L^2(\R^{2n})}^2=\gamma_0\|(-\Delta)^{s/2}u\|_{L^2(\R^n)}^2\\
    &\geq \frac{\gamma_0}{2}(\|(-\Delta)^{s/2}u\|_{L^2(\R^n)}^2+C\|u\|_{L^2(\R^n)}^2)\geq \frac{\gamma_0}{2}\min(1,C)\|u\|_{H^s(\R^n)}^2
    \end{split}
    \]
    for all $u\in \tilde{H}^s(\Omega)$ and therefore $B_{\gamma}$ is (strongly) coercive. Now by Lemma~\ref{lemma:generalWellposedness} the assertions for the conductivity equation follow. 
    
    Next consider the Schr\"odinger equation with exterior value $f\in X$. By part \ref{item:Liouville2} of Theorem~\ref{theorem: Liouville transformation} we have $g\vcentcolon =\gamma^{-1/2}f \in X$ and we already know that there is a unique weak solution $v_g\in H^s(\R^n)$ to the conductivity equation with exterior value $g$. Therefore, part \ref{item:Liouville1} of Theorem~\ref{theorem: Liouville transformation} demonstrates that $u_f\vcentcolon =\gamma^{1/2}v_g\in H^s(\R^n)$ has exterior value $f$ and solves the Schr\"odinger equation as desired. If $u_1,u_2\in H^s(\R^n)$ are weak solutions to the Schr\"odinger equation with exterior values $f\in X$, then $u\vcentcolon =u_1-u_2\in \Tilde{H}^s(\Omega)$ solves again the Schr\"odinger equation but this time with exterior value zero and hence by Theorem~\ref{theorem: Liouville transformation}, part \ref{item:Liouville2} the function $v\vcentcolon =\gamma^{-1/2}u\in \Tilde{H}^s(\Omega)$ is a weak solution to 
    \[
    \begin{split}
            \Div_s(\Theta\cdot\nabla^s v)&= 0\quad\text{in}\quad\Omega,\\
            v&= 0\quad\text{in}\quad\Omega_e.
        \end{split}
    \]
    Since solutions to the conductivity equation are unique we obtain $v\equiv 0$ which in turn implies $u_1=u_2$. Therefore solutions to the Schr\"odinger equation are unique as well. The statements for the DN map associated to the Schr\"odinger equation $\Lambda_q$ follow from Lemma~\ref{prop: bilinear forms conductivity eq} and standard arguments.
\end{proof}

\begin{lemma}
\label{eq: Comparison of DN maps}
    Let $\Omega\subset \R^n$ be an open set which is bounded in one direction and $0<s<\min(1,n/2)$. Assume that $\gamma\in L^{\infty}(\R^n)$ with conductivity matrix $\Theta$, background deviation $m$ and electric potential $q$ satisfies $\gamma(x)\geq \gamma_0>0$ and $m\in H^{2s,\frac{n}{2s}}(\R^n)$. If $f,g\in H^s(\R^n)$ satisfy
    \[
        (\supp(f)\cup\supp(g))\cap\supp(m)=\emptyset,
    \]
    then there holds $\langle\Lambda_{\gamma}f,g\rangle=\langle\Lambda_qf,g\rangle.$
\end{lemma}

\begin{proof}
    If $\supp(m)=\R^n$, there is nothing to prove and hence we can assume $\supp(m)\neq \R^n$. By assumption we have 
    \[
        \gamma^{1/2}h=(m+1)h=mh+h=h
    \]
    for $h=f,g$. By Theorem~\ref{theorem: Liouville transformation} we have $\gamma^{1/2}u_f=v_{\gamma^{1/2}f}=v_f$,
    where $u_f,v_f\in H^s(\R^n)$ are the unique solutions to the homogenous fractional conductivity equation and Schr\"odinger equation with exterior value $f\in H^s(\R^n)$, respectively. Hence, by Remark~\ref{remark: identity of PDOs} and Lemma~\ref{eq: well-posedness results and DN maps} we obtain 
    \[
    \begin{split}
        \langle \Lambda_{\gamma}f,g\rangle &= B_{\gamma}(u_f,g)=B_q(\gamma^{1/2}u_f,\gamma^{1/2}g)=B_q(v_f,g)=\langle\Lambda_{q}f,g\rangle.
    \end{split}
    \]
\end{proof}

\subsection{Uniqueness results for the inverse problem}\label{sec:uniquenessFracCaldSubsection}

\begin{lemma}
\label{thm: uniqueness conductivity equation}
    Let $\Omega\subset \R^n$ be an open set and $0<s<\min(1,n/2)$. Assume that $\gamma_1,\gamma_2\in L^{\infty}(\R^n)$ with background deviations $m_1,m_2$ satisfy $\gamma_1(x),\gamma_2(x)\geq \gamma_0>0$ and $m_1,m_2\in H^{2s,\frac{n}{2s}}(\R^n)$. Moreover, suppose that one of the following assumptions hold:
    \begin{enumerate}[(i)]
        \item $\Omega\subset\R^n$ is bounded and $m\vcentcolon =m_1-m_2\in\Tilde{H}^{2s,\frac{n}{2s}}(\Omega)$,\label{Item:firstUniqFracCond1}
        \item or $\Omega\subset\R^n$ is bounded in one direction and $m\vcentcolon =m_1-m_2\in \Tilde{H}^s(\Omega)$.\label{Item:firstUniqFracCond2}
    \end{enumerate}
    If $W_1,W_2\subset\Omega_e$ are nonempty open sets with
    \[
        (\supp(m_1)\cup\supp(m_2))\cap (W_1\cup W_2)=\emptyset
    \]
and there holds $\left.\Lambda_{\gamma_1}f\right|_{W_2}=\left.\Lambda_{\gamma_2}f\right|_{W_2}$ for all $f\in C_c^{\infty}(W_1)$, then $\gamma_1=\gamma_2$ in $\R^n$.
\end{lemma}

\begin{proof}
    By Lemma~\ref{eq: Comparison of DN maps} we have $\left.\Lambda_{q_1}f\right|_{W_2}=\left.\Lambda_{q_2}f\right|_{W_2}$. Since $L^{\frac{n}{2s}}(\R^n)\hookrightarrow M_0(H^s\to H^{-s})\subset M_{\Omega}(H^s\to H^{-s})$ the assumptions of Theorem~\ref{thm:schrodingeruniqueness} in case \ref{Item:firstUniqFracCond1}, of Theorem~\ref{thm:schrodingeruniquenessUnbounded} in case \ref{Item:firstUniqFracCond2}, respectively, are fulfilled and we deduce in both cases $q_1|_{\Omega}=q_2|_{\Omega}$. If $\Omega\subset\R^n$ is bounded, then using the Sobolev embedding we deduce $m\in H^{2s,\frac{n}{2s}}(\R^n)\hookrightarrow L^{p}(\R^n)$ for all $\frac{n}{2s}\leq p\leq \infty$ and hence using $\Tilde{H}^{2s,\frac{n}{2s}}(\Omega)\hookrightarrow H_{\overline{\Omega}}^{2s\frac{n}{2s}}(\R^n)$ we get $m\in L^q(\R^n)$ for all $1\leq q\leq\infty$. Therefore, using \ref{item 1 interpolation Bessel potential spaces} of Corollary~\ref{cor: interpolation in Bessel potential spaces} with $s_1=0,s_2=2s,\theta=1/2,p_0=n/2s$ and $p_1=n/(n-2s)$ we infer $m\in H^s(\R^n)$. Since all used embeddings are continuous, we obtain $m\in\Tilde{H}^s(\Omega)$ and hence in both cases \ref{Item:firstUniqFracCond1} and \ref{Item:firstUniqFracCond2} we have $m\in\Tilde{H}^s(\Omega)$. Following \cite{covi2019inverse-frac-cond} we next calculate
    \[
        \begin{split}
            0&=\gamma_1^{1/2}\gamma_2^{1/2}(q_2-q_1)=-\gamma_1^{1/2}\gamma_2^{1/2}\left(\frac{(-\Delta)^sm_2}{\gamma_2^{1/2}}-\frac{(-\Delta)^sm_1}{\gamma_1^{1/2}}\right)\\
            &=-\gamma_1^{1/2}(-\Delta)^sm_2+\gamma_2^{1/2}(-\Delta)^sm_1=-(1+m_1)(-\Delta)^sm_2+(1+m_2)(-\Delta)^sm_1\\
            &=(1+m_1)(-\Delta)^sm-(1+m_1)(-\Delta)^sm_1+(1+m_2)(-\Delta)^sm_1\\
            &=(1+m_1)(-\Delta)^sm-m(-\Delta)^sm_1=\gamma_1^{1/2}(-\Delta)^sm+m(-\Delta)^sm_1
        \end{split}
    \]
    in $\Omega$ and thus $m\in\Tilde{H}^s(\Omega)$ solves
    \[
        \begin{split}
            (-\Delta)^sm-\frac{(-\Delta)^sm_1}{\gamma_1^{1/2}}m&=0\quad\text{in}\quad \Omega,\\
            m&=0\quad\text{in}\quad \Omega_e.
        \end{split}
    \]
    By Lemma~\ref{eq: well-posedness results and DN maps} we get $m\equiv 0$ which in turn shows $\gamma_1=\gamma_2$ in $\R^n$.
\end{proof}

\begin{remark}\label{rmk: fractional conductivity UCP remark}
    If one assumes 
    \begin{enumerate}[(I)]
        \item\label{item 1 remark on uniqueness cond eq}  $m\vcentcolon =m_1-m_2\in H^s(\R^n)$ and there exist an open set $V\subset \Omega$, a measurable subset $A\subset V$ with positive measure such that $m=0$ on $A$ and $(-\Delta)^sm_1\in L^{\infty}(V)$ and $\frac{1}{4} < s < 1$
        \item\label{item 2 remark on uniqueness cond eq} or $m\in H^t(\R^n)$ for $t\in \R$ and $m=0$ on some open nonempty subset of $\Omega$
    \end{enumerate} instead of the condition \ref{Item:firstUniqFracCond1} or \ref{Item:firstUniqFracCond2} in Lemma~\ref{thm: uniqueness conductivity equation}, then it directly follows from the (measurable) UCP (cf. \cite[Theorem 3, Remark 5.6]{GRSU-fractional-calderon-single-measurement}) and the given proof above that $m=0$ on all of $\R^n$ and hence $\gamma_1=\gamma_2$ on $\R^n$.
\end{remark}

Up to this point our arguments have been very similar to the ones given in \cite{covi2019inverse-frac-cond}. The main difference related to the solving the inverse problem has been that we do not restrict the potential $q$ of the Schrödinger equation to $\Omega$ as done in \cite{covi2019inverse-frac-cond} but allow it to take values everywhere and therefore the data $\Lambda_\gamma f|_{W_1\cap W_2}$ where $f \in C_c^\infty(W_1 \cap W_2)$ can be used to gain additional information whenever $W_1 \cap W_2 \neq \emptyset$. The main benefit of this becomes clear in the next statement. The other obvious differences have been that we allow $\Omega$ to be unbounded and non-Lipschitz (but bounded in one direction) and the conductivities to be nontrivial in $\R^n \setminus \Omega$. In fact, there is the following characterization of equal exterior data for the fractional conductivity equation. 
\begin{lemma}
    \label{thm: sharpness of condition (ii) in uniqueness thm}
    Let $\Omega\subset \R^n$ be an open set which is bounded in one direction and $0<s<\min(1,n/2)$. Assume that $\gamma_1,\gamma_2\in L^{\infty}(\R^n)$ with background deviations $m_1,m_2$ satisfy $\gamma_1(x)$, $\gamma_2(x)\geq \gamma_0>0$ and $m_1,m_2\in H^{2s,\frac{n}{2s}}(\R^n)$. Moreover, assume that $m_0\vcentcolon =m_1-m_2\in H^s(\R^n)$ and $W_1,W_2\subset\Omega_e$ are nonempty open sets with
    \[
        (\supp(m_1)\cup\supp(m_2))\cap (W_1\cup W_2)=\emptyset.
    \]
Then there holds $\left.\Lambda_{\gamma_1}f\right|_{W_2}=\left.\Lambda_{\gamma_2}f\right|_{W_2}$ for all $f\in C_c^{\infty}(W_1)$ if and only if
\begin{enumerate}[(i)]
    \item\label{first item thm 7.13 } $m_1-m_2$ is the unique solution of
    \begin{equation}
    \label{eq: sharpness}
        \begin{split}
            (-\Delta)^sm-\frac{(-\Delta)^sm_1}{\gamma_1^{1/2}}m&=0\quad\text{in}\quad \Omega,\\
            m&=m_0\quad\text{in}\quad \Omega_e.
        \end{split}
    \end{equation}
    \item\label{second item thm 7.13 } and $\int_{\Omega_e}(-\Delta)^sm_0fg\,dx=0$ for all $f\in C^{\infty}_c(W_1)$, $g\in C_c^{\infty}(W_2)$.
\end{enumerate}
\end{lemma}

\begin{proof}
    If the DN maps agree, then the proof of Lemma~\ref{thm: uniqueness conductivity equation} shows that $m_1-m_2$ is the unique solution of \eqref{eq: sharpness} and $q_1=q_2$ in $\Omega$. Thus, in both cases $m_1-m_2$ uniquely solves \eqref{eq: sharpness} and $q_1=q_2$ in $\Omega$ but then \ref{item 2 well-posedness cond eq} of Lemma~\ref{eq: well-posedness results and DN maps} and Lemma~\ref{eq: Comparison of DN maps} shows
    \[
    \begin{split}
        &\langle \Lambda_{\gamma_1}f,g\rangle =\langle \Lambda_{\gamma_2}f,g\rangle \Leftrightarrow \langle \Lambda_{q_1}f,g\rangle =\langle \Lambda_{q_2}f,g\rangle 
    \end{split}
    \]
    for all $f\in C_c^{\infty}(W_1)$, $g \in C_c^\infty(W_2)$. 
    By the Alessandrini identity for the fractional Schr\"odinger equation (cf.~\cite[Lemma 2.5]{GSU20}, \cite[Lemma 2.7]{RS-fractional-calderon-low-regularity-stability} and Lemma~\ref{alex}), whose proof remains the same in our setting, we obtain
    \begin{equation}
        \langle \Lambda_{\gamma_1}f,g\rangle =\langle \Lambda_{\gamma_2}f,g\rangle\Leftrightarrow \int_{\R^n}(q_1-q_2)u^{(1)}_fu^{(2)}_g\,dx=0,
    \end{equation}
    where $u^{(1)}_f,u^{(2)}_g\in H^s(\R^n)$ are the unique weak solutions of the fractional Schr\"odinger equation with potential $q_1,q_2$, respectively, and exterior values $f,g$. Therefore, we obtain
    \begin{equation}
    \label{eq: first cons for alessandrini}
    \begin{split}
        0=&\int_{\R^n}(q_1-q_2)u^{(1)}_fu^{(2)}_g\,dx\\
        =&\int_{\R^n}(q_1-q_2)(u^{(1)}_f-f)(u^{(2)}_g-g)\,dx+\int_{\R^n}(q_1-q_2)(u^{(1)}_f-f)g\,dx\\
        &+\int_{\R^n}(q_1-q_2)f(u^{(2)}_g-g)\,dx+\int_{\R^n}(q_1-q_2)fg\,dx.
    \end{split}
    \end{equation}
    Since $u^{(1)}_f-f,u_g^{(2)}-g\in\Tilde{H}^s(\Omega)$ there exists $u_k^{(1)},u_k^{(2)}\in C_c^{\infty}(\Omega)$ such that $u_k^{(1)}\to u^{(1)}_f-f$ and $u_k^{(2)}\to u^{(2)}_g-g$ in $H^s(\R^n)$. By Lemma~\ref{lemma: continuity of multiplication map in Bessel potential spaces I V}, the Sobolev embedding and H\"older's inequality we have
    \begin{equation}
    \label{eq: second cons for alessandrini}
    \begin{split}
        0=&\lim_{k\to\infty}\int_{\R^n}(q_1-q_2)u^{(1)}_ku^{(2)}_k\,dx+\lim_{k\to\infty}\int_{\R^n}(q_1-q_2)u^{(1)}_kg\,dx\\
        &+\lim_{k\to\infty}\int_{\R^n}(q_1-q_2)fu^{(2)}_k\,dx+\int_{\R^n}(q_1-q_2)fg\,dx.
    \end{split}
    \end{equation}
    Now the first term is zero since $q_1=q_2$ in $\Omega$, the second and third term are zero as $u_k^{(j)}$ ($j=1,2$) and $f,g$ have disjoint supports. Therefore, since $f,g$ are supported in $\Omega_e$ we deduce
    \[
        \int_{\Omega_e}(q_1-q_2)fg\,dx=0
    \]
    for all $f\in C_c^{\infty}(W_1),g\in C_c^{\infty}(W_2)$. As the assumption on the supports of $m_1,m_2$ imply $\gamma_1=\gamma_2=1$ on $W_1\cup W_2$, the last identity is equivalent to
    \[
        \int_{\Omega_e}(-\Delta)^sm_0fg\,dx=0
    \]
    for all $f\in C_c^{\infty}(W_1),g\in C_c^{\infty}(W_2)$. Hence, this implies that if the DN maps coincide, then the assertions \ref{first item thm 7.13 }, \ref{second item thm 7.13 } of Lemma~\ref{thm: sharpness of condition (ii) in uniqueness thm} hold.
    
    Conversely, if $m_1-m_2$ solves the equation \eqref{eq: sharpness}, then the computation in the proof of Lemma~\ref{thm: uniqueness conductivity equation} shows that $q_1=q_2$ in $\Omega$. Moreover, by using \ref{second item thm 7.13 } we deduce 
    from the previous computation (cf.~\eqref{eq: first cons for alessandrini}, \eqref{eq: second cons for alessandrini}) that 
    \[
    \int_{\R^n}(q_1-q_2)u^{(1)}_fu^{(2)}_g\,dx=\int_{\Omega_e}(q_1-q_2)fg\,dx=0,
    \]
    where $u^{(1)}_f,u^{(2)}_g\in H^s(\R^n)$ are the unique weak solutions of the fractional Schr\"odinger equation with potential $q_1,q_2$, respectively, and exterior values $f,g$.
    Now the Alessandrini identity implies that the DN maps for the Schr\"odinger equation and hence for the conductivity equation coincide.
\end{proof}

We can now prove the main theorem of our work using the characterization of Lemma \ref{thm: sharpness of condition (ii) in uniqueness thm} and the UCP.

\begin{proof}[Proof of Theorem~\ref{thm: characterization of uniqueness}]
    Assume $W_1\cap W_2\neq\emptyset$, then there is an open bounded set $U\subset\R^n$ such that $\overline{U}\subset W_1\cap W_2$. By Theorem~\ref{thm: sharpness of condition (ii) in uniqueness thm} we know that the DN maps coincide if and only if 
    \begin{enumerate}[(I)]
    \item\label{item 1 proof main thm cond eq} $m_1-m_2$ is the unique solution of
    \begin{equation}
        \begin{split}
            (-\Delta)^sm-\frac{(-\Delta)^sm_1}{\gamma_1^{1/2}}m&=0\quad\text{in}\quad \Omega,\\
            m&=m_0\quad\text{in}\quad \Omega_e.
        \end{split}
    \end{equation}
    \item\label{item 2 proof main thm cond eq} and $\int_{\Omega_e}(-\Delta)^sm_0fg\,dx=0$ for all $f\in C^{\infty}_c(W_1)$, $g\in C_c^{\infty}(W_2)$.
\end{enumerate}
    By the property \ref{item 2 proof main thm cond eq} and choosing a cut--off function $f\in C_c^{\infty}(W_1)$ such that $f|_{\overline{U}}=1$ we see that
    \[
        \int_{U}(-\Delta)^sm_0g\,dx=0
    \]
    for all $g\in C_c^{\infty}(U)$. This in turn implies $(-\Delta)^sm=0$ in $U$. On the other hand we have by the support assumption on $m_1$ and $m_2$ that $m_0=0$ in $U$ and hence the UCP (cf.~\cite[Theorem 1.2]{CMR20}) implies $m_0\equiv 0$. Therefore we have $\gamma_1=\gamma_2$. Finally, observe that the assertion \ref{item 2 characterization of uniqueness} is a direct consequence of Theorem~\ref{thm: sharpness of condition (ii) in uniqueness thm}, since the test functions $f\in C_c^{\infty}(W_1)$ and $g\in C_c^{\infty}(W_2)$ have disjoint supports.
\end{proof}

\begin{remark} The crucial difference between general fractional inverse problems and the one arising from the conductivity equation is that the potentials associated with the conductivity equation arise in the form $(-\Delta)^s m/\gamma^{1/2}$ for some functions $m$ and $\gamma>0$. After interior determination is made with the usual argument one is left with the additional exterior data which must also vanish. Therefore, this let us to conclude stronger results than expected using the UCP of fractional Laplacians when having overlapping exterior DN partial data. This is an interesting feature and it may indeed generalize to other nonlocal inverse problems having similar special structures at hand. For instance, whenever potentials $q$ of the equation $(-\Delta)^s + q$ are a priori known to arise from a suitable restriction of the range of some fixed nonlocal operator having the UCP. For example, in Theorem~\ref{thm: characterization of uniqueness} this operator is $(-\Delta)^s$ and one may restrict the domain to be the functions in $H^s(\R^n)$ that agree in $W_1 \cap W_2$.
\end{remark}

\appendix

\section{Multiplication map in Bessel potential spaces}
\label{subsubsec: Multiplication map in Bessel potential spaces}

In this appendix, we adapt the methods in \cite{BrezisComposition} to conclude useful multiplication results for Bessel potential spaces. See also \cite[Lemma 2.2]{RS-fractional-calderon-low-regularity-stability} for related results in terms of Sobolev multipliers. Before stating the main results of this section we recall a possible definition of the Triebel--Lizorkin spaces $F^s_{p,q}=F^s_{p,q}(\R^n)$ following the exposition in \cite{BrezisComposition} or \cite{Triebel-Theory-of-function-spaces,Triebel-Theory-of-functions-spaces-I-I}. In the sequel we will write for brevity $B_r$ instead of $B_r(0)$. Fix any $\psi_0\in C_c^{\infty}(\R^n)$ satisfying
\[
    0\leq \psi_0\leq 1,\quad \psi_0(\xi)=1\quad\text{for}\quad |\xi|\leq 1\quad \text{and}\quad \psi_0(\xi)=0\quad \text{for}\quad |\xi|\geq 2
\]
and let $\psi_j\in C_c^{\infty}(B_{2^{j+1}})$, $j\geq 1$, be given by
\[
    \psi_j(\xi)=\psi_0(\xi/2^j)-\psi_0(\xi/2^{j-1}).
\]
In this section we write $u_j\vcentcolon =u\ast \phi_j=\fourier(\psi_j\hat{u})\in \schwartz(\R^n)$ for any $u\in\tempered(\R^n)$, where $\phi_j=\ifourier(\psi_j)$ ($j\geq 1$) and one has the Littlewood--Paley decomposition
\[
    u=\sum_{j\geq 0}u_j\quad \text{in}\quad \tempered(\R^n).
\]
For $s\in\R$, $0<p,q\leq \infty$ we set (\cite[Section 2.3.1]{Triebel-Theory-of-function-spaces})
\[
    F^s_{p,q}=\{u\in\tempered(\R^n);\quad \|u\|_{F^s_{p,q}}=\|\,\|2^{js}u_j(x)\|_{\ell^q}\|_{L^p(\R^n)}<\infty\}.
\]
\begin{remark}
\label{remark: identification}
    \begin{enumerate}[(i)]
        \item\label{item 1 triebel spaces} For $0<p<\infty$ different choices of $\psi_0$ yield equivalent quasi-norms (see \cite[Section 2.3.5]{Triebel-Theory-of-function-spaces}), but for $p=\infty$, $0<q<\infty$ this is in general wrong as shown in \cite[Section 2.3.2]{Triebel-Theory-of-functions-spaces-I-I}, and for $s\in\R$, $0<p<\infty$, $0<q\leq \infty$ the Triebel-Lizorkin spaces are quasi-Banach spaces and Banach spaces if $p,q\geq 1$ (see \cite[Section 2.3.3]{Triebel-Theory-of-function-spaces}).
        \item\label{item 2 triebel spaces} By the embedding $\ell^{q_1}\hookrightarrow \ell^{q_2}$ for $0<q_1\leq q_2\leq \infty$ we have $F^s_{p,q_1}\subset F^s_{p,q_2}$ when $s\in\R$, $0<p\leq \infty$ and $0<q_1\leq q_2\leq \infty$.
        \item\label{item 3 triebel spaces} We have the following identifications with equivalent norms (\cite{BrezisComposition,Triebel-Theory-of-function-spaces}):
        \begin{enumerate}[(I)]
            \item\label{first identification remark A1} $F^0_{p,2}=L^p(\R^n)$ for $1<p<\infty$,
            \item\label{second identification remark A1} $F^s_{p,2}=H^{s,p}(\R^n)$ for $1<p<\infty$, $s\in\R$,
            \item\label{third identification remark A1} $F^s_{p,p}=W^{s,p}(\R^n)$ for $1\leq p<\infty$, $s\in (0,\infty)\setminus\N$,
            \item\label{fourth identification remark A1} $L^{\infty}(\R^n)\hookrightarrow F^{0}_{\infty,\infty}$ with
            \[
            \|u\|_{F^0_{\infty,\infty}}=\sup_{j\in\N_0,x\in\R^n}|u_j(x)|\leq C\|u\|_{L^{\infty}(\R^n)}.
            \]
        \end{enumerate}
    \end{enumerate}
\end{remark}

The following general interpolation result in Triebel-Lizorkin spaces hold:
\begin{theorem}[{\cite[Lemma 3]{BrezisComposition}}]
\label{thm: interpolation in Triebel spaces}
    Let $-\infty<s_1<s_2<\infty$, $0<p_1,p_2,q_1,q_2\leq \infty$, $0<\theta<1$ and set
    \[
            s=\theta s_1+(1-\theta)s_2\quad\text{and}\quad \frac{1}{p}=\frac{\theta}{p_1}+\frac{1-\theta}{p_2},
    \]
    then for all $0<q\leq \infty$ there holds
   \[
        \|u\|_{F^{s}_{p,q}}\leq C\|u\|_{F^{s_1}_{p_1,q_1}}^{\theta}\|u\|_{F^{s_2}_{p_2,q_2}}^{1-\theta},
   \]
    when $u\in F^{s_1}_{p_1,q_1}\cap F^{s_2}_{p_2,q_2}$.
\end{theorem}

Next we prove the analogous result of \cite[Corollary 2]{BrezisComposition} for Bessel potential spaces. The first estimate can also be obtained more directly by complex interpolation (cf.~\cite[Theorem 6.4.5]{Interpolation-spaces} or \cite{TRI-interpolation-function-spaces}).
\begin{corollary}
\label{cor: interpolation in Bessel potential spaces}
    Let $0<\theta<1$.
    \begin{enumerate}[(i)]
        \item\label{item 1 interpolation Bessel potential spaces} If $-\infty<s_1<s_2<\infty$, $1<p_1,p_2<\infty$ satisfy
        \[
            s=\theta s_1+(1-\theta)s_2\quad\text{and}\quad \frac{1}{p}=\frac{\theta}{p_1}+\frac{1-\theta}{p_2},
        \]
        then
        \[
            \|u\|_{H^{s,p}(\R^n)}\leq C\|u\|_{H^{s_1,p_1}(\R^n)}^{\theta}\|u\|_{H^{s_2,p_2}(\R^n)}^{1-\theta}
        \]
        for all $u\in H^{s_1,p_1}(\R^n)\cap H^{s_2,p_2}(\R^n)$.
        \item\label{item 2 interpolation Bessel potential spaces} If $0<s<\infty$, $1<p<\infty$, $0<q\leq \infty$, then 
        \[
            \|u\|_{F^{\theta}_{p/\theta,q}}\leq C\|u\|_{H^{s,p}(\R^n)}^{\theta}\|u\|_{L^{\infty}(\R^n)}^{1-\theta}
        \]
        for all $u\in H^{s,p}(\R^n)\cap L^{\infty}(\R^n)$.
        \item\label{item 3 interpolation Bessel potential spaces} If $0<s<\infty$, $1<p<\infty$, then 
        \[
            \|u\|_{H^{\theta s,p/\theta}}\leq C\|u\|_{H^{s,p}(\R^n)}^{\theta}\|u\|_{L^{\infty}(\R^n)}^{1-\theta}
        \]
        for all $u\in H^{s,p}(\R^n)\cap L^{\infty}(\R^n)$.
    \end{enumerate}
\end{corollary}

\begin{proof}
    The proof is completely analogous to the one given in \cite[Corollary 2]{BrezisComposition}, but for convenience of the reader we reproduce it here.\\
    \ref{item 1 interpolation Bessel potential spaces} Using the identification $F^{s}_{p,2}=H^{s,p}(\R^n)$ for $s\in\R$, $1<p<\infty$ and setting $q=q_1=q_2=2$ in Theorem~\ref{thm: interpolation in Triebel spaces} already gives the result.\\
    \ref{item 2 interpolation Bessel potential spaces} If we redefine $\theta\to 1-\theta$, apply Theorem~\ref{thm: interpolation in Triebel spaces} with $s_1=0,s_2>0$, $p_1=q_1=\infty$, $1<p_2<\infty$, $q_2=2$ and use \ref{second identification remark A1}, \ref{fourth identification remark A1} of Remark~\ref{remark: identification}, we obtain the result up to relabelling the indices.\\
    \ref{item 3 interpolation Bessel potential spaces} This follows from \ref{item 2 interpolation Bessel potential spaces} by choosing $q=2$.
\end{proof}

Now we state the Runst-Sickel lemmas which establish continuity results for the multiplication map in Triebel-Lizorkin spaces and eventually allow us to show that the multiplication map between certain (local) Bessel potential spaces are continuous. In the following we denote for any $f\in L^1_{loc}(\R^n)$ by $Mf$ the Hardy-Littlewood maximal function, that is, 
\[
    Mf(x)\vcentcolon =\sup_{r>0}\frac{1}{|B_r(x)|}\int_{B_r(x)}|f(y)|\,dy.
\]

\begin{proposition}[{Runst-Sickel lemma I, \cite[Lemma 5]{BrezisComposition}}]
\label{prop: Runs-Sickel Lemma I}
    Let $0<s<\infty$, $1<q<\infty$, $1<p_1,p_2,r_1,r_2\leq \infty$ satisfy
    \[
        0<\frac{1}{p}\vcentcolon =\frac{1}{p_1}+\frac{1}{r_2}=\frac{1}{p_2}+\frac{1}{r_1}<1,
    \]
then for all $f\in F^s_{p_1,q}\cap L^{r_1}(\R^n)$, $g\in F^s_{p_2,q}\cap L^{r_2}(\R^n)$ there holds
\[
    \|fg\|_{F^{s}_{p,q}}\leq C(\|Mf(x)\|2^{sj}g_j(x)\|_{\ell^q}\|_{L^p(\R^n)}+\|Mg(x)\|2^{sj}f_j(x)\|_{\ell^q}\|_{L^p(\R^n)})
\]
and
\[
    \|fg\|_{F^{s}_{p,q}}\leq C(\|f\|_{F^{s}_{p_1,q}}\|g\|_{L^{r_2}(\R^n)}+\|g\|_{F^{s}_{p_2,q}}\|f\|_{L^{r_1}(\R^n)}).
\]
\end{proposition}

\begin{proposition}[{Runst-Sickel lemma II, \cite[Corollary 3]{BrezisComposition}}]
\label{prop: Runs-Sickel Lemma II}
    Let $0<s<\infty$, $1<q<\infty$. Then the following assertions hold:
    \begin{enumerate}[(i)]
        \item\label{item 1 runst sickel II } If $1<p_1,p_2,r_1,r_2< \infty$ satisfy
    \[
        0<\frac{1}{p}\vcentcolon =\frac{1}{p_1}+\frac{1}{r_2}=\frac{1}{p_2}+\frac{1}{r_1}<1,
    \] then the multiplication map 
        \[
            (F^{s}_{p_1,q}\cap L^{r_1}(\R^n))\times (F^{s}_{p_2,q}\cap L^{r_2}(\R^n))\ni (f,g)\mapsto fg\in F^s_{p,q}
        \]
        is continuous.
        \item\label{item 2 runst sickel II } If $1<p<\infty$ and there holds
        \[
            \begin{cases}
                    f^k\to f\quad \text{in}\quad F^s_{p_1,q},\, & \|f^k\|_{L^{\infty}(\R^n)}\leq C\\
                    g^k\to g\quad \text{in}\quad F^s_{p_1,q},\, & \|g^k\|_{L^{\infty}(\R^n)}\leq C
            \end{cases}
        \]
        for some $C>0$, then $f^kg^k\to fg$ in $F^{s}_{p,q}$.
        \item\label{item 3 runst sickel II } Let $1<p_1,r,p<\infty$ be such that
        \[
            \frac{1}{p}=\frac{1}{p_1}+\frac{1}{r}
        \]
        and there holds
        \[
            \begin{cases}
                    f^k\to f\quad \text{in}\quad F^s_{p_1,q},\, & \|f^k\|_{L^{\infty}(\R^n)}\leq C\\
                    g^k\to g\quad \text{in}\quad F^s_{p,q},\, &
            \end{cases}
        \]
        for some $C>0$, then $f^kg^k\to fg$ in $F^{s}_{p,q}$.
    \end{enumerate}
\end{proposition}

\begin{lemma}
\label{lemma: continuity of multiplication map in Bessel potential spaces I}
    Let $0<s_1<\infty$, $1<p_1,p_2,r_2<\infty$, $0<\theta<1$ satisfy
    \[
        \frac{1}{p_2}=\frac{\theta}{p_1}+\frac{1}{r_2},
    \]
    then 
    \[
        (H^{s_1,p_1}(\R^n)\cap L^{\infty}(\R^n))\times (H^{\theta s_1,p_2}(\R^n)\cap L^{r_2}(\R^n))\ni (f,g)\mapsto fg\in H^{\theta s_1,p_2}(\R^n)
    \]
    is well-defined and there holds:
    \begin{enumerate}[(i)]
        \item\label{item 1 continuity bessel potential product} For all $f\in H^{s_1,p_1}(\R^n)\cap L^{\infty}(\R^n)$, $g\in H^{\theta s_1,p_2}(\R^n)\cap L^{r_2}(\R^n)$ we have
        \[
            \|fg\|_{H^{\theta s_1,p_2}(\R^n)}\leq C(\|f\|_{L^{\infty}(\R^n)}\|g\|_{H^{\theta s_1,p_2}(\R^n)}+\|g\|_{L^{r_2}(\R^n)}\|f\|^{\theta}_{H^{s_1,p_1}(\R^n)}\|f\|_{L^{\infty}(\R^n)}^{1-\theta}).
        \]
        \item\label{item 2 continuity bessel potential product} If $f^k\to f$ in $H^{s_1,p_1}(\R^n)$, $\|f^k\|_{L^{\infty}(\R^n)}\leq C$ for some $C>0$ and $g^k\to g$ in $H^{\theta s_1,p_2}(\R^n)\cap L^{r_2}(\R^n)$, then $f^kg^k\to fg$ in $H^{\theta s_1,p_2}(\R^n)$.
    \end{enumerate}
\end{lemma}

\begin{proof}
    \ref{item 1 continuity bessel potential product} By the statement \ref{item 3 interpolation Bessel potential spaces} of Corollary~\ref{cor: interpolation in Bessel potential spaces} we have
    \[
        \|f\|_{H^{\theta s_1,p_1/\theta}(\R^n)}\leq C\|f\|_{H^{s_1,p_1}(\R^n)}^{\theta}\|f\|_{L^{\infty}(\R^n)}^{1-\theta}.
    \]
    Hence the Runst-Sickel lemma I (Proposition~\ref{prop: Runs-Sickel Lemma I}) with $q=2$ and Remark~\ref{remark: identification}, \ref{fourth identification remark A1} shows
    \[
    \begin{split}
         \|fg\|_{H^{\theta s_1,p}(\R^n)}&\leq C(\|f\|_{H^{\theta s_1,p_1/\theta}(\R^n)}\|g\|_{L^{r_2}(\R^n)}+\|f\|_{L^{\infty}(\R^n)}\|g\|_{H^{\theta s_1,p_2}(\R^n)})\\
         &\leq C(\|f\|_{H^{s_1,p_1}(\R^n)}^{\theta}\|f\|_{L^{\infty}(\R^n)}^{1-\theta}\|g\|_{L^{r_2}(\R^n)}+\|f\|_{L^{\infty}(\R^n)}\|g\|_{H^{\theta s_1,p_2}(\R^n)}).
    \end{split}
    \]
\ref{item 2 continuity bessel potential product} First we show that $f\in H^{\theta s_1,p_1/\theta}(\R^n)$. Since, $(f^k)_{k\in\N}$ is bounded in $L^{\infty}(\R^n)$ by the Banach--Alaoglu theorem there is a subsequence, which we still denote by $(f^k)$, which converges weak-$\ast$ to some $\Tilde{f}\in L^{\infty}(\R^n)$ and this will be denoted from now on as $f^k\overset{\ast}{\rightharpoonup} \Tilde{f}$ in $L^{\infty}(\R^n)$. Using the embedding $H^{s_1,p_1}(\R^n)\hookrightarrow L^{p_1}(\R^n)$ we deduce $f=\Tilde{f}$ with 
\[
    \|f\|_{L^{\infty}(\R^n)}\leq \liminf_{k\to\infty}\|f^k\|_{L^{\infty}(\R^n)}\leq C.
\]  
But now the assertion \ref{item 3 interpolation Bessel potential spaces} of Corollary~\ref{cor: interpolation in Bessel potential spaces} implies $f^k\to f$ in $H^{\theta s_1,p_1/\theta}(\R^n)$. By applying part \ref{item 3 runst sickel II } of the Runst-Sickel lemma II (Proposition~\ref{prop: Runs-Sickel Lemma II}) we obtain $f^kg^k\to fg$ in $H^{\theta s_1,p_2}(\R^n)$. Since, we can use this argument for any subsequence of $f^k$ we deduce that the whole sequence $f^kg^k$ converges to $fg$ in $H^{\theta s_1,p_2}(\R^n)$.
\end{proof}

As a special case we obtain the following result
\begin{corollary}
\label{cor: continuity of multiplication map in Bessel potential spaces I I}
    Let $0<s<n/2$, then the multiplication map 
    \[
        (H^{2s,n/2s}(\R^n)\cap L^{\infty}(\R^n))\times H^s(\R^n)\ni (f,g)\mapsto fg\in H^s(\R^n)
    \]
    is well defined and there holds:
    \begin{enumerate}[(i)]
        \item\label{item 1 Hs continuity} For all $f\in H^{2s,n/2s}(\R^n)\cap L^{\infty}(\R^n)$, $g\in H^{s}(\R^n)$ we have
        \[
            \|fg\|_{H^{s}(\R^n)}\leq C(\|f\|_{L^{\infty}(\R^n)}+\|f\|^{\theta}_{H^{2s,n/2s}(\R^n)}\|f\|_{L^{\infty}(\R^n)}^{1-\theta})\|g\|_{H^s(\R^n)}.
        \]
        \item\label{item 2 Hs continuity} If $f^k\to f$ in $H^{2s,n/2s}(\R^n)$, $\|f^k\|_{L^{\infty}(\R^n)}\leq C$ for some $C>0$ and $g^k\to g$ in $H^{s}(\R^n)$, then $f^kg^k\to fg$ in $H^{s}(\R^n)$.
    \end{enumerate}
\end{corollary}

\begin{proof}
Use Lemma~\ref{lemma: continuity of multiplication map in Bessel potential spaces I} with $s_1=2s, p_1=\frac{n}{2s}, p_2=2, r_2=\frac{2n}{n-2s}, \theta=1/2$ and the standard Sobolev embedding $H^s(\R^n)\hookrightarrow L^{\frac{2n}{n-2s}}(\R^n)$.
\end{proof}

We also have the following local version of Corollary~\ref{cor: continuity of multiplication map in Bessel potential spaces I I}:
\begin{corollary}
\label{cor: continuity of multiplication map in Bessel potential spaces I I I}
    Let $\Omega\subset \R^n$ be an open set and $0<s<n/2$, then the multiplication map 
    \[
        (H^{2s,n/2s}(\R^n)\cap L^{\infty}(\R^n))\times \tilde{H}^s(\Omega)\ni (f,g)\mapsto fg\in \tilde{H}^s(\Omega)
    \]
    is well-defined and there holds:
    \begin{enumerate}[(i)]
        \item\label{item 1 continuity bessel space cor} For all $f\in H^{2s,n/2s}(\R^n)\cap L^{\infty}(\R^n)$, $g\in \tilde{H}^{s}(\Omega)$ we have
        \[
            \|fg\|_{H^{s}(\R^n)}\leq C(\|f\|_{L^{\infty}(\R^n)}+\|f\|^{\theta}_{H^{2s,n/2s}(\R^n)}\|f\|_{L^{\infty}(\R^n)}^{1-\theta})\|g\|_{H^s(\R^n)}.
        \]
        \item\label{item 2 continuity bessel space cor} If $f^k\to f$ in $H^{2s,n/2s}(\R^n)$, $\|f^k\|_{L^{\infty}(\R^n)}\leq C$ for some $C>0$ and $g^k\to g$ in $\tilde{H}^{s}(\Omega)$, then $f^kg^k\to fg$ in $H^{s}(\R^n)$.
    \end{enumerate}
\end{corollary}

\begin{proof}
We only need to prove that the multiplication map is well-defined. From Corollary~\ref{cor: continuity of multiplication map in Bessel potential spaces I I} we already know that $fg\in H^s(\R^n)$. Denote by $(\rho_{\epsilon})_{\epsilon>0}\subset C_c^{\infty}(\R^n)$ the standard mollifiers and set $f^k\vcentcolon =f\ast \rho_{\epsilon_k}\in C^{\infty}(\R^n)$ for some $\epsilon_k\to 0$ as $k\to\infty$. Using standard results one deduces that $f^k\to f$ in $H^{2s,\frac{n}{2s}}(\R^n)$ and $\|f^k\|_{L^{\infty}(\R^n)}\leq \|f\|_{L^{\infty}(\R^n)}$. By assumption there is a sequence $(g^k)_{k\in\N}\subset C_c^{\infty}(\Omega)$ such that $g^k\to g$ in $H^s(\R^n)$ as $k\to\infty$. Then $f^kg^k\in C_c^{\infty}(\Omega)$ and by the statement \ref{item 2 Hs continuity} of Corollary~\ref{cor: continuity of multiplication map in Bessel potential spaces I I} we obtain $f^kg^k\to fg$ in $H^s(\R^n)$. Therefore we can conclude the proof.
\end{proof}

Another important ingredient in the uniqueness proof of the inverse problem associated to the fractional conductivity equation will be deduced from the following continuity result:

\begin{proposition}[{\cite[Theorem 6.1]{MultiplicationSobolev}}]
\label{proposition: multiplication map in fractional Sobolev spaces}
    Let $0\leq s_0\leq s_1,s_2<\infty$, $1\leq p_0,p_1,p_2<\infty$ satisfy $s_0\in \N_0$ and
    \begin{enumerate}[(i)]
        \item\label{assump 1 multiplication} $1/p_0-s_0/n\geq 1/p_i-s_i/n$ for $i=1,2$ and $s_1+s_2-s_0>n(1/p_1+1/p_2-1/p_0)\geq 0$
        \item\label{assump 2 multiplication} or $1/p_0-s_0/n> 1/p_i-s_i/n$ for $i=1,2$ and $s_1+s_2-s_0\geq n(1/p_1+1/p_2-1/p_0)\geq 0$,
    \end{enumerate}
    then the multiplication map
    \[
        W^{s_1,p_1}(\R^n)\times W^{s_2,p_2}(\R^n)\ni (f,g)\mapsto fg\in W^{s_0,p_0}(\R^n)
    \]
    is a continuous bilinear map and in particular there holds
    \[
        \|fg\|_{W^{s_0,p_0}(\R^n)}\leq C\|f\|_{W^{s_1,p_1}(\R^n)}\|g\|_{W^{s_2,p_2}(\R^n)}
    \]
for all $f\in W^{s_1,p_1}(\R^n)$, $g\in W^{s_2,p_2}(\R^n)$.
\end{proposition}

We have the following special case:

\begin{lemma}
\label{lemma: continuity of multiplication map in Bessel potential spaces I V}
    Let $n\in\N$ and $0<s<n/2$, then the map 
    \[
        L^{\frac{n}{2s}}(\R^n)\times H^s(\R^n)\ni (f,g)\mapsto fg\in L^{\frac{2n}{n+2s}}(\R^n)
    \]
    is a continuous bilinear map with
    \[
        \|fg\|_{L^{\frac{2n}{n+2s}}}\leq C\|f\|_{L^{\frac{n}{2s}}(\R^n)}\|g\|_{H^s(\R^n)}
    \] 
    for all $f\in L^{\frac{n}{2s}}(\R^n)$, $g\in H^s(\R^n)$.
\end{lemma}
\begin{proof}
    Let $s_0=s_1=0$, $s_2=s>0$, $p_0=\frac{2n}{n+2s}$, $p_1=\frac{n}{2s}$, $p_2=2$. Then 
    \[
        \frac{1}{p_0}-\frac{s_0}{n}=\frac{n+2s}{2n},\quad \frac{1}{p_1}-\frac{s_1}{n}=\frac{2s}{n},\quad \frac{1}{p_2}-\frac{s_2}{n}=\frac{n-2s}{2n}
    \]
    and 
    \[
        s_1+s_2-s_0=s,\quad \frac{1}{p_1}+\frac{1}{p_2}-\frac{1}{p_0}=\frac{s}{n}.
    \]
    Therefore assumption \ref{assump 2 multiplication} of Proposition~\ref{proposition: multiplication map in fractional Sobolev spaces} is fulfilled and we can conclude the proof.
\end{proof}

\begin{corollary}
\label{corollary: continuity of multiplication map in Bessel potential spaces V}
    Let $n\in\N$ and $0<s<n/2$, then the multiplication map 
    \[
        L^{\frac{n}{2s}}(\R^n)\times H^s(\R^n)\ni (f,g)\mapsto fg\in H^{-s}(\R^n)
    \]
    is continuous.
\end{corollary}

\begin{proof}
    The statement directly follows by combining Lemma~\ref{lemma: continuity of multiplication map in Bessel potential spaces I V} with the Sobolev embedding and noting that $\frac{2n}{n+2s}$ is the conjugate exponent to $\frac{2n}{n-2s}$. 
\end{proof}

\bibliography{refs} 

\bibliographystyle{alpha}

\end{document}